\documentclass[preprint,12pt]{elsarticle}

\usepackage{amssymb}
\usepackage{amsmath}
\usepackage{amsthm}
\usepackage{xcolor}
\usepackage{empheq}
\usepackage{esvect}
\usepackage{marginnote}
\usepackage{graphicx}
\usepackage{subcaption}
\usepackage{algorithm}
\usepackage{algpseudocode}
\usepackage{hyperref}
\makeatletter
\def\ps@pprintTitle{%
  \let\@oddhead\@empty
  \let\@evenhead\@empty
  \let\@oddfoot\@empty
  \let\@evenfoot\@empty
}
\makeatother
\makeatother

\graphicspath{{fig/}}

\long\def\tmpcomment#1\endtmpcomment{#1}

\DeclareMathOperator{\sign}{sgn}
\newtheorem{theorem}{Theorem}
\newtheorem{proposition}{Proposition}
\newtheorem{lemma}{Lemma}

\theoremstyle{definition}
\newtheorem{definition}{Definition}

\theoremstyle{remark}
\newtheorem{remark}{Remark}

\newtheoremstyle{exampleStyle}
  {\topsep}   % Space above
  {\topsep}   % Space below
  {\normalfont}  % Body font (upright)
  {}          % Indent amount
  {\bfseries} % Head font
  {.}         % Punctuation after head
  {1em}       % Space after head
  {}          % Head spec

\theoremstyle{exampleStyle}
\newtheorem{example}{Example}[section]

\journal{Computer-Aided Design, Elsevier}

\newcommand{\bluenew}[1]{#1}
\newcommand{\lastfix}[1]{#1}

\begin{document}

\begin{frontmatter}

%% Title, authors and addresses

%% use the tnoteref command within \title for footnotes;
%% use the tnotetext command for theassociated footnote;
%% use the fnref command within \author or \affiliation for footnotes;
%% use the fntext command for theassociated footnote;
%% use the corref command within \author for corresponding author footnotes;
%% use the cortext command for theassociated footnote;
%% use the ead command for the email address,
%% and the form \ead[url] for the home page:
%% \title{Title\tnoteref{label1}}
%% \tnotetext[label1]{}
%% \author{Name\corref{cor1}\fnref{label2}}
%% \ead{email address}
%% \ead[url]{home page}
%% \fntext[label2]{}
%% \cortext[cor1]{}
%% \affiliation{organization={},
%%             addressline={},
%%             city={},
%%             postcode={},
%%             state={},
%%             country={}}
%% \fntext[label3]{}

\title{Quasi-Symmetric Nets: A Constructive Approach to the Equimodular Elliptic Type of Kokotsakis Polyhedra\\[-0.35cm]}

%\title{On the Explicit Construction of Flexible Kokotsakis Polyhedra of the Equimodular Elliptic Type}

%% use optional labels to link authors explicitly to addresses:
%% \author[label1,label2]{}
%% \affiliation[label1]{organization={},
%%             addressline={},
%%             city={},
%%             postcode={},
%%             state={},
%%             country={}}
%%
%% \affiliation[label2]{organization={},
%%             addressline={},
%%             city={},
%%             postcode={},
%%             state={},
%%             country={}}

% \author{Anonymous CAD submission}  %% Correcposing Author name
%\ead{ }

%\cortext[cor1]{ }

%% Author affiliation
%\affiliation%[label1]
%{organization={}%,%Department and Organization
%           addressline={}%,
%           city={}%,
%           postcode={}%,
%           state={}%,
%          country={}%
%          }

\author[label1]{A. Nurmatov}  %% Correcposing Author name
\ead{abdukhomid.nurmatov@kaust.edu.sa}

\author[label1]{M. Skopenkov}  %% Author name
\ead{mikhail.skopenkov@gmail.com}

\author[label1]{F. Rist}  %% Author name
\ead{florian.rist@kaust.edu.sa}

\author[label1]{J. Klein}  %% Author name
\ead{jonathan.klein@kaust.edu.sa}

\author[label1]{D. L. Michels%\corref{cor1}
}  %% Author name
\ead{dominik.michels@kaust.edu.sa}

%\cortext[cor1]{Corresponding author}

%% Author affiliation
\affiliation[label1]{organization={KAUST Computational Sciences Group},%Department and Organization
            addressline={Campus Building 1}, 
            city={Thuwal},
            postcode={23955-6900}, 
            % state={},
            country={KSA\\[-1.25cm]}}

%% Abstract

\begin{abstract}
\bluenew{A Kokotsakis polyhedron is a polyhedral mesh in three-dimensional
Euclidean space formed by a central $n$-gonal face (the base), $n$ quadrilateral faces each sharing one edge with the base, and $n$ triangular faces inserted between every two adjacent quadrilaterals; it is called flexible if it admits a continuous deformation that preserves the rigidity of every face.} This work investigates flexible Kokotsakis polyhedra with a quadrangular base \bluenew{($n=4$)} of \emph{equimodular elliptic} type, filling a significant gap in the literature by providing the first \emph{explicit constructions} of this type together with an \emph{explicit algebraic characterization} in terms of flat and dihedral angles. A straightforwardly constructible class of polyhedra -- called \emph{quasi-symmetric nets} (\emph{QS-nets}) -- is introduced, characterized by a symmetry relation among flat angles. It is shown that every elliptic QS-net has equimodular elliptic type and is flexible in real three-dimensional Euclidean space (rather than only in complex configuration spaces), except for a few exceptional choices of dihedral angles, and that its flexion admits a closed-form parameterization. Examples are constructed that are non-self-intersecting and belong \emph{exclusively} to the equimodular elliptic type. To support applications in computational geometry, a numerical pipeline is developed that searches for candidate solutions, verifies them using the explicit algebraic characterization, and constructs and visualizes the resulting polyhedra; numerical validations achieve high precision. Taken together, these results provide constructive criteria, algorithms, and validated examples for the equimodular elliptic type, enabling the design of a broad range of flexible Kokotsakis mechanisms.\\[-0.4cm]
%
%
%We investigate flexible Kokotsakis polyhedra with a quadrangular base of the \emph{equimodular elliptic} type, a subclass within the equimodular class defined by Izmestiev. We close a key gap in the literature by providing the first \emph{explicit constructions} of this type and deriving an \emph{explicit existence criterion} that links flat and dihedral angles. Beyond existence, we prove that our analytical examples are flexible in $\mathbb{R}^3$ and provide closed-form parameterizations of their flexion. We further show that these examples are non-self-intersecting and belong \emph{exclusively} to the equimodular elliptic type, even after switching the boundary strips.
%
%To support practical use in computational geometry, we develop a numerical pipeline that searches for candidate solutions, verifies them against the explicit existence criterion, and constructs and visualizes the resulting polyhedra; numerical validations are achieved to a tolerance of at least $10^{-12}$. We also outline the fabrication of a physical prototype to corroborate the realizability of our designs.
%
%Together, these results provide constructive criteria, algorithms, and validated examples for the equimodular elliptic type, enabling modeling, simulation, and design of a variety of flexible Kokotsakis mechanisms.
\end{abstract}

%%Graphical abstract
%\begin{graphicalabstract}
%\includegraphics{grabs}
%\end{graphicalabstract}

%% Research highlights
%\begin{highlights}
%\item First explicit realizations of equimodular–elliptic Kokotsakis polyhedra (with quadrangular base) that is non-self-intersecting, flex in $\mathbb{R}^3$ and belong exclusively to this class (covering $M<1$ and $M>1$).
%\item A constructive existence criterion linking flat and dihedral angles, enabling direct synthesis and certification of candidates.
%\item Theory-guided numerical search and verification with $10^{-12}$ (Bricard) and $10^{-14}$ (period) tolerances.
%\item Reproducible CAD pipeline: edge/angle generation, coordinate formulas, and interactive flexion with re-validation at each frame.
%\item Open repository with Mathematica notebooks (symbolic checks) and Python code for search, construction, and visualization.
%\end{highlights}

%% Keywords
\begin{keyword}
Equimodular elliptic class \sep elliptic function %Existence Criterion
\sep explicit construction
\sep flexible net \sep Kokotsakis polyhedron \sep quasi-symmetric net.
% keywords here, in the form: keyword \sep keyword
% PACS codes here, in the form: \PACS code \sep code
% MSC codes here, in the form: \MSC code \sep code
% or \MSC[2008] code \sep code (2000 is the default)
% flexible nets \sep Kokotsakis polyhedron \sep equimodular elliptic type
\end{keyword}
\end{frontmatter}

%% Add \usepackage{lineno} before \begin{document} and uncomment
%% following line to enable line numbers
%% \linenumbers
%% main text
%%
%% Use \section commands to start a section

%\newpage
\section{Introduction and Prior Work}
\label{sec:intro}
The increasing interest in flexible and deployable structures stems from their broad utility across fields like robotics, solar cells, meta-materials, art, and architecture. In architecture, these structures simplify construction, enable innovative designs, and enhance building functionality. Notable examples include Santiago Calatrava Valls’ projects such as the Florida Polytechnic University, the UAE Pavilion at Expo Dubai, and the Quadracci Pavilion in Milwaukee. Norman Foster’s Bund Finance Center also showcases such principles. While architects typically rely on simple mechanisms, we explore a wider range of mechanisms with potential for architectural design. Unlike traditional kinematic frame structures (\nocite{escrig1993,calatrava1981,li2020}see \cite{escrig1993}--\cite{li2020}), our flexible polyhedral mechanisms enable full surface coverage, making them suitable for diverse applications. Such deployable structures, particularly those that transition into flat or curved shapes, are gaining attention for their potential to optimize construction. Recent work includes rigid-foldable origami (\nocite{demaine-book,evans2015,evans2016,ciang-2019-cf,song-2017,Tac09a,Tac10a,tachi-2010,tachi-2011-onedof,tachi-2013-composite}see \cite{demaine-book}--\cite{tachi-2013-composite}) and programmable meta-materials (\nocite{Callens2018,dieleman2020,Dudte2016,auxetic-2022,Konakovic:2016:BDC:2897824.2925944,konakovic2018,silverberg2014}see \cite{Callens2018}--\cite{silverberg2014}), which often involve mechanisms with multiple degrees of freedom.

The study of flexible polyhedra -- polyhedral structures capable of continuous isometric deformations while their faces remain rigid -- has long been an active area of research at the interface of geometry, kinematics, and mechanics. Interest in such mechanisms is driven both by their intrinsic mathematical appeal and by their applications in architectural geometry, deployable structures, and origami-inspired design. Since Cauchy's classical rigidity theorem~\cite{Cauchy1813} established that convex polyhedra are rigid, research has focused on identifying and characterizing non-convex configurations that admit flexibility. Bricard's seminal classification of flexible octahedra~\cite{Bricard1897} not only revealed the existence of such extraordinary mechanisms but also introduced algebraic methods that remain central to the field.

Within this context, \emph{Kokotsakis polyhedra} form one of the most fundamental families of flexible mechanisms, named after Antonios Kokotsakis, who first studied them in the 1930s~\cite{kokotsakis33}. A Kokotsakis polyhedron is a polyhedral surface in $\mathbb{R}^3$ consisting of one $n$-gon (the base), $n$ quadrilaterals attached to each side of the $n$-gon, and $n$ triangles placed between every two consecutive quadrilaterals. \bluenew{See Figure~\ref{fig:pic000ab}}. While generically rigid, certain configurations of Kokotsakis polyhedra admit continuous flexion. Kokotsakis derived necessary and sufficient conditions for infinitesimal flexibility and described several flexible subclasses.

% \begin{figure}[!t]%[thbp!]
%     \centering
%     \includesvg[scale=0.6]{img/pic000_1.svg}
%     \caption{A QS-net is a particular Kokotsakis polyhedron with quadrangular base. The flat angles of the same color are equal, and the angles of different shades of one color complement each other to $\pi$. Such a net is generically flexible. % besides a few exceptional values of flat and dihedral angles.
%     }%A Kokotsakis polyhedron with quadrangular base and the notation for its angles. \textbf{INSTEAD OF INTRODUCING THE NOTATION HERE, WE CAN ILLUSTRATE CONDITION \eqref{eq:QS} BY HIGHLIGHTING EQUAL AND COMPLEMENTARY ANGLES IN THE FIGURE, INTRODUCING SOME COLOR-CODE FOR THAT.}} %  and vertices.}
%     \label{fig:pic000}
% \end{figure}

\begin{figure}[!t]
    \centering
    \includegraphics[width=0.4\textwidth]{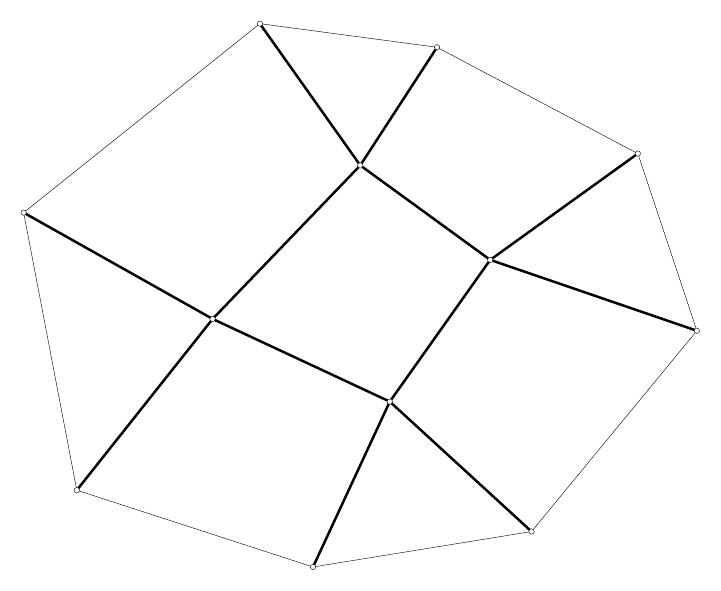}
    \hfill 
    \includegraphics[width=0.46\textwidth]{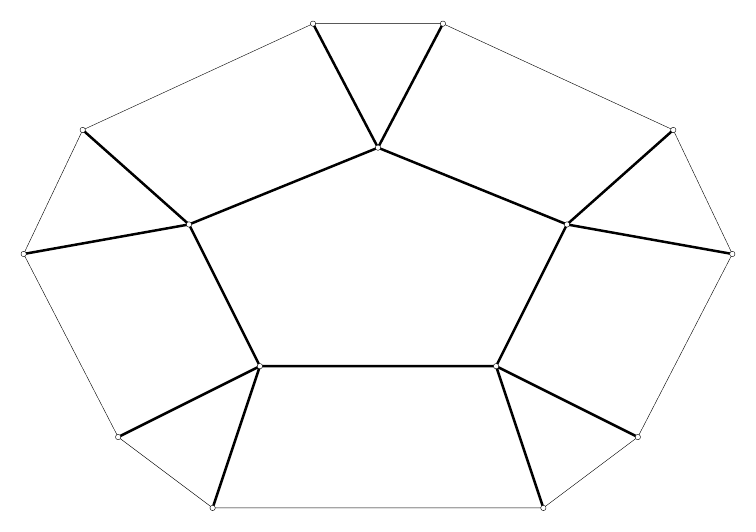}
    \caption{\bluenew{Examples of Kokotsakis polyhedra with (left) a quadrangular base ($n=4$) and (right) a pentagonal base ($n=5$).}}
    \label{fig:pic000ab}
\end{figure}

The quadrangular case ($n=4$) has been the subject of extensive investigation, starting with early contributions by Sauer and Graf~\cite{sauer-graf} and later by Stachel and collaborators (\nocite{stachel2010,nawratil-2011,nawratil-2012}see \cite{stachel2010}--\cite{nawratil-2012}), who employed resultant-based algebraic methods to classify several subclasses. Infinitesimal flexibility was addressed by Bobenko, Hoffmann, Schief \cite{Schief2008}, and Karpenkov \cite{karpenkov-2010}.
The culmination of this line of work is the comprehensive classification by Izmestiev~\cite{izmestiev-2017}, who established eight principal classes of flexible Kokotsakis polyhedra and showed that every flexible quadrangular Kokotsakis polyhedron belongs to one of them (possibly after switching boundary strips): orthodiagonal, isogonal, equimodular, conjugate-modular, linear compound, linearly conjugate, chimera, and trivial class.

Since the classification of flexible polyhedra is complicated~\cite{LiuEtAl2024}, it is natural to search for polyhedra that are approximately flexible with high tolerance, which is enough for any practical applications. Beautiful examples were constructed by Gor’kavyi and Milka~\cite{Milka}. A general approach was developed in a recent series of papers~\nocite{quadmech-2024, pirahmad2024, pirahmad2025, yorov2025}\cite{quadmech-2024}--\cite{yorov2025}, in which approximately flexible polyhedra were constructed by optimizing their analogs in isotropic geometry.

The detailed investigation of Izmestiev's classes is ongoing. Each class in Izmestiev's framework is subdivided into subclasses, and several open questions remain unresolved. First, it is not guaranteed that every subclass contains realizable examples (flexible Kokotsakis polyhedra). Second, even when examples exist, flexibility is often established only in~$\mathbb{C}^3$, not necessarily in~$\mathbb{R}^3$. Third, realizable examples in $\mathbb{R}^3$ may self-intersect, %be self-intersecting,
raising the question of if %whether
each subclass admits non-self-intersecting flexible polyhedra. Finally, since some classes intersect with each other, an important question is if %whether
there exist examples belonging exclusively to a given class. Addressing these questions begins with constructing explicit examples within each class.

Progress has been made in several directions. The orthodiagonal involutive type, first studied by Sauer and Graf~\cite{sauer-graf} as the \emph{T-surfaces}, or \emph{T-nets}, and the orthodiagonal antiinvolutive type, analyzed by Erofeev and Ivanov~\cite{ErofeevIvanov2020}, provide explicit examples and parametrizations within the orthodiagonal class, confirming conjectures posed by Stachel. Recent work has extended this class to Kokotsakis meshes with skew (non-planar) faces, see Aikyn et al.~\cite{AikynEtAl2024a} and Liu et al.~\cite{LiuEtAl2024}. Flexible polyhedra of the isogonal class were also identified in~\cite{sauer-graf} as \emph{discrete Voss surfaces}, or \emph{V-nets}. These advances have yielded examples for the first two classes (orthodiagonal and isogonal) in Izmestiev's classification. \bluenew{In addition, other variations, such as the equimodular conic type and linear compound types, have been actively explored in recent literature; for instance, through the classification of pluripotent tiling patterns \cite{dieleman2020}, the kinematic analysis of quadrilateral creased papers \cite{he2020rigid, he2026infinitely}, and the compatibility of rigid-foldable polygons \cite{hu2026lorentz}.} However, until now, no explicit examples have been known for 
%the third class, 
the equimodular \bluenew{elliptic} type.

This paper focuses on the \emph{equimodular elliptic type}, a subclass of the equimodular class identified by Izmestiev~\cite{izmestiev-2017}.
This subclass is defined in terms of Jacobi elliptic functions and believed to have the maximal number %(eight)
of degrees of freedom. %among all the classes
%~\cite[\S3.3.1]{izmestiev-2017}.
Our work contributes new constructive and analytical results, thereby advancing the understanding and practical realization of this class of flexible mechanisms and addressing the open problems described above. For the first time, we present an explicit construction of such polyhedra, thereby demonstrating that the equimodular elliptic class is non-empty.
Namely, we introduce an % naturally defined and
easily constructible class of \emph{quasi-symmetric nets}, or \emph{QS-nets}, defined by equality and complementarity (to $\pi$) of certain flat angles as shown in \bluenew{Figure~\ref{fig:pic000}}.
%%%%%%%%%
%the ones whose flat angles (see Figure~\ref{fig:pic000}) satisfy the relations
%\begin{equation*}
%\begin{aligned}
%    &\alpha_1 = \alpha_4 = \delta_2 = \pi - \delta_3,
%    &\qquad &\beta_1 = \beta_4 = \gamma_2 = \pi -  \gamma_3,
%    \\
%    &\gamma_1 = \gamma_4 = \beta_2 = \pi -  \beta_3,
%    &\qquad &\delta_1 = \delta_4 = \alpha_2 = \pi - \alpha_3.
%\end{aligned}
%    \label{eq:QS} \tag{QS}
%\end{equation*}
In Theorem~\ref{main_th:th1}, we show that under a simple additional condition (ellipticity), all QS-nets belong to the equimodular elliptic class, are flexible in $\mathbb{R}^3$ besides a few exceptional values of dihedral angles, %for a certain range of their dihedral angles,
and admit a simple closed-form expression for the flexion.

% \begin{figure}[!t]
%         \centering
%         \includegraphics[width=0.6\textwidth]{figure_1.pdf}
%         %\includegraphics[width=0.6\textwidth]{fig/figure_1.pdf}
%         % \includegraphics[scale=0.4]{fig/figure_1.pdf}
%         \caption{A QS-net is a particular Kokotsakis polyhedron with quadrangular base. The flat angles of the same \bluenew{solid arc color (or arc count) are equal; dashed arcs represent supplementary angles of their solid counterparts.} Such a net is generically flexible.}
%          \label{fig:pic000}
% \end{figure}

\begin{figure}[!t]
        \centering
        \includegraphics[width=0.6\textwidth]{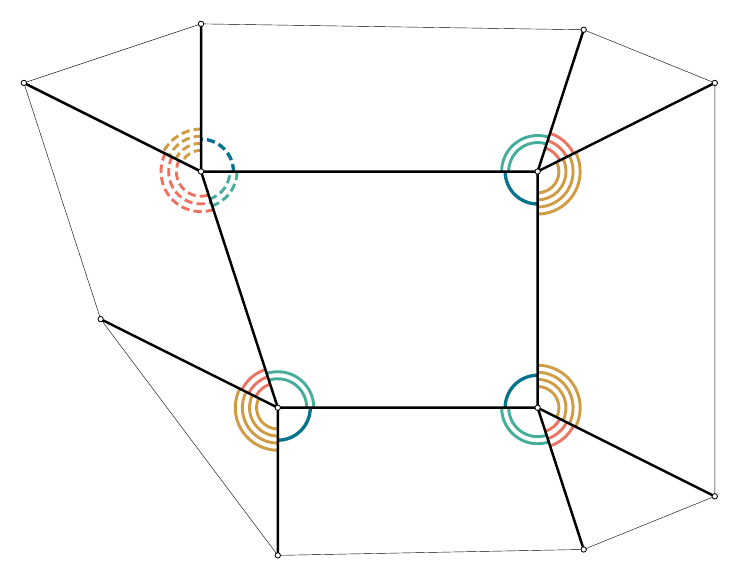}
        \caption{A QS-net is a particular Kokotsakis polyhedron with quadrangular base. The flat angles of the same \bluenew{solid arc color (or arc count) are equal; dashed arcs represent supplementary angles of their solid counterparts.} Such a net is \protect\lastfix{flexible for almost any choice of such flat angles.}}
         \label{fig:pic000}
\end{figure}

Furthermore, we introduce a construction with the dihedral rather than flat angles as design parameters, allowing for better control of the visual
appearance of %the resulting
polyhedra. %Building on Izmestiev's framework,
In Proposition~\ref{main_th:prop1}, we derive an algebraic characterization %explicit existence criteria
of equimodular elliptic type in terms of flat and dihedral angles.

Beyond existence, we %prove that our
give examples that are flexible in $\mathbb{R}^3$, are non-self-intersecting, and belong exclusively to the equimodular elliptic type, even after switching boundary strips. We further use Proposition~\ref{main_th:prop1} to obtain numerical examples with tolerance $10^{-12}$ or better, sufficient for applications in real-world mechanism design. To complement the theoretical developments, we introduce computational algorithms that (i) search for and verify candidate solutions to the existence system with desired tolerance, (ii) construct and visualize the resulting polyhedra with interactive flexion, and (iii) guide the physical construction of models.

In summary, the contributions of this work are:
\begin{itemize}
    \item the construction of the first explicit closed-form examples of flexible (in $\mathbb{R}^3$), non-self-intersecting polyhedra that belong exclusively to equimodular elliptic type;
    %\item the derivation of simple closed-form examples of equimodular elliptic type;
    \item %the derivation of
    an explicit algebraic characterization %existence criteria for
    of equimodular elliptic type; % polyhedra;
    \item the development of numerical algorithms for %solution
    search, verification, and visualization of such polyhedra;
    %\item the construction of the first explicit closed-form examples all flexible in $\mathbb{R}^3$, non-self-intersecting and purely equimodular elliptic;
    \item the construction of the numerical examples of polyhedra of equimodular elliptic type with desired tolerance.
\end{itemize}

The paper is organized as follows. Section~\ref{sec:notation} introduces the necessary notation
and %. Section~\ref{sec:mainprop}
states the main results,  Theorem~\ref{main_th:th1} and Proposition~\ref{main_th:prop1}. Section~\ref{sec:prexcrit} contains the proofs of these results. %Section~\ref{sec:alg} presents the numerical search algorithm.
Section~\ref{sec:ex} presents the numerical search algorithm and reports closed-form and numerical examples. We have produced a video capturing the flexions; see the supplementary material (\texttt{QS-Nets.mp4})~\cite{Nurmatov2025Repo}.
%(\texttt{QS-Nets Anonymous.mp4}).
%which can be viewed on YouTube.
Section~\ref{sec:disc} discusses implications for further research on flexible polyhedra and potential CAD applications. %, and contains a
A brief conclusion in Section~\ref{sec:concl} closes the paper.
% The appendices collect technical details: coordinate formulas (\ref{app:coord}), edge generation (\ref{app:edg}), angle generation (\ref{app:ang}), inequalities and Bricard's equations (\ref{app:guessflex}), and the construction/interactive flexion algorithm (\ref{app:vis}).

\section{Notation and Main Results}
\label{sec:notation}

\bluenew{Before diving into the rigorous mathematical framework, we provide an informal roadmap of our main results and elaborate on the motivation for introducing them. Constructing examples of the equimodular elliptic type directly from its formal definition is non-trivial, whether approached numerically or explicitly. Consequently, our objective was not merely to find a single isolated example, but rather to develop a systematic framework that offers greater design versatility -- specifically, the ability to prescribe the dihedral angles and the flat angles of the base face. This goal directly motivated the algebraic characterization formulated in Proposition~\ref{main_th:prop1}.}

\bluenew{The central challenge in designing flexible Kokotsakis polyhedra lies in satisfying a dense system of algebraic compatibility conditions. As we formalize later in Proposition~\ref{main_th:prop1}, characterizing a general flexible mesh of the equimodular elliptic type requires solving a highly non-linear system of algebraic equations. % (specifically, 12 equations for the 16 unknown flat angles). 
Although these general %kinematic 
equations can be satisfied numerically, obtaining explicit, closed-form solutions for valid flat angles remains notoriously difficult. Therefore, a practical strategy is to seek specific geometric configurations that inherently enforce these conditions.} % to hold.}

\bluenew{The concept of Quasi-Symmetric nets (QS-nets) was discovered directly through this algebraic challenge. \lastfix{Specifically, by calculating some numerical solutions to the system of Proposition~\ref{main_th:prop1}, we observed highly structured symmetries among the resulting flat angles. Imposing these symmetries forces complex cross-terms in the compatibility equations to naturally cancel out. We used these observed symmetries to construct a general family of examples, introduced in Theorem~\ref{main_th:th1} as the QS-net.}
%To make the problem tractable and practically applicable, Theorem~\ref{main_th:th1} introduces the QS-net as a highly structured subclass that reduces this daunting algebraic system into a straightforward, explicit generative recipe. 
In short, Theorem~\ref{main_th:th1} states that by focusing just on one vertex of the base face, fixing its corresponding base angle to exactly $90^\circ$, and freely choosing the three other independent flat angles around this vertex, the entire 16-angle polyhedron is immediately determined by these three flat angles via the symmetry and complementarity relations shown in Figure~\ref{fig:pic000}, yielding an explicitly flexible Kokotsakis polyhedron.}

\bluenew{To build concrete intuition for this mechanism before navigating the technical proofs, Figure~\ref{dfig:pic001} visualizes the continuous folding motion of a QS-net (the specific parameters for this mesh are detailed later in Example~\ref{ex:ex1}). The remainder of this section, along with Section 3, establishes the rigorous definitions and algebraic proofs required to validate these claims.}

\begin{figure}[htbp!]
        \centering
        \includegraphics[width=0.6\textwidth]{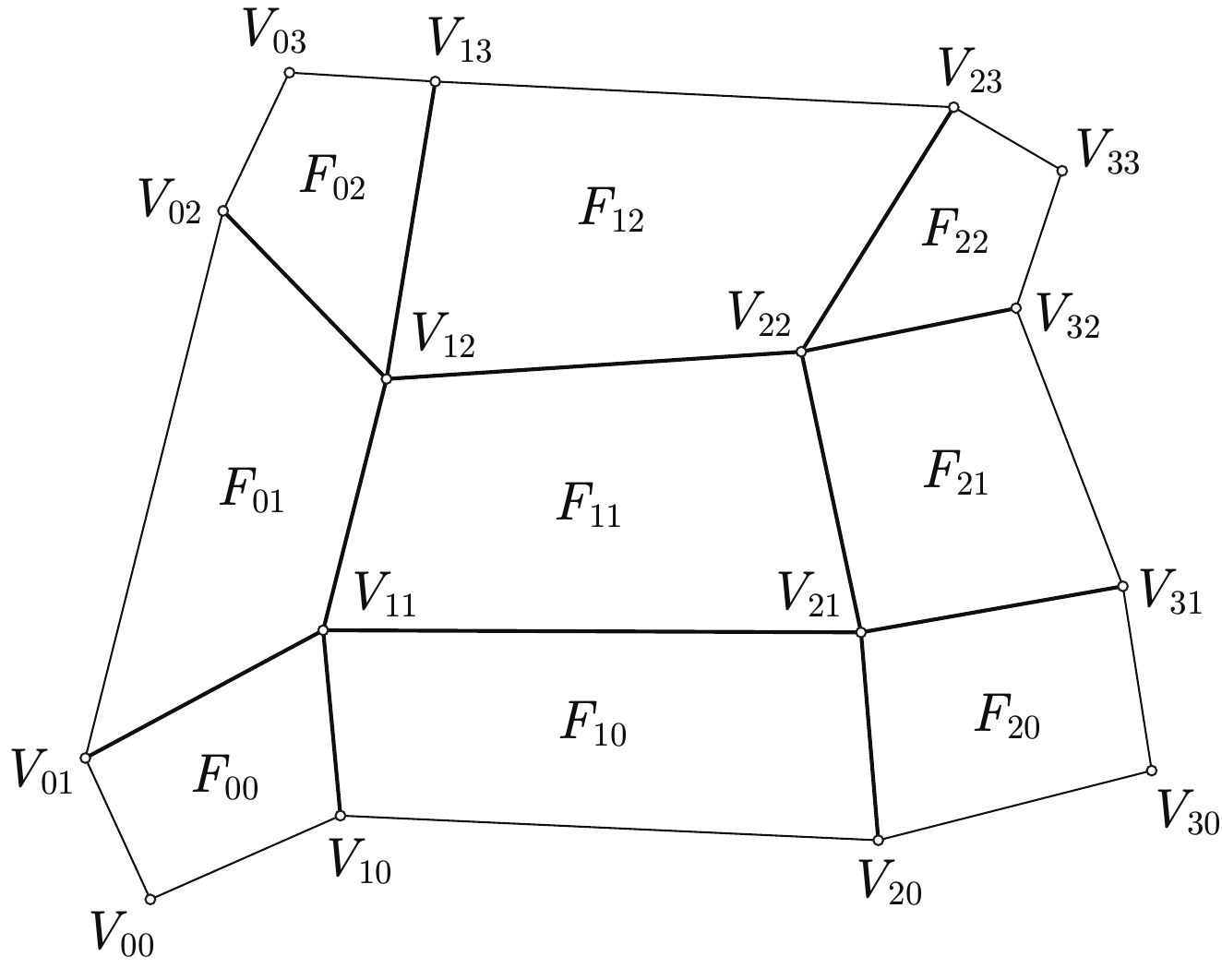}
    \caption{Illustration of a $3 \times 3$ net.}
    \label{fig:pic001}
\end{figure}

Let us introduce the main notions, notation, and state the main results.

%\begin{definition}[See Figure~\ref{fig:pic001}]
An $m \times n$ \emph{net} (\emph{with convex faces}) is a collection of $(m+1)(n+1)$ points $V_{ij} \in \mathbb{R}^3$, indexed by integers $0 \leq i \leq m$ and $0 \leq j \leq n$ such that, for all $0 \leq i < m$ and $0 \leq j < n$, the points $V_{ij}, V_{i+1,j}, V_{i+1,j+1}, V_{i,j+1}$ form the consecutive vertices of a convex quadrilateral, denoted $F_{ij}$. See Figure~\ref{fig:pic001}. (See \cite{Mor20} for an elementary introduction to indexed/labeled polyhedra.)

An $m \times n$ net $V_{ij}$ is \emph{flexible} (in $\mathbb{R}^3$) if belongs to a continuous family (called \emph{flexion}) of $m \times n$ nets $V_{ij}(t)$, where 
\lastfix{$t \in [0, \epsilon)$} %$t \in (-\epsilon, \epsilon)$ 
for some $\epsilon > 0$, such that all corresponding quadrilaterals $F_{ij}(t)$ are congruent to $F_{ij}(0)$ for all \lastfix{$t \in [0, \epsilon)$}, %$t \in (-\epsilon, \epsilon)$, 
$V_{ij}(0)$ is the original net $V_{ij}$, and the nets $V_{ij}(t)$ are pairwise non-congruent.
%and the net $V_{ij}(t)$ is not congruent to $V_{ij}$ for $t \neq 0$.
%\end{definition}

In what follows, we consider only $3\times 3$ nets and call them simply \emph{polyhedra}. We omit corner points $V_{00}$, $V_{30}$, $V_{03}$, and $V_{33}$, %as they are
irrelevant for the flexibility.

We denote the vertices $A_2:=V_{11}$, $B_2:=V_{10}$, $C_2:=V_{01}$, etc. as shown in Figure~\ref{fig:pic002}. Denote also $A_0:=A_4$ and $A_5:=A_1$. By \emph{flat angles} of the $3\times 3$ net we mean $\alpha_i = \angle A_{i-(-1)^i}A_iB_i$, $\beta_i = \angle B_iA_iC_i$, $\gamma_i = \angle C_iA_iA_{5-i}$, $\delta_i = \angle A_{i-1}A_iA_{i+1}$, where $i=1, \ldots, 4$. By definition, $\alpha_i$, $\beta_i$, $\gamma_i$, $\delta_i \in (0, \pi)$.
% \begin{center}
%     \centering
%     \includesvg[scale=0.7]{img/pic002.svg}
%     \captionof{figure}{Example of $3 \times 3$ net without corner points}
%     \label{fig:pic002}
% \end{center}

% \begin{figure}[htbp!]
%    \centering
%    \includesvg[scale=0.7]{img/pic002.svg}
%    \caption{Notation for the vertices, flat and dihedral angles of a $3 \times 3$ net without corners.} % points.}
%    \label{fig:pic002}
% \end{figure}

\begin{figure}[htbp!]
        \centering
        \includegraphics[width=0.6\textwidth]{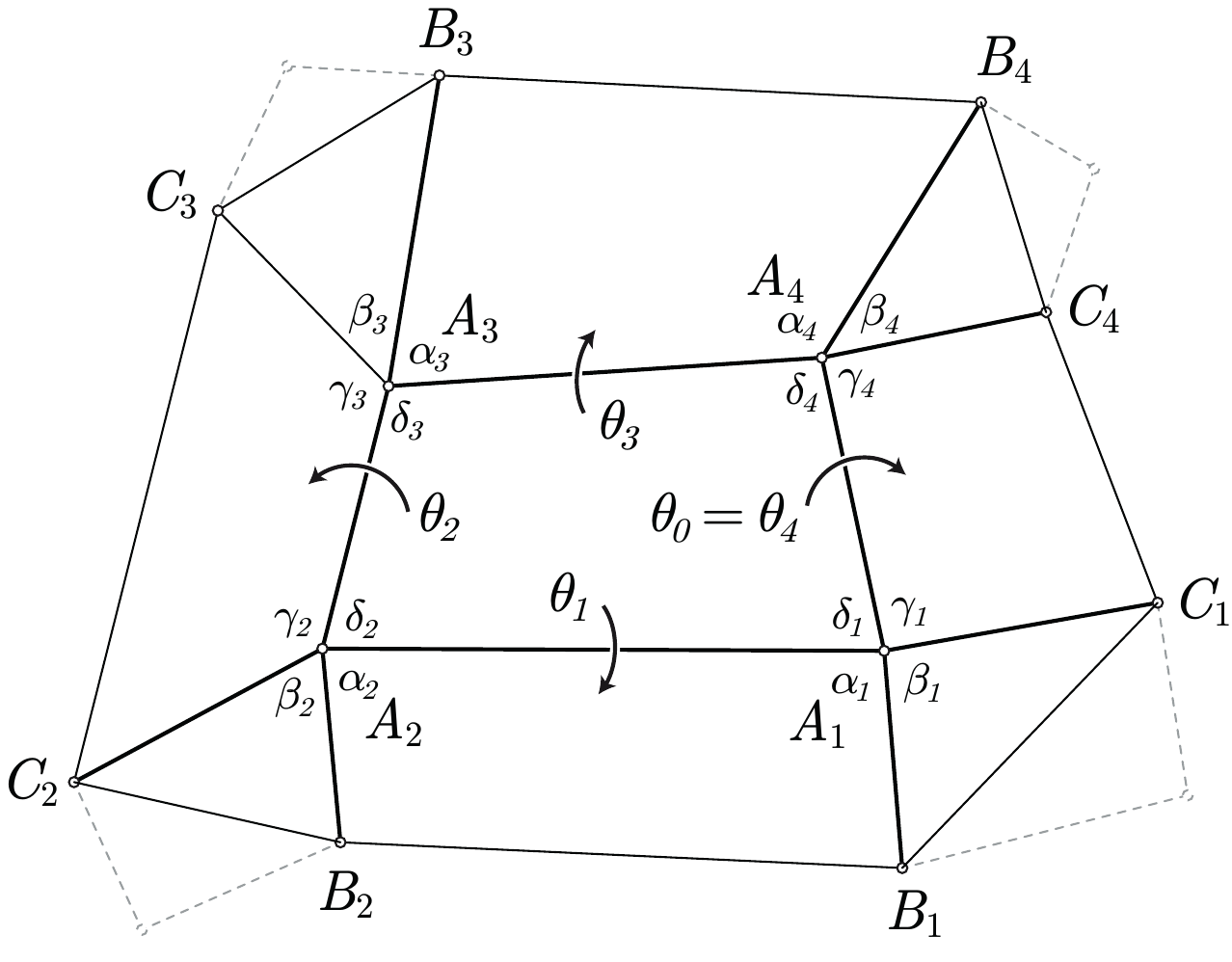}
    \caption{Notation for the vertices, flat and dihedral angles of a $3 \times 3$ net without corners.} % points.}
   \label{fig:pic002}
\end{figure}

% Introduce oriented normals $\mathbf{n}_0 :=\mathbf{A}_2\mathbf{A}_1\times\mathbf{A}_2\mathbf{A}_3$, $\mathbf{n}_1 :=\mathbf{A}_2\mathbf{B}_2\times\mathbf{A}_2\mathbf{A}_1$, $\mathbf{n}_2 :=\mathbf{A}_2\mathbf{A}_3\times\mathbf{A}_2\mathbf{C}_2$, $\mathbf{n}_3 :=\mathbf{A}_3\mathbf{A}_4\times\mathbf{A}_3\mathbf{B}_3$, $\mathbf{n}_4 :=\mathbf{A}_1\mathbf{C}_1\times\mathbf{A}_1\mathbf{A}_4$. For each $i = 1, \ldots, 4$, the (\emph{oriented}) \emph{dihedral angle} of the $3\times 3$ net is defined as
% \begin{align*}
%  \theta_i:=
%     \begin{cases}
%      \angle \left(\mathbf{n}_0, \mathbf{n}_i\right) \in (0, \pi), &\text{if }  \det\left(\mathbf{n}_i,\mathbf{n}_0,\mathbf{A}_{i+1}\mathbf{A}_i\right) > 0,\\
%      -\angle \left(\mathbf{n}_0, \mathbf{n}_i\right) \in [-\pi, 0], &\text{if }  \det\left(\mathbf{n}_i,\mathbf{n}_0,\mathbf{A}_{i+1}\mathbf{A}_i\right)\leq 0.
%     \end{cases}
% \end{align*}

%\newpage

Introduce oriented normals $\vv{n_0} \coloneqq\vv{A_2A_1}\times\vv{A_2A_3}$,
$\vv{n_1} \coloneqq\vv{A_2A_1}\times\vv{A_2B_2}$,
$\vv{n_2} \coloneqq\vv{A_2C_2}\times\vv{A_2A_3}$,
$\vv{n_3} \coloneqq\vv{A_3B_3}\times\vv{A_3A_4}$,
$\vv{n_4} \coloneqq\vv{A_1A_4}\times\vv{A_1C_1}$.
For each $i = 1, \ldots, 4$, the (\emph{oriented})
\emph{dihedral angle} of the $3\times 3$ net is defined as
\begin{equation*}
 \theta_i\coloneqq
    \begin{cases}
     \angle \left(\vv{n_0}, \vv{n_i}\right) \in (0, \pi),
     &\text{if }  \det\left(\vv{n_0},\vv{n_i},\vv{A_{i+1}A_i}\right) > 0,\\
     -\angle \left(\vv{n_0}, \vv{n_i}\right) \in [-\pi, 0],
     &\text{if }  \det\left(\vv{n_0},\vv{n_i},\vv{A_{i+1}A_i}\right)\leq 0.
    \end{cases}
\end{equation*}
% Introduce oriented normals $\vv{n_0} :=\vv{A_2A_1}\times\vv{A_2A_3}$, $\vv{n_1} :=\vv{A_2B_2}\times\vv{A_2A_1}$, $\vv{n_2} :=\vv{A_2A_3}\times\vv{A_2C_2}$, $\vv{n_3} :=\vv{A_3A_4}\times\vv{A_3B_3}$, $\vv{n_4} :=\vv{A_1C_1}\times\vv{A_1A_4}$. For each $i = 1, \ldots, 4$, the (\emph{oriented}) \emph{dihedral angle} of the $3\times 3$ net is defined as
% \begin{align*}
%  \theta_i:=
%     \begin{cases}
%      \angle \left(\vv{n_0}, \vv{n_i}\right) \in (0, \pi), &\text{if }  \det\left(\vv{n_i},\vv{n_0},\vv{A_{i+1}A_i}\right) > 0,\\
%      -\angle \left(\vv{n_0}, \vv{n_i}\right) \in [-\pi, 0], &\text{if }  \det\left(\vv{n_i},\vv{n_0},\vv{A_{i+1}A_i}\right)\leq 0.
%     \end{cases}
% \end{align*}
Although the theoretical framework presented is valid for dihedral angles in the full range $[-\pi, \pi)$, %without loss of generality, in our work
we consider the case where $\theta_i \in (0, \pi)$, unless otherwise explicitly indicated. Furthermore, we define $\theta_0 := \theta_4$
%, as illustrated in
(see Figure~\ref{fig:pic002}).

\begin{definition}\label{def:qs}
A $3\times 3$ net is \emph{quasi-symmetric}, or a \emph{QS-net}, if its flat angles satisfy the relations (see Figure~\ref{fig:pic000})
\begin{equation*}
\begin{aligned}
    &\alpha_1 = \alpha_4 = \delta_2 = \pi - \delta_3,
    &\qquad &\beta_1 = \beta_4 = \gamma_2 = \pi -  \gamma_3,
    \\
    &\gamma_1 = \gamma_4 = \beta_2 = \pi -  \beta_3,
    &\qquad &\delta_1 = \delta_4 = \alpha_2 = \pi - \alpha_3.
\end{aligned}
    \label{eq:QS} \tag{1}
\end{equation*}
\end{definition}

A polyhedron has \emph{elliptic type} if
\begin{gather*}
\alpha_i \pm \beta_i \pm \gamma_i \pm \delta_i \neq 0 \, (\operatorname{mod}{2\pi}) \label{eq:N.0} \tag{2}
\end{gather*}
for all choices of the signs \(\pm\) and each $i=1, \ldots, 4$. In what follows, we consider polyhedra of elliptic type only, unless otherwise explicitly indicated, and call them \emph{elliptic polyhedra}. We introduce the notation:
% \begin{align*}\label{eq:N.1}\tag{N.1}
% &\scalebox{1.0}{$\sigma_i := \frac{\alpha_i + \beta_i + \gamma_i + \delta_i}{2},
% \
% \overline{\alpha}_i := \sigma_i - \alpha_i,
% \
% \overline{\beta}_i := \sigma_i - \beta_i,
% \
% \overline{\gamma}_i := \sigma_i - \gamma_i,
% \
% \overline{\delta}_i := \sigma_i - \delta_i,
% \
% \varepsilon_i:= \sign{(\sin{\sigma_i})}$}.
% \end{align*}
\begin{align*}\label{eq:N.1}\tag{3}
\left.
\begin{aligned}
&\sigma_i := \frac{\alpha_i + \beta_i + \gamma_i + \delta_i}{2},
&&\overline{\alpha}_i := \sigma_i - \alpha_i,
&&
\overline{\beta}_i := \sigma_i - \beta_i,
\\
&\varepsilon_i:= \sign{(\sin{\sigma_i})},
&&\overline{\gamma}_i := \sigma_i - \gamma_i,
&&
\overline{\delta}_i := \sigma_i - \delta_i.
\end{aligned}
\right.
\end{align*}
Next, we define
\begin{align*}\label{eq:N.2}\tag{4}
    \left.
    \begin{aligned}
    &a_i := \frac{\sin \alpha_i}{\sin \overline{\alpha}_i},
&&
b_i := \frac{\sin \beta_i}{\sin \overline{\beta}_i},
&&
c_i := \frac{\sin \gamma_i}{\sin \overline{\gamma}_i},
&&
d_i := \frac{\sin \delta_i}{\sin \overline{\delta}_i},
\\
&M_i := a_i \, b_i \, c_i \, d_i,
&&
r_i := a_i d_i,
&&
s_i := c_i d_i,
&&
f_i := a_i c_i,
\\
&u_i := 1 - M_i,
&&
x_i := \frac{1}{r_i - 1},
&&
y_i := \frac{1}{s_i - 1},
&&
z_i := \frac{1}{f_i - 1}.
    \end{aligned}
    \right.
\end{align*}

A polyhedron has \emph{equimodular elliptic type} if it has elliptic type and the following three conditions hold:
\begin{itemize}
    \item[(1)] The vertices have equal \emph{moduli}, that is,
    \begin{gather*}
        M_1 = M_2 = M_3 = M_4 =: M. \label{eq:N.3} \tag{5}
    \end{gather*}
    \item[(2)] The \emph{amplitudes} at common vertices match, that is,
    \begin{gather*}
        r_1 = r_2,
        \quad
        s_1 = s_4,
        \quad
        r_3 = r_4,
        \quad
        s_2 = s_3.
         \label{eq:N.4} \tag{6}
    \end{gather*}
    \item[(3)] The sum of \emph{phase shifts} is a \emph{period}, that is, for some $e_1$, $e_2$, $e_3 \in \left\{+1, -1\right\}$ we have
    \begin{align*}\label{eq:N.5} \tag{7}
        &t_1 + e_1t_2 + e_2t_3 + e_3t_4 \in \Lambda.
    \end{align*}
\end{itemize}

    In the latter equation, we use the following notation. First, we introduce the \emph{elliptic modulus}
\begin{gather*}
        k:=
    \begin{cases}
        \sqrt{1-M}, & \text{if } M < 1, \\
        \sqrt{1-M^{-1}}, & \text{if } M > 1.
    \end{cases}
    \label{eq:N.7} \tag{8}
    \end{gather*}
    The \emph{basis quater periods} %$K$ and $K'$
    are
\begin{align*}\label{eq:N.6}\tag{9}
\left.
\begin{aligned}
    &K:= K(k) = \int_{0}^{1} \frac{dx}{\sqrt{(1-x^2)(1-k^2x^2)}}, \\
    &K':= K(k') = \int_{0}^{1} \frac{dx}{\sqrt{(1-x^2)(1-k'^2x^2)}},
\end{aligned}
\right.
\end{align*}
where \(k' = \sqrt{1 - k^2}\).
Then we introduce the lattice
    \begin{align*}
        &\Lambda :=
    \begin{cases}
        \{4Km + 2\mathbf{i}K'n: \ m, n \in \mathbb{Z}\}, & \text{if } M < 1, \\
        \{4Km + (2K + 2\mathbf{i}K')n: \ m, n \in \mathbb{Z}\}, & \text{if } M > 1,
    \end{cases}
    \end{align*}
where $\mathbf{i}$ (in bold font) is the imaginary unit.

%%%%%%%%%%%%%%%%%%%%%%%%%%%%%%%%%%%%%%%%%%%%%%%%%%%%%%%%
%and phase shifts $t_i$, basis periods $K$ and $K'$ are defined as follows.
%\begin{itemize}
%    \item[] The \emph{basis periods} $K$ and $K'$ are
%\begin{align*}\label{eq:N.6}\tag{N.6}
%\left.
%\begin{aligned}
%    &K:= K(k) = \int_{0}^{1} \frac{dx}{\sqrt{(1-x^2)(1-k^2x^2)}}, \\
%    &K':= K(k') = \int_{0}^{1} \frac{dx}{\sqrt{(1-x^2)(1-k'^2x^2)}},
%\end{aligned}
%\right.
%\end{align*}
%\noindent where \(k' = \sqrt{1 - k^2}\) and \(k\) is the \emph{elliptic modulus} given by
%\begin{gather*}
%        k:=
%    \begin{cases}
%        \sqrt{1-M}, & \text{if } M < 1, \\
%        \sqrt{1-M^{-1}}, & \text{if } M > 1.
%    \end{cases}
%    \label{eq:N.7} \tag{N.7}
%    \end{gather*}

    The \emph{amplitudes} \(p_i\) and \(q_i\) are given by
\begin{align*}\label{eq:N.8}\tag{10}
    p_i:= \sqrt{r_i - 1} \in \mathbb{R}_{>0} \cup \mathbf{i} \mathbb{R}_{>0},
\quad
q_i:= \sqrt{s_i - 1} \in \mathbb{R}_{>0} \cup \mathbf{i} \mathbb{R}_{>0}.
\end{align*}

    Finally, the \emph{phase shift} $t_i$ is given by the following formula (up to slight ambiguity to be resolved in a moment):
\begin{gather*}
        \operatorname{dn}{t_i} :=\operatorname{dn}{(t_i, k)} =
    \begin{cases}
        \sqrt{f_i}, & \text{if } M < 1, \\
        \dfrac{1}{\sqrt{f_i}}, & \text{if } M > 1.
    \end{cases}
    \label{eq:N.9} \tag{11}
    \end{gather*}
Here, the value of the delta amplitude \(\operatorname{dn}{t_i}\) determines the value $t_i$ up to the real half-period \(2K\). This ambiguity is resolved by Table~\ref{tab:tpqsigma}, which shows intervals where $t_i$ lies, depending on the values of $p_iq_i$ and $\sigma_i$.
% \begin{center}
% \centering
% \begin{tabular}{c|c|c|c}
% & \( p_iq_i \in \mathbb{R}_{>0} \) & \( p_iq_i \in \mathbf{i} \mathbb{R}_{>0} \) & \( p_iq_i \in \mathbb{R}_{<0} \) \\
% \hline
% \( \sigma_i < \pi \) & \( (0, \mathbf{i}K') \) & \( (K, K + \mathbf{i}K') \) & \( (2K, 2K + \mathbf{i}K') \) \\
% \( \sigma_i > \pi \) & \( (2K, 2K + \mathbf{i}K') \) & \( (3K, 3K + \mathbf{i}K') \) & \( (0, \mathbf{i}K') \)
% \end{tabular}
% \captionof{table}{Phase shift intervals for different \( p_iq_i \) and $\sigma_i$}
% \label{table:T.1}
% \end{center}
\begin{table}[htbp!]
\centering
\begin{tabular}{c|c|c|c}
& \( p_iq_i \in \mathbb{R}_{>0} \) & \( p_iq_i \in \mathbf{i} \mathbb{R}_{>0} \) & \( p_iq_i \in \mathbb{R}_{<0} \) \\
\hline
\( \sigma_i < \pi \) & \( (0, \mathbf{i}K') \) & \( (K, K + \mathbf{i}K') \) & \( (2K, 2K + \mathbf{i}K') \) \\
\( \sigma_i > \pi \) & \( (2K, 2K + \mathbf{i}K') \) & \( (3K, 3K + \mathbf{i}K') \) & \( (0, \mathbf{i}K') \)
\end{tabular}
\caption{Phase shift intervals for different \( p_iq_i \) and $\sigma_i$.}
\label{tab:tpqsigma}
\end{table}

%\begin{remark}\label{notation:rem1}
Note that there is a minor typographical error in the original definition of the equimodular elliptic type by Izmestiev~\cite[\S3.3.1]{izmestiev-2017}, specifically in the condition equating certain amplitudes; this was first noted by He~\cite{HePhD2021}. Throughout this paper, we adopt the corrected conditions~\eqref{eq:N.4}; the correction is purely notational and does not alter the classification framework.
%\end{remark}

%%%%%%%%%
%\begin{definition}\label{def:qs}
%Numbers $\alpha_i$, $\beta_i$, $\gamma_i$, $\delta_i$, where $i=1\ldots, 4$, are called \emph{quasi-symmetric} if
%\begin{align*}
%    &\alpha_1 = \alpha_4 = \delta_2 = \pi - \delta_3,
%    &&\beta_1 = \beta_4 = \gamma_2 = \pi -  \gamma_3,
%    \\
%    &\gamma_1 = \gamma_4 = \beta_2 = \pi -  \beta_3,
%    &&\delta_1 = \delta_4 = \alpha_2 = \pi - \alpha_3.
%    \label{eq:QS} \tag{QS}
%\end{align*}
%Moreover, a \emph{QS-net} is a polyhedron with quasi-symmetric flat angles.
%\end{definition}
%%%%% 

%\section{Main Results}
%\label{sec:mainprop}
%\input{004_main_proposition}

We are ready to state the main theorem.

\begin{theorem}\label{main_th:th1}
Any elliptic QS-net has equimodular elliptic type and its flat angles satisfy
\begin{equation*}\label{eq:M.0.a}\tag{12a}
\delta_1 = {\pi}/{2}\qquad\text{and}\qquad
\alpha_1, \beta_1, \gamma_1,\overline{\alpha}_1, \overline{\beta}_1, \overline{\gamma}_1, \overline{\delta}_1\in(0, \pi).
\end{equation*}
Conversely, any $\alpha_i$, $\beta_i$, $\gamma_i$, $\delta_i$ satisfying~\eqref{eq:QS}, \eqref{eq:N.0} and~\eqref{eq:M.0.a} are
flat angles of some flexible elliptic QS-net with a flexion given by
%%%%%%%%%%%%%%%%%%%%%%%%%%%%%%%%%
%A QS-net with the flat angles %$\alpha_1$, $\beta_1$, $\gamma_1$, $\delta_1$
%$\alpha_i$, $\beta_i$, $\gamma_i$, $\delta_i$ %, where $i=1,2,3,4$,
%satisfying~\eqref{eq:QS} exists if and only if $\delta_1 = {\pi}/{2}$ and
%$\alpha_1, \beta_1, \gamma_1,\overline{\alpha}_1, \overline{\beta}_1, \overline{\gamma}_1, \overline{\delta}_1\in(0, \pi)$.
%If, in addition, $\alpha_1$, $\beta_1$, $\gamma_1$, $\delta_1$ satisfy condition~\eqref{eq:N.0} for $i=1$, then there exists a QS-net with the flat angles %%$\alpha_1$, $\beta_1$, $\gamma_1$, $\delta_1$
%$\alpha_i$, $\beta_i$, $\gamma_i$, $\delta_i$
%and the dihedral angles
%%%%%%%%%%%%%%%%%%%%%%%%%%%%%%%%%%%%%%%%
%Moreover, any QS-net satisfying condition~\eqref{eq:N.0} for $i=1$ has equimodular elliptic type, is flexible in $\mathbb{R}^3$, and the flexion is given by
%{
%\fontsize{6}{7}\selectfont
\begin{align*}\label{eq:M.0.b}\tag{12b}
\left.
\begin{aligned}
    &\cot{\frac{\theta_1(t)}{2}} =
\frac{-t \sin\beta_1 \pm \sqrt{D(t)}}{\sin\overline{\delta}_1 \sin(\overline{\alpha}_1 - \beta_1) + t^2\sin\overline{\alpha}_1 \sin(\overline{\delta}_1 - \beta_1) },
 \\
    &\cot{\frac{\theta_2(t)}{2}}  = t, \\
    &\cot{\frac{\theta_3(t)}{2}}  =
\frac{t \sin\beta_1 \pm \sqrt{D(t)}}{\sin\overline{\delta}_1 \sin(\overline{\alpha}_1 - \beta_1) + t^2\sin\overline{\alpha}_1 \sin(\overline{\delta}_1 - \beta_1) },
 \\
    &\cot{\frac{\theta_4(t)}{2}}  =\pm
\frac{\sqrt{D(t)}}{ \sin\overline{\gamma}_1\sin\overline{\delta}_1 +  t^2\sin(\overline{\gamma}_1 - \beta_1)  \sin(\overline{\delta}_1 - \beta_1)},
\end{aligned}
%%%%%%%%%%%%%%%%%%%%%%%%%%%%%%%%%%%%
%\begin{aligned}
%    &\cot{\frac{\theta_1(t)}{2}} =
%\frac{-t \sin\beta_1 \pm \sqrt{ \left(\sin\sigma_1 \sin(\bar{\alpha}_1 - \beta_1) + t^2\sin\bar{\alpha}_1 \sin\bar{\beta}_1\right) \left(t^2\sin(\bar{\gamma}_1 - \beta_1) \sin(\bar{\delta}_1 - \beta_1) + \sin\bar{\gamma}_1 \sin\bar{\delta}_1 \right) }}{\sin\bar{\delta}_1 \sin(\bar{\alpha}_1 - \beta_1) + t^2\sin\bar{\alpha}_1 \sin(\bar{\delta}_1 - \beta_1) },
% \\
%    &\cot{\frac{\theta_2(t)}{2}}  := t, \\
%    &\cot{\frac{\theta_3(t)}{2}}  =
%\frac{t \sin\beta_1 \pm \sqrt{ \left(\sin\sigma_1 \sin(\bar{\alpha}_1 - \beta_1) + t^2\sin\bar{\alpha}_1 \sin\bar{\beta}_1\right) \left(t^2\sin(\bar{\gamma}_1 - \beta_1) \sin(\bar{\delta}_1 - \beta_1) + \sin\bar{\gamma}_1 \sin\bar{\delta}_1 \right) }}{\sin\bar{\delta}_1 \sin(\bar{\alpha}_1 - \beta_1) + t^2\sin\bar{\alpha}_1 \sin(\bar{\delta}_1 - \beta_1) },
 %\\
%    &\cot{\frac{\theta_4(t)}{2}}  =\pm
%\frac{\sqrt{ \left(\sin\sigma_1 \sin(\bar{\alpha}_1 - \beta_1) + t^2\sin\bar{\alpha}_1 \sin\bar{\beta}_1\right) \left(t^2\sin(\bar{\gamma}_1 - \beta_1) \sin(\bar{\delta}_1 - \beta_1) + \sin\bar{\gamma}_1 \sin\bar{\delta}_1 \right) }}{ \sin\bar{\gamma}_1\sin\bar{\delta}_1 +  t^2\sin(\bar{\gamma}_1 - \beta_1)  \sin(\bar{\delta}_1 - \beta_1)},
%\end{aligned}
%%%%%%%%%%%%%%%%%%%%%%%%%%%%%%%%%%%%
\right.
\end{align*}
%}
%\\
%\intertext{
where $t$ varies in some interval, the signs in $\pm$ agree, and
\begin{multline*}
D(t):=\left(\sin\sigma_1 \sin(\overline{\alpha}_1 - \beta_1) + t^2\sin\overline{\alpha}_1 \sin\overline{\beta}_1\right) \times\\ \left(t^2\sin(\overline{\gamma}_1 - \beta_1) \sin(\overline{\delta}_1 - \beta_1) + \sin\overline{\gamma}_1 \sin\overline{\delta}_1 \right).
\end{multline*}
%and $t$ varies in some interval.
%is any real number such that $D(t)>0$. This QS-net has equimodular elliptic type, is flexible in $\mathbb{R}^3$, and the flexion is given by~\eqref{eq:M.0.b} for $t$ varying in some interval.
%%%%
%There exists a QS-net if and only if $\alpha_1$, $\beta_1$, $\gamma_1\in(0, \pi)$ and $\delta_1 = \frac{\pi}{2}$ satisfy the condition \eqref{eq:N.0}, and $\overline{\alpha}_1$, $\overline{\beta}_1$, $\overline{\gamma}_1$, $\overline{\delta}_1\in(0, \pi)$. Moreover, any QS-net is of equimodular elliptic type.
\end{theorem}

We show that there exist two elliptic QS-nets with the same flat angles, one of which is flexible and the other one is not (see Proposition~\ref{ex:prop1a}).

Furthermore, we establish an algebraic characterization of polyhedra of elliptic equimodular type
with prescribed flat and dihedral angles of the central face. We aim at a system of algebraic equations, a minimal number of inequalities, and a few ``discrete'' conditions ensuring realizability in~$\mathbb{R}^3$. The result is most concisely stated in terms of the parameters
$u$, $x_i$, $y_i$, $z_i$,
%so that we need just one inequality $u<1$. Such a characterization
leading to an efficient construction of polyhedra of elliptic equimodular type.
%existence criterion of a polyhedron of elliptic equimodular type with prescribed dihedral angles and flat angles of the central face.

\begin{proposition}\label{main_th:prop1}
Assume that $\alpha_i$, $\beta_i$, $\gamma_i$, $\delta_i$, $\theta_i \in(0, \pi)$ satisfy condition \eqref{eq:N.0} for each $i=1,\ldots, 4$. Then a polyhedron with flat angles $\alpha_i$, $\beta_i$, $\gamma_i$, $\delta_i$ and dihedral angles $\theta_i$ exists and has equimodular elliptic type, if and only if
\begin{empheq}[left=\empheqlbrace]{align}
    &\delta_1 + \delta_2 + \delta_3 + \delta_4 = 2\pi,\label{eq:M.1.0} \tag{13a}\\
    &u_1 = u_2 = u_3 = u_4 =: u < 1, \label{eq:M.1.a} \tag{13b}\\
    &x_1 = x_2, \label{eq:M.1.b} \tag{13c} \\
    &x_3 = x_4, \label{eq:M.1.c} \tag{13d}\\
    &y_1 = y_4, \label{eq:M.1.d} \tag{13e}\\
    &y_2 = y_3, \label{eq:M.1.e}  \tag{13f} \\
    &\cos{\delta_i} =  \varepsilon_i\frac{1 - y_iz_iu - x_iz_iu + x_iy_iu}{2\sqrt{x_iy_iu(1+z_i)(1+uz_i)}}, \label{eq:M.1.f} \tag{13g} \\
    &\sin{\theta_i} \sin{\theta_{i-1}} = \varepsilon_i\frac{A_{i1} + A_{i2}y_iz_iu + A_{i3}x_iz_iu + A_{i4}x_iy_iu}{2\sqrt{x_iy_iu(1+z_i)(1+uz_i)}}, \label{eq:M.1.g} \tag{13h}
\end{empheq}
for each $i=1,\ldots, 4$, and there exist $e_1$, $e_2$, $e_3 \in \{+1, -1\}$ such that condition \eqref{eq:N.5} holds, where $\varepsilon_i$, $x_i$, $y_i$, $z_i$, $u_i$ are introduced in \eqref{eq:N.1}, \eqref{eq:N.2}, and
{
%\scriptsize
\begin{align*}
   A_{ij} := 4\cos^2{\left(\frac{\theta_{i}}{2} + \frac{\pi}{4}ij(j-1) + \frac{\pi j}{2} \right)} \cos^2{\left(\frac{\theta_{i-1}}{2} + \frac{\pi}{4}(i-1)j(j-1) + \frac{\pi j}{2}\right)},
   %\quad i, j =1, \ldots, 4.
   \label{eq:M.2}\tag{13i}
\end{align*}
}for $i, j =1, \ldots, 4$.
\end{proposition}

%This result gives a system of \emph{algebraic} equations on the parameters $u$, $x_i$, $y_i$, and $z_i$ of a polyhedron with prescribed dihedral angles $\theta_i$ and flat angles $\delta_i$ of the central face. Such a characterization leads to the efficient construction of polyhedra of elliptic equimodular type.
%%existence criterion of a polyhedron of elliptic equimodular type with prescribed dihedral angles and flat angles of the central face.

We see that this indeed is a system of algebraic equations in $u$, $x_i$, $y_i$, and $z_i$, and just one inequality $u<1$, if $\theta_i$ and $\delta_i$ are prescribed.

%\begin{remark}\label{main_th:rem1}
All equations in the system have a simple geometric meaning.
Equation \eqref{eq:M.1.0} is just the condition on the sum of %interior
angles of the planar quadrilateral $A_1A_2A_3A_4$. %, see Figure~\ref{fig:pic002}.
Equations \eqref{eq:M.1.a}--\eqref{eq:M.1.e} are just conditions \eqref{eq:N.3} and \eqref{eq:N.4}.
The set of four equations \eqref{eq:M.1.f} is a relation between the angles $\delta_i$ and other parameters of the net introduced in \eqref{eq:N.1} and \eqref{eq:N.2}. The set of four equations \eqref{eq:M.1.g} is the relation between the flat and dihedral angles, which is equivalent to Bricard's equations \cite[Eq.~(19)]{izmestiev-2017} for %$\theta_i\ne 0$ and
$u<1$.
%, see \eqref{def:D.1} and Remark~\ref{plan_pr:lem5:rem1}.
%\end{remark} 

% \section{Proof of the Existence Criterion \texorpdfstring{(Proposition~\ref{main_th:prop1})}{(Proposition~1)}}
\section{Proofs}
%\section{Proof of the Main Results}
\label{sec:prexcrit}

We first prove Theorem~\ref{main_th:th1} and then Proposition~\ref{main_th:prop1}. The former requires three simple Lemmas~\ref{plan_pr:lem1}--\ref{plan_pr:lem2.6}, while the latter requires Lemmas~\ref{plan_pr:lem2.5}--\ref{plan_pr:prop2}. Henceforth, we use the notation provided in \eqref{eq:N.1}--\eqref{eq:N.9}. % and \eqref{eq:M.2}.
%The proof of the existence criterion (Proposition~\ref{main_th:prop1}) consists of the three main steps (Propositions~\ref{plan_pr:prop1}--\ref{plan_pr:prop2}), while the proof of Theorem~\ref{main_th:th1} uses Lemmas~\ref{plan_pr:lem1}--\ref{plan_pr:lem2}.

\begin{lemma}\label{plan_pr:lem1}
If $a_i, b_i, c_i, d_i$ are well-defined (that is, $\sin \overline{\alpha}_i$, $\sin \overline{\beta}_i$, $\sin \overline{\gamma}_i$, $\sin \overline{\delta}_i\neq 0$; see \eqref{eq:N.2}) then
{
\scriptsize
    \begin{align*}
    \left.
    \begin{aligned}
    &1 - a_ib_i = \frac{\sin{\sigma}_i\sin{(\overline{\alpha}_i - \beta_i)}}{\sin\overline{\alpha}_i\sin\overline{\beta}_i},
    \qquad
    1 - b_ic_i =  \frac{\sin{\sigma}_i\sin{(\overline{\gamma}_i - \beta_i)}}{\sin\overline{\beta}_i\sin\overline{\gamma}_i},
    \qquad
    1 - b_id_i = \frac{\sin{\sigma}_i\sin{(\overline{\delta}_i - \beta_i)}}{\sin\overline{\beta}_i\sin\overline{\delta}_i},
    \\
    &c_id_i - 1 = \frac{\sin{\sigma}_i\sin{(\overline{\alpha}_i - \beta_i)}}{\sin\overline{\gamma}_i\sin\overline{\delta}_i},
    \qquad
    a_id_i - 1 = \frac{\sin{\sigma}_i\sin{(\overline{\gamma}_i - \beta_i)}}{\sin\overline{\alpha}_i\sin\overline{\delta}_i},
    \qquad
    a_ic_i - 1 = \frac{\sin{\sigma}_i\sin{(\overline{\delta}_i - \beta_i)}}{\sin\overline{\alpha}_i\sin\overline{\gamma}_i},
    \\
    &1 - M_i =  \frac{\sin{\sigma_i} \sin(\overline{\alpha}_i - {\beta}_i)\sin(\overline{\gamma}_i - {\beta}_i)\sin(\overline{\delta}_i -\beta_i)}{\sin \overline{\alpha}_i\sin \overline{\beta}_i\sin \overline{\gamma}_i\sin \overline{\delta}_i}.
    \end{aligned}
    \right.
    \end{align*}
}
\end{lemma}

\begin{proof}
The equalities are verified in Section~1 of the supplementary material (\texttt{Criterion/helper.nb})~\cite{Nurmatov2025Repo}.
\end{proof}

\begin{lemma}\label{plan_pr:lem2}
%Assume that $\alpha_i$, $\beta_i$, $\gamma_i$, $\delta_i \in (0, \pi)$. Then condition \eqref{eq:N.0} holds if and only if $a_i$, $b_i$, $c_i$, $d_i$ are well-defined and
If $\alpha_i$, $\beta_i$, $\gamma_i$, $\delta_i \in (0, \pi)$ satisfy~\eqref{eq:N.0} then $a_i$, $b_i$, $c_i$, $d_i$ are well-defined and
\[
 a_ib_i, \ a_i c_i, \ a_i d_i, \ b_i c_i, \ b_i d_i, \ c_i d_i, \ M_i \neq 0, 1.
\]
\end{lemma}

\begin{proof}
%($\Rightarrow$)
Condition~\eqref{eq:N.0} implies $\sin \overline{\alpha}_i$, $\sin \overline{\beta}_i$, $\sin \overline{\gamma}_i$, $\sin \overline{\delta}_i\neq 0$. Thus $a_i$, $b_i$, $c_i$, $d_i$ are well-defined. Then the right sides of equations in Lemma~\ref{plan_pr:lem1} do not vanish because $\sigma_i$, $\overline{\alpha}_i - {\beta}_i$, $\overline{\gamma}_i - {\beta}_i$, $\overline{\delta}_i -\beta_i \not\in \pi\mathbb{Z}$ by~\eqref{eq:N.0}. Since $\alpha_i$, $\beta_i$, $\gamma_i$, $\delta_i \in (0, \pi)$ it follows that $a_i b_i$, $a_i c_i$, $a_i d_i$, $b_i c_i$, $b_i d_i$, $c_i d_i$, $M_i \neq 0$, see \eqref{eq:N.2}.
%
%($\Leftarrow$) Since $a_i$, $b_i$, $c_i$, $d_i$ are well-defined then $\sin \overline{\alpha}_i$, $\sin \overline{\beta}_i$, $\sin \overline{\gamma}_i$, $\sin \overline{\delta}_i\neq 0$. Thus $\overline{\alpha}_i$, $\overline{\beta}_i$, $\overline{\gamma}_i$, $\overline{\delta}_i \notin \pi\mathbb{Z}$. In addition, since $M_i \neq 1$ it follows that $\sigma_i$, $\overline{\alpha}_i - {\beta}_i$, $\overline{\gamma}_i - {\beta}_i$, $\overline{\delta}_i -\beta_i \not\in \pi\mathbb{Z}$. Hence condition~\eqref{eq:N.0} is fulfilled.
\end{proof}

\begin{remark}\label{plan_pr:lem2:rem1}
    %We have proved that
    Conversely, one can show that condition \eqref{eq:N.0} holds if and only if $a_i$, $b_i$, $c_i$, $d_i$ are well-defined and $M_i \neq 1$.
\end{remark}

\begin{lemma}\label{plan_pr:lem2.6}
Replacement of numbers $\alpha_i$, $\beta_i$, $\gamma_i$, $\delta_i$ with their complements to~$\pi$ for all $i=1,\ldots, 4$ preserves %the values
$M_i$, $r_i$, $s_i$, $f_i$ and conditions~\eqref{eq:N.0}, \eqref{eq:N.3}--\eqref{eq:N.5}.
%%%
    %Assume that numbers $\alpha_i$, $\beta_i$, $\gamma_i$, $\delta_i \in (0, \pi)$ satisfy equimodular ellipticity, that is, condition \eqref{eq:N.0}, for each $i=1,\ldots, 4$, and conditions \eqref{eq:N.3}--\eqref{eq:N.5} are satisfied. Then changing all the numbers $\alpha_i$, $\beta_i$, $\gamma_i$, $\delta_i$ by their complements to $\pi$ preserves equimodular ellipticity and leaves the values of $M_i$, $r_i$, $s_i$, and $f_i$ invariant.
\end{lemma}

\begin{proof}
    First, for new angles $\alpha_i^\star := \pi-\alpha_i$, $\beta_i^\star := \pi -\beta_i$, $\gamma_i^\star := \pi - \gamma_i$, $\delta_i^\star := \pi - \delta_i$, we have $\alpha_i^\star \pm \beta_i^\star \pm \gamma_i^\star \pm \delta_i^\star  = -\alpha_i \mp \beta_i \mp \gamma_i \mp \delta_i%\neq 0
    \, (\operatorname{mod}{2\pi})$. Thus, condition
    \eqref{eq:N.0} is preserved. %satisfied for all $i=1, \ldots, 4$.

    Second, since $\sigma_i^\star = 2\pi - \sigma_i$, $\overline{\alpha}_i^\star = \pi - \overline{\alpha}_i$, $\overline{\beta}_i^\star = \pi - \overline{\beta}_i$, $\overline{\gamma}_i^\star = \pi - \overline{\gamma}_i$, and $\overline{\delta}_i^\star = \pi - \overline{\delta}_i$, by \eqref{eq:N.2} it follows that new $r_i^\star = \frac{\sin \alpha_i^\star \sin \delta_i^\star}{\sin \overline{\alpha}_i^\star \sin \overline{\delta}_i^\star} =  \frac{\sin \alpha_i \sin \delta_i}{\sin \overline{\alpha}_i \sin \overline{\delta}_i} = r_i$. % for each $i=1,\ldots, 4$.
    Similarly, %for each $i=1,\ldots, 4$ we have
    $s_i^\star=s_i$, $f_i^\star=f_i$, $M_i^\star=M_i$ and hence %conditions
    \eqref{eq:N.3} and \eqref{eq:N.4} are preserved. %satisfied.

    Third, since $\sigma_i^\star = 2\pi - \sigma_i$, $r_i^\star = r_i$, $s_i^\star = s_i$, $f_i^\star = f_i$, and $M_i^\star = M_i$, it follows that %we have either
    $t_i^\star - t_i = \pm 2K$ %or $t_i^\star - t_i = -2K$
    for all $i=1,\ldots, 4$ by \eqref{eq:N.9}, Table~\ref{tab:tpqsigma}, and the $2K$-periodicity of the delta amplitude $\mathrm{dn}(t)$. Since $\pm 2K \pm 2K\pm2K \pm 2K = 0 \, (\operatorname{mod}{4K})$ for any choice of signs in $\pm$, it follows that %condition
    \eqref{eq:N.5} is %satisfied. Thus equimodular ellipticity is
    preserved.
\end{proof}

\begin{remark}\label{rem:half-symmetry}
    Lemma~\ref{plan_pr:lem2.6} remains true (with the same proof) if the
    replacement is performed %numbers $\alpha_i$, $\beta_i$, $\gamma_i$, $\delta_i$ are replaced
    for $i=2$ and $i=4$ only.
\end{remark}

\begin{proof}[Proof of Theorem~\ref{main_th:th1}]
    First, any QS-net satisfies~\eqref{eq:M.0.a}. Indeed, by~\eqref{eq:QS} we must have $\delta_1={\pi}/{2}$ because $\delta_2+\delta_3 = \pi$, $\delta_1=\delta_4$, and $\delta_1 + \delta_2 + \delta_3 + \delta_4 = 2\pi$. By Lemmas~2.1 and 4.1 from \cite{izmestiev-2017}, we have $\alpha_1, \beta_1, \gamma_1, \overline{\alpha}_1, \overline{\beta}_1, \overline{\gamma}_1, \overline{\delta}_1\in(0, \pi)$.
    Let us prove that any elliptic QS-net has equimodular elliptic type.
    Since the net is elliptic, it follows that \eqref{eq:N.0} is satisfied for all $i=1, \ldots, 4$.
    %the `moreover' part.
    %Since $\alpha_1$, $\beta_1$, $\gamma_1$, $\delta_1$ satisfy condition \eqref{eq:N.0}, it follows that their permutations  $\alpha_2$, $\beta_2$, $\gamma_2$, $\delta_2$ and $\alpha_4$, $\beta_4$, $\gamma_4$, $\delta_4$ do. Also, $\alpha_3 \pm \beta_3 \pm \gamma_3 \pm \delta_3  = -\alpha_1 \mp \beta_1 \mp \gamma_1 \mp \delta_1\neq 0 \, (\operatorname{mod}{2\pi})$. Thus condition \eqref{eq:N.0} is satisfied for all $i=1, \ldots, 4$.

    Directly from \eqref{eq:N.2}, we get $M_1 = M_2 = M_3 = M_4=:M$, $r_1 = r_2 = r_3 = r_4=:r$, $s_1 = s_4$, and $s_2= s_3$. % for an arbitrary QS-net.
    Thus, conditions \eqref{eq:N.3} and \eqref{eq:N.4} are satisfied.

    Let us verify condition \eqref{eq:N.5}. Observe that $\sigma_1 = \sigma_2 = \sigma_4$ and $\sigma_3 = 2\pi - \sigma_1$.
    Since $\overline{\alpha}_1$, $\overline{\beta}_1$, $\overline{\gamma}_1$, $\overline{\delta}_1\in(0, \pi)$, it follows that their permutations $\overline{\alpha}_2$, $\overline{\beta}_2$, $\overline{\gamma}_2$, $\overline{\delta}_2\in(0, \pi)$ and $\overline{\alpha}_4$, $\overline{\beta}_4$, $\overline{\gamma}_4$, $\overline{\delta}_4\in(0, \pi)$. Since $\alpha_3 = \pi -\delta_1$ and $\sigma_3 = 2\pi -\sigma_1$, it follows that $\overline{\alpha}_3 = \pi - \overline{\delta}_1$ and hence $\overline{\alpha}_3 \in (0, \pi)$. Similarly, $\overline{\beta}_3$, $\overline{\gamma}_3$, $\overline{\delta}_3\in(0, \pi)$.

    Thus, $r_i$, $s_i$, $f_i$, $M_i\neq 1$ and are strictly positive for each $i=1, \ldots, 4$ by %~\eqref{eq:M.0.a},
    Lemma~\ref{plan_pr:lem2} and \eqref{eq:N.2}. %Next, observe that $\sigma_1 = \sigma_2 = \sigma_4$ and $\sigma_3 = 2\pi - \sigma_1$.
    By Lemma~\ref{plan_pr:lem1} and \eqref{eq:QS}, we have
    $$
    (s_1 - 1)(s_2 - 1)
    =\tfrac{\sin{\sigma}_1\sin{(\overline{\alpha}_1 - \beta_1)}}{\sin\overline{\gamma}_1\sin\overline{\delta}_1}\cdot
    \tfrac{\sin{\sigma}_2\sin{(\overline{\alpha}_2 - \beta_2)}}{\sin\overline{\gamma}_2\sin\overline{\delta}_2}
    =-\tfrac{\sin^2{\sigma}_1\sin^2(\overline{\alpha}_1 - \beta_1)}
    {\sin \overline{\alpha}_1\sin\overline{\beta}_1 \sin\overline{\gamma}_1\sin\overline{\delta}_1}
    <0.
    $$

    Without loss of generality, we may assume that $s_1 > 1>s_2$ and $\sigma_1<\pi$. Indeed, the case $\sigma_1>\pi$ is reduced to $\sigma_1<\pi$ using Lemma~\ref{plan_pr:lem2.6}, and the case $s_1 < 1<s_2$ is reduced to $s_1 > 1>s_2$ using a cyclic permutation of indices and Remark~\ref{rem:half-symmetry}.

    Now, by \eqref{eq:N.9}, we have $t_1 = t_4$ because $f_1 = f_4 = \frac{\sin\alpha_1\sin\gamma_1}{\sin \overline{\alpha}_1 \sin \overline{\gamma}_1} > 0$ and the delta amplitude $\operatorname{dn}(t)$ is injective on $(mK, mK + \mathbf{i}K')$ for each $m=0,\ldots, 3$. Since $f_2 = \frac{\sin\beta_1}{\sin \overline{\beta}_1 \sin \overline{\delta}_1}$ and $M = f_1f_2$, it follows that $$
    \operatorname{dn}{\left(t_2 - K\right)}
    = \frac{k'}{\operatorname{dn}{t_2}} =
    \begin{cases}
        \sqrt{{M}/{f_2}} = \sqrt{f_1} = \operatorname{dn}{t_1},
        &\text{for $M < 1$},\\
        \sqrt{{f_2}/{M}} = 1/\sqrt{{f_1}} =\operatorname{dn}{t_1},
        &\text{for $M > 1$}.
    \end{cases}
    $$
    Thus $\operatorname{dn}{\left(t_2 - K\right)}  =  \operatorname{dn}{t_1}$ and hence $t_2 - K = t_1$ because, by Table~\ref{tab:tpqsigma}, they both belong to either $(0, \mathbf{i}K')$ (when $r>1$) or $(K, K + \mathbf{i}K')$ (when $r<1$). Similarly, $\operatorname{dn}{\left(t_3 + K\right)}  =  \operatorname{dn}{t_1}$, hence
    $t_3 -3 K = t_1$ (when $r>1$) and $t_3 + K = t_1$ (when $r<1$).  Hence $t_1 - t_2 - t_3 + t_4 = 0\, (\operatorname{mod}{4K})$, and \eqref{eq:N.5} holds.

    Conversely, assume that $\alpha_i$, $\beta_i$, $\gamma_i$, $\delta_i$, where $i=1,\ldots, 4$, satisfy~\eqref{eq:M.0.a}, \eqref{eq:N.0}, and \eqref{eq:QS}. Let us show that $\alpha_i$, $\beta_i$, $\gamma_i$, $\delta_i$ are flat angles of a flexible elliptic QS-net.

    First, by~\eqref{eq:M.0.a} and \eqref{eq:QS}, we have $\delta_1+\delta_2+\delta_3+\delta_4=2\pi$ and $M:=M_1>0$. By~\eqref{eq:N.0}, the denominators in~\eqref{eq:M.0.b} have finitely many roots.

    %First, since $\alpha_1$, $\beta_1$, $\gamma_1 \in (0, \pi)$ it obviously follows by \eqref{eq:QS} that $\alpha_i$, $\beta_i$, $\gamma_i$, $\delta_i\in (0, \pi)$. Also, since $\overline{\alpha}_1$, $\overline{\beta}_1$, $\overline{\gamma}_1$, $\overline{\delta}_1\in(0, \pi)$, it follows that their permutations $\overline{\alpha}_2$, $\overline{\beta}_2$, $\overline{\gamma}_2$, $\overline{\delta}_2\in(0, \pi)$ and $\overline{\alpha}_4$, $\overline{\beta}_4$, $\overline{\gamma}_4$, $\overline{\delta}_4\in(0, \pi)$. Since $\alpha_3 = \pi -\delta_1$ and $\sigma_3 = 2\pi -\sigma_1$, it follows that $\overline{\alpha}_3 = \pi - \overline{\delta}_1$ and hence $\overline{\alpha}_3 \in (0, \pi)$. Similarly, $\overline{\beta}_3$, $\overline{\gamma}_3$, $\overline{\delta}_3\in(0, \pi)$. Thus $\overline{\alpha}_i$, $\overline{\beta}_i$, $\overline{\gamma}_i$, $\overline{\delta}_i\in(0, \pi)$ for all $i=1, \ldots, 4$.

    %Second, since $\delta_1 + \delta_2 + \delta_3 + \delta_4 = 2\pi$ and $\alpha_i$, $\beta_i$, $\gamma_i$, $\delta_i$, $\overline{\alpha}_i$, $\overline{\beta}_i$, $\overline{\gamma}_i$, $\overline{\delta}_i\in(0, \pi)$ for all $i=1, \ldots, 4$, it follows by Lemmas~2.1 and 4.1 from \cite{izmestiev-2017} that there exists a polyhedron with flat angles $\alpha_i$, $\beta_i$, $\gamma_i$, $\delta_i$, where $i=1,\ldots, 4$.

    %Next,
    Observe that there is always a non-empty interval for $t$ such that $D(t)>0$ in \eqref{eq:M.0.b}. Indeed, since $M=s_1s_2$, it follows by Lemma~\ref{plan_pr:lem1} that
    \begin{align*}
         D(t)=
         \sin\overline{\alpha}_1 \sin\overline{\beta}_1 \sin\overline{\gamma}_1 \sin\overline{\delta}_1\left(1 - s_2 + t^2\right)\left(\frac{1-M}{1-s_2}t^2 + 1\right).
    \end{align*}
    This expression is positive, for example, for $t\in \mathbb{R}$ (when $s_2<1$ and $M<1$); for $t \in \left(\sqrt{s_2-1}, \sqrt{\frac{s_2-1}{1-M}}\right)$ (when $s_2>1$ and $M<1$); for $t \in \left(\sqrt{s_2-1}, +\infty\right)$ (when $s_2>1$ and $M>1$); and for $t \in \left(0, \sqrt{\frac{1-s_2}{M-1}}\right)$ (when $s_2<1$ and $M>1$).

    For each $t$ in those intervals, expressions~\eqref{eq:M.0.b} satisfy Bricard's equations \cite[Eq.~(20)]{izmestiev-2017}; for an automated verification, see Section~0 of the supplimentary material (\texttt{Criterion/helper.nb})~\cite{Nurmatov2025Repo}. Then
    %    By Bricard's equations \cite[Eq.~(20)]{izmestiev-2017} %(see \cite[Theorem~1.1]{stachel2010} and \cite[Lemma~4.3]{izmestiev-2017})
    %verified in~\cite{Nurmatov2025Repo} (\texttt{Criterion/helper.nb}) automatically, for each $t$ in those intervals,
    there exists a polyhedron with the flat angles $\alpha_i$, $\beta_i$, $\gamma_i$, $\delta_i$ and cotangents of dihedral half-angles~\eqref{eq:M.0.b} (the sufficiency of Bricard's equations and the condition $\delta_1+\delta_2+\delta_3+\delta_4=2\pi$ is
    well-known and reproved in Lemma~\ref{plan_pr:prop2} below in an equivalent form). In particular,
    any such polyhedron is flexible in $\mathbb{R}^3$ and~\eqref{eq:M.0.b} is a flexion. %The polyhedron
    It is an elliptic QS-net by \eqref{eq:N.0} and \eqref{eq:QS}.
\end{proof}

%%%%%%%%%%%%%%%%%%%%%%%%%%%%%%%

%Now we identify three major steps of the proof of the algebraic characterization %existence criterion
%(Proposition~\ref{main_th:prop1}):

For the proof of Proposition~\ref{main_th:prop1}, we need several inequalities on the parameters $M_i, r_i, s_i, f_i$, also useful for the implementation of our algorithm.

\begin{lemma}\label{plan_pr:lem2.5}
     %Assume that $\alpha_i$, $\beta_i$, $\gamma_i$, $\delta_i \in (0, \pi)$ satisfy condition \eqref{eq:N.0} and $M_i > 0$ for each $i=1,\ldots, 4$. Then $\overline{\alpha}_i$, $\overline{\beta}_i$, $\overline{\gamma}_i$, $\overline{\delta}_i\in(0, \pi)$ and
     If $\alpha_i$, $\beta_i$, $\gamma_i$, $\delta_i \in (0, \pi)$ satisfy condition \eqref{eq:N.0} and $M_i > 0$, then $\overline{\alpha}_i$, $\overline{\beta}_i$, $\overline{\gamma}_i$, $\overline{\delta}_i\in(0, \pi)$ and
     {
     \scriptsize
     \begin{align*}
     &r_i, \, s_i, \, f_i, \, (r_i-1)(r_i - M_i), \, (s_i- 1)(s_i - M_i), \, (f_i - 1)(f_i-M_i), \, (r_i-1)(s_i-1)(f_i-1)(1-M_i) >0.
     \end{align*}
     }
     %for each $i=1, \ldots, 4$. In addition, if conditions \eqref{eq:N.3} and \eqref{eq:N.4} hold, then $(f_1 - 1)(f_2 - 1)(f_3 - 1)(f_4 - 1) > 0$.
\end{lemma}

\begin{proof}
    By Lemma~\ref{plan_pr:lem2}, all the parameters in~\eqref{eq:N.2} are well-defined and we have $r_i$, $s_i$, $f_i$, $M_i\neq 0, 1$. Using \eqref{eq:N.2} and Lemma~\ref{plan_pr:lem1}, we find the product $(r_i - 1)(1 - M_i/r_i) = M_i\sin^2{\sigma}_i\sin^2{(\overline{\gamma}_i - \beta_i)} / (\sin{\alpha_i}\sin{\beta_i}\sin{\gamma_i}\sin{\delta_i})$. Hence, using \eqref{eq:N.0} we get $M_ir_i(r_i-1)(r_i - M_i)>0$. Thus, since $M_i>0$, we get $r_i \in \left(0, \min{\{1; M_i\}}\right) \cup \left(\max{\{1; M_i\}}, +\infty\right)$. Similarly, $s_i$, $f_i \in \left(0, \min{\{1; M_i\}}\right) \cup \left(\max{\{1; M_i\}}, +\infty\right)$.

    %Now, s
    Since $r_i$, $s_i$, $f_i$, $M_i > 0$, it follows that $\sign{a_i} = \sign{b_i} = \sign{c_i} =\sign{d_i}$. Then by %definition of $a_i$, $b_i$, $c_i$, $d_i$, see
    \eqref{eq:N.2}, we have either $\sin{\overline{\alpha}_i}$, $\sin{\overline{\beta}_i}$, $\sin{\overline{\gamma}_i}$, $\sin{\overline{\delta}_i} > 0$ or $\sin{\overline{\alpha}_i}$, $\sin{\overline{\beta}_i}$, $\sin{\overline{\gamma}_i}$, $\sin{\overline{\delta}_i} < 0$. Let us show that the latter is impossible. Indeed, otherwise $\overline{\alpha}_i$, $\overline{\beta}_i$, $\overline{\gamma}_i$, $\overline{\delta}_i \notin (0, \pi)$. Since $0 < \sigma_i < 2\pi$, it follows that $\overline{\alpha}_i$, $\overline{\beta}_i$, $\overline{\gamma}_i$, $\overline{\delta}_i \in (-\pi, 0)\cup (\pi, 2\pi)$. If a pair of angles, say, $\overline{\alpha}_i$ and $\overline{\beta}_i$, belong to different intervals among $(-\pi, 0)$ and $(\pi, 2\pi)$ then $|\overline{\alpha}_i - \overline{\beta}_i| = |{\alpha}_i - {\beta}_i|>\pi$, contradicting ${\alpha}_i$, ${\beta}_i \in (0, \pi)$. If all $\overline{\alpha}_i$, $\overline{\beta}_i$, $\overline{\gamma}_i$, $\overline{\delta}_i$ belong to the same interval, then $-2\pi < \sigma_i < 0$ or $2\pi < \sigma_i < 4\pi$ which contradicts $0 < \sigma_i < 2\pi$. Thus $\sin{\overline{\alpha}_i}$, $\sin{\overline{\beta}_i}$, $\sin{\overline{\gamma}_i}$, $\sin{\overline{\delta}_i} > 0$, which implies $\overline{\alpha}_i$, $\overline{\beta}_i$, $\overline{\gamma}_i$, $\overline{\delta}_i \in (0, \pi)$.

    Finally, since $\overline{\alpha}_i$, $\overline{\beta}_i$, $\overline{\gamma}_i$, $\overline{\delta}_i\in(0, \pi)$, it follows by Lemmas~\ref{plan_pr:lem1}--\ref{plan_pr:lem2} that
    {
    \footnotesize
    \begin{align*}
       &(r_i-1)(s_i-1)(f_i-1)(1-M_i) = \frac{\sin^4{\sigma_i} \sin^2(\overline{\alpha}_i - {\beta}_i)\sin^2(\overline{\gamma}_i - {\beta}_i)\sin^2(\overline{\delta}_i -\beta_i)}{\sin^3{\overline{\alpha}_i}\sin{\overline{\beta}_i}\sin^3{\overline{\gamma}_i}\sin^3{\overline{\delta}_i}} > 0.
    \end{align*}
    }
    %Since $\prod\limits_{i=1}^4 (r_i-1)(s_i-1)(f_i-1)(1-M_i) > 0$ then by \eqref{eq:N.3} and \eqref{eq:N.4} we get $(f_1 - 1)(f_2 - 1)(f_3 - 1)(f_4 - 1) > 0$.
\end{proof}

\begin{remark} If~\eqref{eq:N.3} and \eqref{eq:N.4} hold, then, multiplying the latter inequalities for all $i=1,\dots,4$, we get $(f_1 - 1)(f_2 - 1)(f_3 - 1)(f_4 - 1) > 0$.
\end{remark}

As a corollary, we deduce the existence of a polyhedron just from the inequalities $M_i > 0$ and the condition $\delta_1 + \delta_2 + \delta_3 + \delta_4 = 2\pi$.

\begin{lemma}\label{plan_pr:lem3}
    Assume that $\alpha_i$, $\beta_i$, $\gamma_i$, $\delta_i \in (0, \pi)$ satisfy condition \eqref{eq:N.0} for each $i=1,\ldots, 4$. Then a polyhedron with flat angles $\alpha_i$, $\beta_i$, $\gamma_i$, $\delta_i$ exists if and only if $\delta_1 + \delta_2 + \delta_3 + \delta_4 = 2\pi$ and $M_i > 0$ for each $i=1, \ldots, 4$.
\end{lemma}

\begin{proof}
    ($\Rightarrow$) If a polyhedron exists, then $\delta_1 + \delta_2 + \delta_3 + \delta_4 = 2\pi$ and $\overline{\alpha}_i$, $\overline{\beta}_i$, $\overline{\gamma}_i$, $\overline{\delta}_i \in (0, \pi)$ by Lemmas 2.1 and 4.1 from \cite{izmestiev-2017}. Thus $M_i > 0$ by %definition. %; see
    \eqref{eq:N.2}.

    ($\Leftarrow$) By Lemma~\ref{plan_pr:lem2.5} we have $\overline{\alpha}_i$, $\overline{\beta}_i$, $\overline{\gamma}_i$, $\overline{\delta}_i \in (0, \pi)$. Since %$\overline{\alpha}_i$, $\overline{\beta}_i$, $\overline{\gamma}_i$, $\overline{\delta}_i$,
    $\alpha_i, \beta_i, \gamma_i, \delta_i \in (0, \pi)$ as well and $\delta_1 + \delta_2 + \delta_3 + \delta_4 = 2\pi$, by Lemmas 2.1 and 4.1 from \cite{izmestiev-2017} there exists a polyhedron with flat angles $\alpha_i$, $\beta_i$, $\gamma_i$, $\delta_i$, where $i=1,\ldots, 4$.
\end{proof}

The flat angles are expressed in terms of %the parameters~
$x_i,y_i,z_i,u$ as follows.

\begin{lemma}\label{plan_pr:prop1}
Assume that $\alpha_i$, $\beta_i$, $\gamma_i$, $\delta_i \in (0, \pi)$ satisfy condition~\eqref{eq:N.0} and $M_i > 0$ for each $i=1,\ldots, 4$. Then
%For any elliptic polyhedron with flat angles $\alpha_i$, $\beta_i$, $\gamma_i$, $\delta_i\in (0, \pi)$, where $i=1,\ldots, 4$, we have
{
\footnotesize
\begin{align*}
    &\cos{\alpha_i} =  \varepsilon_i\frac{1 - y_iz_iu_i + x_iz_iu_i - x_iy_iu_i}{2\sqrt{x_iz_iu_i(1+y_i)(1+u_iy_i)}},
    \qquad
    \cos{\gamma_i} =  \varepsilon_i\frac{1 + y_iz_iu_i - x_iz_iu_i - x_iy_iu_i}{2\sqrt{y_iz_iu_i(1+x_i)(1+u_ix_i)}},
    \\
    &\cos{\delta_i} =  \varepsilon_i\frac{1 - y_iz_iu_i - x_iz_iu_i + x_iy_iu_i}{2\sqrt{x_iy_iu_i(1+z_i)(1+u_iz_i)}},
    \qquad
    \cos{\sigma_i} = \frac{1 - u_i(x_iy_i + x_iz_i + y_iz_i + 2x_iy_iz_i)}{2\sqrt{x_iy_iz_iu_i^2(1+x_i)(1+y_i)(1+z_i)}},
    \\
    &\cos{\beta_i} =  \varepsilon_i\frac{u_i(1+x_i)(1+y_i)(1+z_i) + (1+u_ix_i)(1+u_iy_i)(1+u_iz_i) - u_ix_iy_iz_i(u_i - 1)^2}{2\sqrt{u_i(1+x_i)(1+y_i)(1+z_i)(1+u_ix_i)(1+u_iy_i)(1+u_iz_i)}}.
\end{align*}
}
\end{lemma}

\begin{proof}%[Proof of Lemma~\ref{plan_pr:prop1}]
Since $a_i$, $b_i$, $c_i$, $d_i$ and $r_i$, $s_i$, $f_i$, $M_i$, $M_i/r_i$, $M_i/s_i$, $M_i/f_i \neq 0, 1$ are well-defined (by Lemma~\ref{plan_pr:lem2}) it follows that $x_i$, $y_i$, $z_i \neq -1, 0$ and $u_i\neq 0, 1$ are well-defined, and $u_ix_i$, $u_iy_i$, $u_iz_i \neq -1$. Next, since $M_i>0$ it follows by Lemma~\ref{plan_pr:lem2.5} that $\overline{\alpha}_i$, $\overline{\beta}_i$, $\overline{\gamma}_i$, $\overline{\delta}_i\in(0, \pi)$. Finally, under the assumptions $\alpha_i$, $\beta_i$, $\gamma_i$, $\delta_i$, $\overline{\alpha}_i$, $\overline{\beta}_i$, $\overline{\gamma}_i$, $\overline{\delta}_i\in(0, \pi)$, all the desired formulae are verified in Section~2 of the supplementary material (\texttt{Criterion/helper.nb})~\cite{Nurmatov2025Repo}.
\end{proof}

Finally, we need the following characterization of possible flat and dihedral angles of an elliptic polyhedron (equivalent to Bricard's equations).

\begin{lemma}\label{plan_pr:prop2}
Assume that $\alpha_i$, $\beta_i$, $\gamma_i$, $\delta_i$, $\theta_i \in (0, \pi)$ satisfy condition \eqref{eq:N.0} for each $i=1,\ldots, 4$. Then a polyhedron with flat angles $\alpha_i$, $\beta_i$, $\gamma_i$, $\delta_i$ and dihedral angles $\theta_i$ exists if and only if $\delta_1 + \delta_2 + \delta_3 + \delta_4 = 2\pi$, $M_i > 0$, and equation \eqref{eq:M.1.g} holds for each $i = 1, \ldots, 4$.
\end{lemma}

% \begin{figure}[htbp!]
%   \centering
%       \centering
%     \includesvg[scale=0.5]{img/pic004b.svg}
%   \caption{A polyhedron with some of the angles prescribed.}
%   \label{fig:pic3}
% \end{figure}

\begin{figure}[htbp!]
        \centering
        \includegraphics[width=0.6\textwidth]{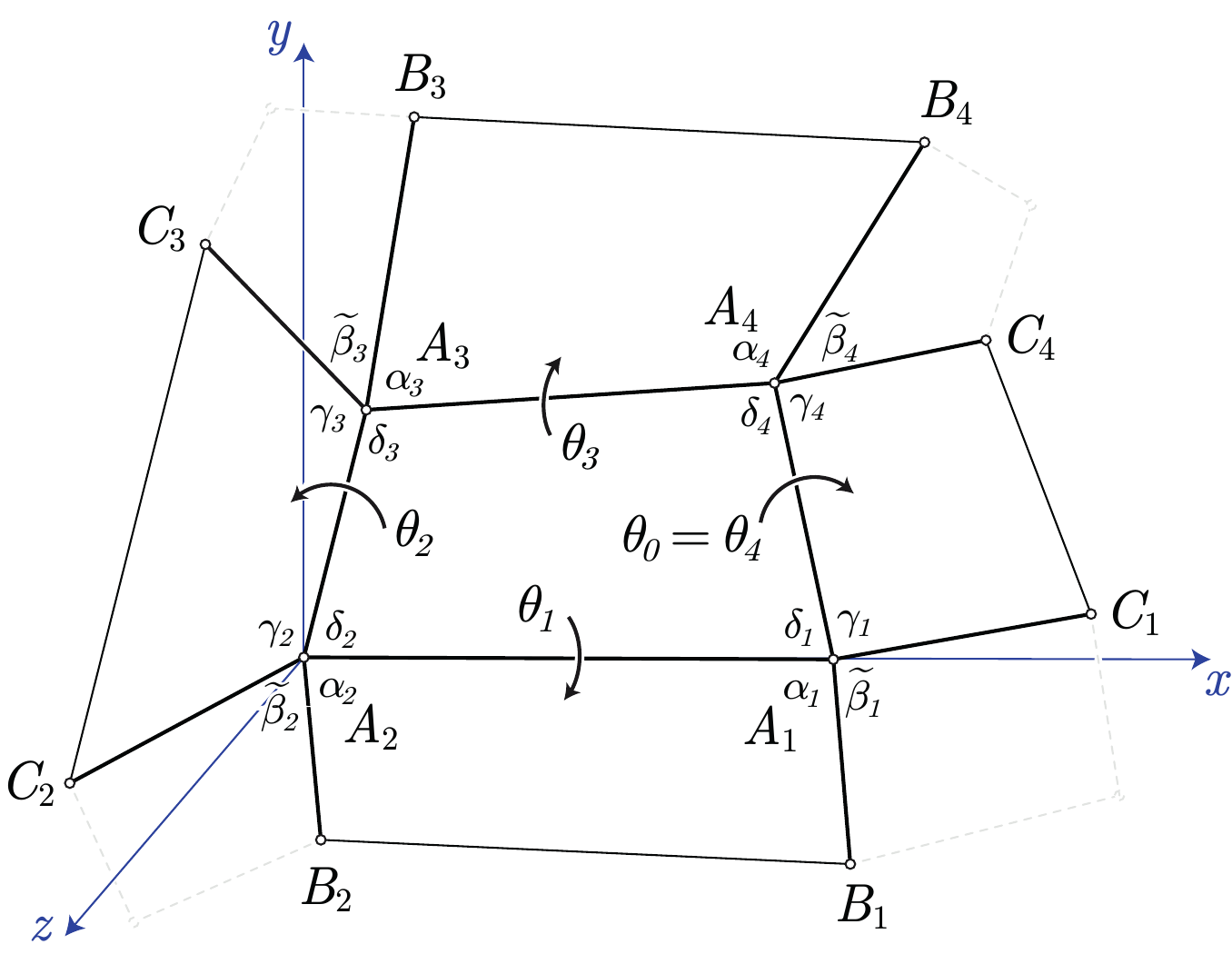}
  \caption{A polyhedron with some of the angles prescribed.}
  \label{fig:pic3}
\end{figure}

\begin{proof}%[Proof of Proposition~\ref{plan_pr:prop2}]
%By Lemma~\ref{plan_pr:lem3}, $\delta_1 + \delta_2 + \delta_3 + \delta_4 = 2\pi$ and $M_i>0$ are necessary conditions for the existence of the desired polyhedron. Thus we may assume that these conditions hold.
We may restrict ourselves to the values $\alpha_i,\beta_i,\gamma_i,\delta_i$ satisfying
$\delta_1 + \delta_2 + \delta_3 + \delta_4 = 2\pi$ and $M_i>0$ because by Lemma~\ref{plan_pr:lem3}, those conditions are necessary for the existence of a polyhedron. % with those flat angles.

Clearly, there exists a polyhedron with the given dihedral angles $\theta_i$ and flat angles $\alpha_i,\gamma_i,\delta_i$ (excluding $\beta_i$) for $i=1,\dots,4$. Denote by $\widetilde\beta_i$, where $i=1,\dots,4$, the remaining flat angles of this polyhedron; see Figure~\ref{fig:pic3}. It suffices to prove that $\widetilde\beta_i=\beta_i$ if and only if~\eqref{eq:M.1.g} holds.

This follows from the following chain of equalities:
\begin{multline*}
\frac{
A_{i1} + A_{i2}y_i z_i u_i + A_{i3}x_i z_i u_i + A_{i4}x_i y_i u_i
}{
2\sqrt{x_i y_i u_i (1 + z_i)(1 + u_i z_i)}
}
= \\ =
\varepsilon_i
\left(
\frac{
\cos\beta_i
- \cos\gamma_i\bigl(\cos\alpha_i\cos\delta_i + \cos\xi_i\sin\alpha_i\sin\delta_i\bigr)
}{
\sin\alpha_i\sin\gamma_i
}
-
\right.
\\[-0.3em]
\left.
- \frac{
\cos\eta_i\sin\gamma_i\bigl(\cos\alpha_i\sin\delta_i - \cos\xi_i\sin\alpha_i\cos\delta_i\bigr)
}{
\sin\alpha_i\sin\gamma_i
}
\right)
=
\\[0.5em]
=
\varepsilon_i\sin\theta_i\sin\theta_{i-1}
+ \varepsilon_i
\frac{
\cos\beta_i - \cos\widetilde{\beta}_i
}{
\sin\alpha_i\sin\gamma_i
}.
\end{multline*}
%}
Here $A_{ij}$ %for $i, j=1, \ldots, 4$
are defined in~\eqref{eq:M.2}, $\varepsilon_i = \sign{(\sin{\sigma_i})}\neq 0$, %see \eqref{eq:N.0}--\eqref{eq:N.1};
and we denote $\xi_{i} = \theta_{i}$, $\eta_i = \theta_{i-1}$ for $i=1, 3$ and $\xi_{i} = \theta_{i-1}$, $\eta_i = \theta_{i}$ for $i=2, 4$. By the assumption $M_i>0$ and Lemma~\ref{plan_pr:lem2.5}, we have
$\overline{\alpha}_i$, $\overline{\beta}_i$, $\overline{\gamma}_i$, $\overline{\delta}_i\in(0, \pi)$.
%The first equality in the chain holds by~\eqref{eq:N.0}, Lemma~\ref{plan_pr:lem1}, and \eqref{eq:N.2}. The second equality holds because $\overline{\beta}_i \in (0, \pi)$ and $\sin\overline{\beta}_i>0$ %$\sqrt{\sin^2\overline{\beta}_i} = \sin\overline{\beta}_i$% by the assumption $M_i>0$ and Lemma~\ref{plan_pr:lem2.5}. The third equality is checked in Section~3 of the supplementary material (\texttt{Criterion/helper.nb})~\cite{Nurmatov2025Repo}.
Under the assumptions $\alpha_i$, $\beta_i$, $\gamma_i$, $\delta_i$, $\overline{\alpha}_i$, $\overline{\beta}_i$, $\overline{\gamma}_i$, $\overline{\delta}_i\in(0, \pi)$, the first equality in the chain is checked in Section~3 of the supplementary material (\texttt{Criterion/helper.nb})~\cite{Nurmatov2025Repo}. To prove the second equality, assume that $i = 2$ without loss of generality. Introduce the Cartesian coordinate system with the origin at $A_2$ such that the face $A_2A_3A_4A_1$ lies in the half-plane $y\geq0$ of the $xy$-plane and the ray $A_2A_1$ is the $x$-axis as shown in Figure~\ref{fig:pic3}. The unit vectors $\vv{e_1}$ and $\vv{e_2}$ along ${A}_2{B}_2$ and ${A}_2{C}_2$ equal
$$
\vv{e_1} = \left(
\begin{matrix}
     \cos{\alpha_2} \\
     \sin{\alpha_2}\cos{\theta}_1 \\
     \sin{\alpha_2}\sin{\theta}_1
\end{matrix}
\right)
\qquad\text{and}\qquad
\vv{e_2} = \left(
\begin{matrix}
\cos{\gamma_2}\cos{\delta_2} + \sin{\gamma_2}\cos{\theta_2}\sin{\delta_2}\\ \cos{\gamma_2}\sin{\delta_2} - \sin{\gamma_2}\cos{\theta_2}\cos{\delta_2}\\ \sin{\gamma_2}\sin{\theta_2}
\end{matrix}
\right).
$$
%{
%\footnotesize
%\begin{align*}
%    &\vv{e_1} = \left(\cos{\alpha_2}, \sin{\alpha_2}\cos{\theta}_1, \sin{\alpha_2}\sin{\theta}_1 \right), \\
%    &\vv{e_2} = \left(\cos{\gamma_2}\cos{\delta_2} + \sin{\gamma_2}\cos{\theta_2}\sin{\delta_2}, \cos{\gamma_2}\sin{\delta_2} - \sin{\gamma_2}\cos{\theta_2}\cos{\delta_2}, \sin{\gamma_2}\sin{\theta_2}\right).
%\end{align*}
%}
Since $\vv{e_1} \cdot \vv{e_2} = \cos{\widetilde\beta_2}$, the last equality in the chain follows.
\end{proof}

\begin{proof}[Proof of Proposition~\ref{main_th:prop1}]
By Lemma~\ref{plan_pr:prop2}, a polyhedron with flat angles $\alpha_i$, $\beta_i$, $\gamma_i$, $\delta_i$ and dihedral angles $\theta_i$ exists if and only if conditions~\eqref{eq:M.1.0}, \eqref{eq:M.1.g}, and $u_i < 1$ hold for each $i=1,\ldots, 4$. The polyhedron has elliptic equimodular type if and only if conditions~\eqref{eq:M.1.a}--\eqref{eq:M.1.e} and \eqref{eq:N.5} hold for some $e_1,e_2,e_3\in\{\pm1\}$, because \eqref{eq:M.1.a}--\eqref{eq:M.1.e} are equivalent to \eqref{eq:N.3} and \eqref{eq:N.4}. For $u_i < 1$, equations~\eqref{eq:M.1.f} are automatic by Lemma~\ref{plan_pr:prop1}.
%%%%%%%%%% OLD VERSION %%%%%%%%%%%%%%%%%%%
%($\Rightarrow$) By Lemma~\ref{plan_pr:prop2} we have equations \eqref{eq:M.1.g} and by Proposition~\ref{plan_pr:prop1.5} we have all the remaining conditions.
%
%($\Leftarrow$) By~\eqref{eq:N.2} and Lemma~\ref{plan_pr:lem2.5} we have $x_iy_iu_i(1+z_i)(1+u_iz_i) > 0$. Thus denominators in \eqref{eq:M.1.f} and \eqref{eq:M.1.g} are well-defined. Since $\alpha_i$, $\beta_i$, $\gamma_i$, $\delta_i \in (0, \pi)$ satisfy conditions \eqref{eq:N.0}, $\delta_1 + \delta_2 + \delta_3 + \delta_4 = 2\pi$, $u_i < 1$, and \eqref{eq:M.1.g} it follows by Lemma~\ref{plan_pr:prop2} that there exists a polyhedron with flat angles $\alpha_i$, $\beta_i$, $\gamma_i$, $\delta_i$ and dihedral angles $\theta_i$, where $i=1,\ldots, 4$. By Proposition~\ref{plan_pr:prop1.5} this polyhedron is of equimodular elliptic type.
\end{proof} 

\section{Algorithms and Examples}
\label{sec:ex}
%\label{sec:alg}
%\input{006_algorithm}

%\section{Examples}
%\label{sec:ex}
%\input{007_examples}

We now present several constructions of polyhedra of equimodular elliptic type, illustrated with examples. We start with closed-form examples, which are particular cases of Theorem~\ref{main_th:th1}, and establish their paradoxical properties. Then we construct an essentially different closed-form example by leveraging our findings, such as Proposition~\ref{main_th:prop1} and the lemmas in Section~\ref{sec:prexcrit}. This construction also demonstrates how the concept of quasi-symmetry was discovered. Next, we present a numerical search algorithm based on Proposition~\ref{main_th:prop1} and generate several numerical examples this way. \tmpcomment Finally, we present a small-scale prototype. \endtmpcomment

%Now we present several examples of polyhedra of equimodular elliptic type. We start with a closed-form example, which is a particular case of Theorem~\ref{main_th:th1}, and establish its geometric properties. Next, we construct an essentially different closed-form example by leveraging our findings, such as Propositions~\ref{main_th:prop1}--\ref{plan_pr:prop1}. This construction also demonstrates how the concept of quasi-symmetry was discovered. Finally, we generate numerical examples following our algorithmic approach based on Proposition~\ref{main_th:prop1}.

%In the following, we present several examples. First, we construct a closed-form example directly using Theorem~\ref{main_th:th1}, demonstrating for the first time that such constructions are possible, as this has not been previously shown in the literature. Next, we construct an additional example by leveraging our findings (e.g., Propositions~\ref{main_th:prop1}--\ref{plan_pr:prop1}) introduced earlier in this manuscript. Finally, we generate numerical examples following our algorithmic approach.

\subsection{Closed-form Examples}

\bluenew{Before presenting the specific configurations, we briefly outline the purpose of each example in this subsection to clarify the distinction between our specific and general results. Example~\ref{ex:ex1} demonstrates the direct application of Theorem~\ref{main_th:th1}, yielding an explicitly flexible QS-net. Example~\ref{ex:ex1a} presents a paradoxical case showing that there exist two QS-nets with the same flat angles, one of which is flexible while the other is not. Finally, Example~\ref{ex:ex2} illustrates the broader applicability of our general algebraic characterization from Proposition~\ref{main_th:prop1}; it provides a closed-form flexible polyhedron that belongs to the general equimodular elliptic type but does not fall into the QS-net subclass.}

The most direct way to construct polyhedra of equimodular elliptic type is to apply Theorem~\ref{main_th:th1}. Let us discuss a specially chosen particular case.

\begin{example}[See Figure~\ref{exfig:pic001}]\label{ex:ex1}
Consider the polyhedron with flat angles
    \begin{align*}
    &\alpha_1 = \alpha_4 = \delta_2 = 180^\circ - \delta_3 = 105^\circ,
    &&\beta_1 = \beta_4 = \gamma_2 = 180^\circ -  \gamma_3 = 15^\circ,
    \\
    &\gamma_1 = \gamma_4 = \beta_2 = 180^\circ -  \beta_3 = 120^\circ,
    &&\delta_1 = \delta_4 = \alpha_2 =  \alpha_3 =  90^\circ,
\end{align*}
and the cotangents of dihedral half-angles
% {
% \small
% \begin{align*}\label{eq:E.1}\tag{E.1}
% \left.
% \begin{aligned}
%     &\cot{\frac{\theta_1}{2}}  =t,\\
%     &\cot{\frac{\theta_2}{2}}  = \frac{ 2\sqrt{2}\,t \mp\sqrt{2\sqrt{3} + 2}  \sqrt{\left( \sqrt{3}\, t^2 - 1 \right) \left(1 - (2\sqrt{3} + 3)\,t^2 \right) } }
% { (\sqrt{3} + 1)\,(1 - 3\,t^2) }, \\
%     &\cot{\frac{\theta_3}{2}}  = \frac{(17\sqrt{3} + 36)\, t \mp 2\sqrt{3\sqrt{3} + 3} \sqrt{\left( \sqrt{3}\,t^2 - 1 \right) \left( 1 - (2\sqrt{3} + 3)\,t^2 \right) } }
% { \sqrt{3} + 12\,(5\sqrt{3} + 9)\,t^2 }, \\
%     &\cot{\frac{\theta_4}{2}}  = \frac{(2\sqrt{3} + 3)\, t \mp\sqrt{\sqrt{3} + 1} \sqrt{\left(\sqrt{3}\,t^2 - 1\right) \left(1 - (2\sqrt{3} + 3)\,t^2\right)}}
% {\sqrt{3} + 1 + (2\sqrt{3} + 3)\, t^2},
% \end{aligned}
%     \right.
% \end{align*}
% }
{
\normalsize
\begin{align*}\label{eq:E.1}\tag{14}
\left.
\begin{aligned}
&\cot{\frac{\theta_1}{2}} = \frac{(\sqrt{2} - \sqrt{6})\, t \pm \sqrt{2 + 6 t^2 + 3 t^4}}{1 + \sqrt{3} + 3 t^2}, &
&\cot{\frac{\theta_2}{2}} = t,\\
&\cot{\frac{\theta_3}{2}}= \frac{(\sqrt{6}-\sqrt{2})\, t \pm \sqrt{2 + 6 t^2 + 3 t^4}}{1 + \sqrt{3} + 3 t^2}, &
&\cot{\frac{\theta_4}{2}} = \pm\frac{\sqrt{2 + 6 t^2 + 3 t^4}}{1 + \sqrt{3} + \sqrt{3} t^2},
\end{aligned}
    \right.
\end{align*}
}where $t$ is a parameter in the range $(-\infty, +\infty)$ and the signs in $\pm$ agree.
\end{example}

% \begin{figure}[htbp!]
%   \centering
%   \begin{subfigure}[t]{0.45\textwidth}
%     \centering
%     \includesvg[scale=0.5]{img/exact_ex1_1-2.svg}
%     \caption{$t=0.5$}%$t=\frac{1}{2}$}
%     \label{exfig:pic001a}
%   \end{subfigure}
%   \begin{subfigure}[t]{0.45\textwidth}
%     \centering
%     \includesvg[scale=0.5]{img/exact_ex1_1-1.svg}
%     \caption{$t=1$}%$t=\frac{3}{4}$}
%     \label{exfig:pic001b}
%   \end{subfigure}
%   \caption{The QS-net in Example~\ref{ex:ex1} at different values of parameter $t$. The sign ``$+$'' is chosen in each $\pm$ in~\eqref{eq:E.1}.}
%   %\tmpcomment\textbf{REMOVE THE LABELING OF VERTICES IN THIS AND THE NEXT FIGURES!}\endtmpcomment}
%   \label{exfig:pic001}
% \end{figure}

\begin{figure}[htbp!]
  \centering
  \begin{subfigure}[t]{0.475\textwidth}
    \centering
    \includegraphics[scale=0.1]{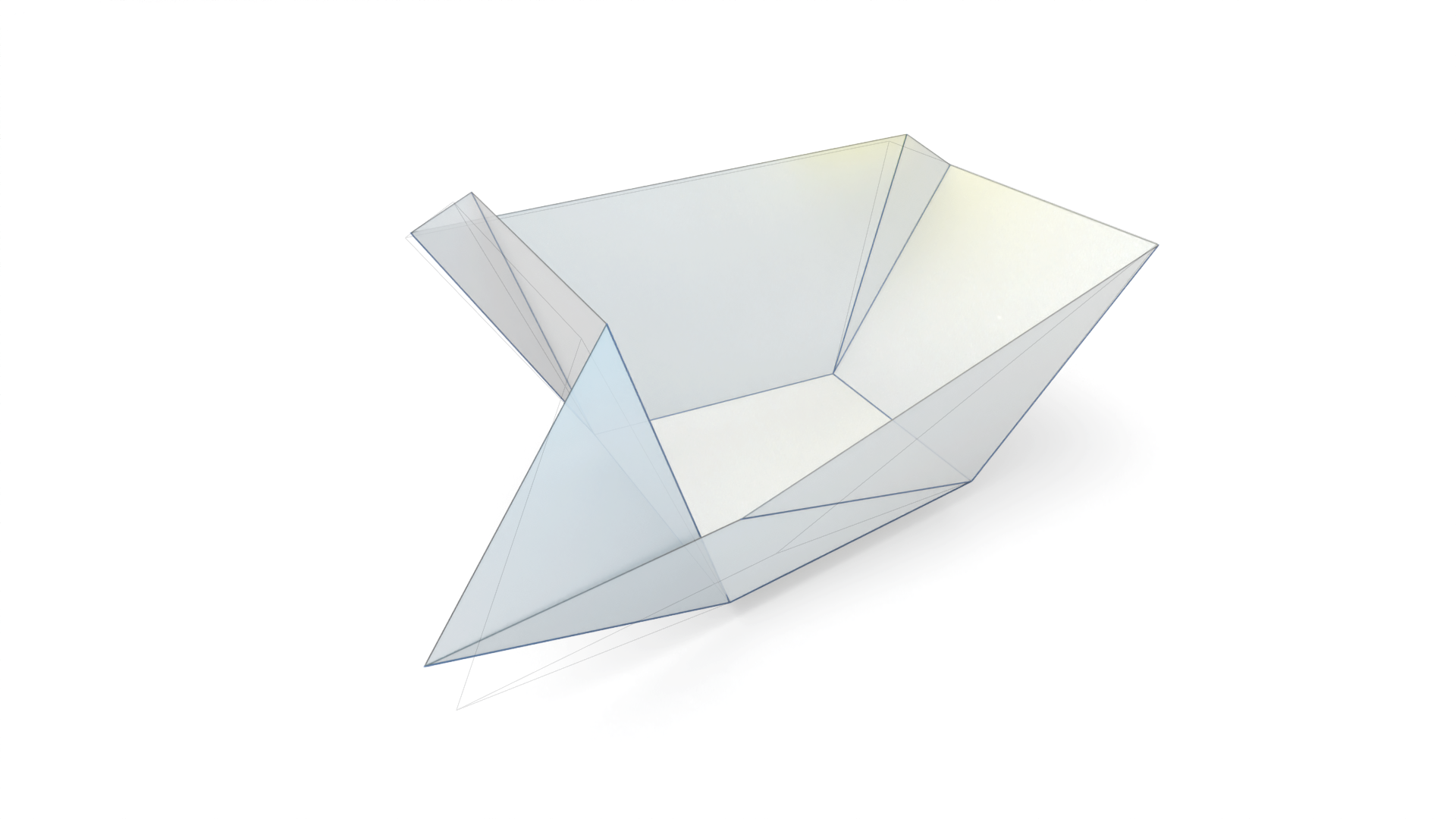}
    \caption{$t=0.5$}%$t=\frac{1}{2}$}
    \label{exfig:pic001a}
  \end{subfigure}
  \begin{subfigure}[t]{0.475\textwidth}
    \centering
    \includegraphics[scale=0.1]{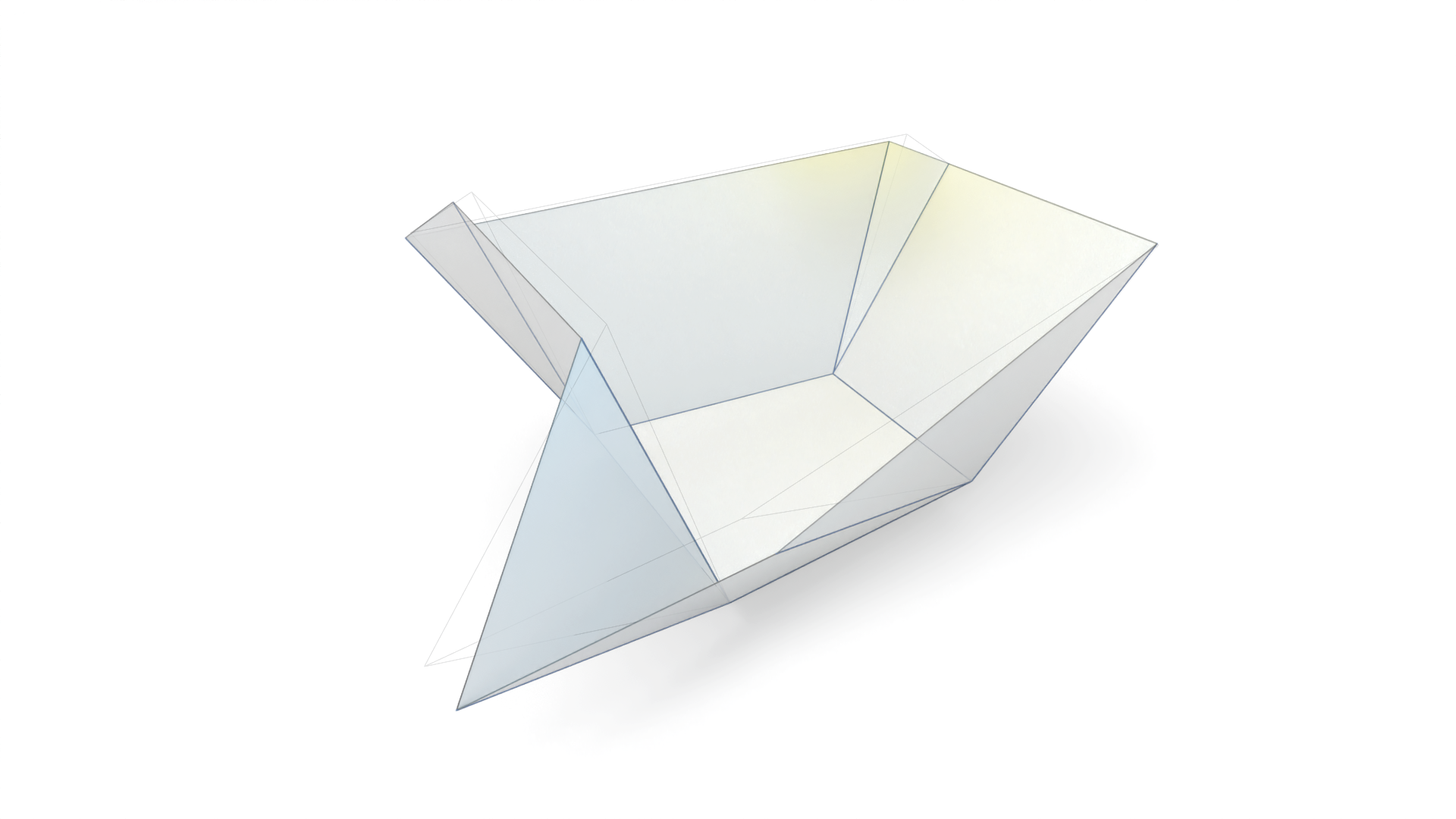}
    \caption{$t=1$}%$t=\frac{3}{4}$}
    \label{exfig:pic001b}
  \end{subfigure}
  \caption{The QS-net in Example~\ref{ex:ex1} at different values of parameter $t$. The sign ``$+$'' is chosen in each $\pm$ in~\eqref{eq:E.1}.}
  %\tmpcomment\textbf{REMOVE THE LABELING OF VERTICES IN THIS AND THE NEXT FIGURES!}\endtmpcomment}
  \label{exfig:pic001}
\end{figure}

%\begin{example}\label{ex:ex1anti}
%Consider the polyhedron with the same flat angles as in Example~\ref{ex:ex1} and the cotangents of dihedral half-angles $$\cot{\frac{\theta_1}{2}} =\cot{\frac{\theta_3}{2}} =0, \qquad\cot{\frac{\theta_2}{2}} =\sqrt{2}\cot{\frac{\theta_4}{2}} =\mathbf{i}\sqrt{\sqrt{3}-1}.$$
%\end{example}

\begin{proposition}\label{ex:prop1}
    The polyhedron in Example~\ref{ex:ex1} belongs to equimodular elliptic type, does not belong to any other classes in Izmestiev's classification (even after switching the boundary strips), is flexible in $\mathbb{R}^3$, and a flexion is given by \eqref{eq:E.1}.
    %The polyhedron in Example~\ref{ex:ex1anti} %has the same flat angles but
    %is not flexible (even in $\mathbb{C}^3$).
\end{proposition}

%\begin{corollary}\label{cor:paradox}
%    There are two polyhedra with the same flat angles, one of which is flexible (in $\mathbb{R}^3$) and the other one is not (even in $\mathbb{C}^3$).
%\end{corollary}

Hereafter, we refer to the classes introduced by Izmestiev \cite[Section~3]{izmestiev-2017}. We do not recall their definition here, because most of them are excluded immediately. \emph{Switching the right boundary strip} means replacement of $\beta_1$, $\beta_4$, $\gamma_1$, $\gamma_4$ with their complements to $180^\circ$ and changing $\theta_4$ to $\theta_4 - 180^\circ$, leaving the other flat and dihedral angles unchanged  \cite[Section 2.5]{izmestiev-2017}. Switching the left, top, and bottom boundary strips is defined analogously.

\begin{proof}[Proof of Proposition~\ref{ex:prop1}]
%In this computer-assisted proof, we use the supplementary
This proof is supplemented by a Mathematica notebook (\texttt{Example1/helper1.nb})~\cite{Nurmatov2025Repo}, which %computes all the quantities $r_i$, $s_i$, $f_i$, $\sigma_i$, $M_i$ and
verifies every computation below.

First, %since $\alpha_1 = 105^\circ$, $\beta_1 =15^\circ$, $\gamma_1 = 120^\circ$, $\delta_1 = 90^\circ\in(0, 180^\circ)$
we verify condition \eqref{eq:N.0} by considering all possible choices of signs in $\pm$. Next, observe that $\overline{\alpha}_1 = 60^\circ$, $\overline{\beta}_1 = 150^\circ$, $\overline{\gamma}_1 = 45^\circ$, $\overline{\delta}_1 = 75^\circ\in(0, 180^\circ)$. Thus by Theorem~\ref{main_th:th1} our example is %a QS-net and
%realizable in $\mathbb{R}^3$ and has equimodular elliptic type.
realizable and flexible in $\mathbb{R}^3$, with a flexion given by \eqref{eq:E.1}, and it has equimodular elliptic type.

%Next, by Lemma%s 2.2 and
%~4.3 from \cite{izmestiev-2017}, our example is flexible in $\mathbb{R}^3$ with the flexion given by \eqref{eq:E.1}, as verified in supplementary Mathematica notebook (\texttt{Example1/helper1.nb})~ \cite{Nurmatov2025Repo}.

Let us prove that this example belongs exclusively to this %the \textsf{equimodular elliptic}
type and does not fall into any of the other classes introduced by Izmestiev, %see Main Theorem in \cite{izmestiev-2017},
even after switching the boundary strips.

All the classes except for the orthodiagonal, conjugate-modular, linear compound, and trivial one are ruled out because they do not satisfy~\eqref{eq:N.0} for some $i$.

% The orthodiagonal class is ruled out because $\cos{\alpha_1}\cos{\gamma_1}\neq 0=\cos{\beta_1}\cos{\delta_1}$.
Since $\cos{\alpha_1}\cos{\gamma_1}\neq 0=\cos{\beta_1}\cos{\delta_1}$, it follows that the orthodiagonal class is ruled out. Switching the right boundary strip acts as replacement of $\beta_1$, $\beta_4$, $\gamma_1$, $\gamma_4$ with their complements to $180^\circ$ and thus does not change this condition. Switching any number of the other boundary strips neither does.

Our example does \emph{not} belong to the conjugate-modular class, because $M_i=\sqrt{3}-1 < 1$ for all $i = 1, \ldots, 4$ so that $1/M_i+1/M_{i+1}\ne 1$. Switching the boundary strips does not transform our example to conjugate-modular class because it does not change $M_i$.

Our example does \emph{not} belong to the trivial class, because none of the dihedral angles remains fixed during flexion by~\eqref{eq:E.1}. Switching any number of the boundary strips does not transform our example to trivial class.

It remains to rule out the linear compound class. Recall that this class is defined by the condition that a component of the complexified configuration space of the polyhedron satisfies
$\cot{\frac{\theta_1}{2}}/\cot{\frac{\theta_3}{2}}=\mathrm{const}$ (or $\cot{\frac{\theta_1}{2}}\cot{\frac{\theta_3}{2}}=\mathrm{const}$ after switching the boundary strips) \cite[Section~3.5]{izmestiev-2017}. Here, the \emph{complexified configuration space} is the set of complex solutions of the Bricard equations \cite[Eq.~(20)]{izmestiev-2017} in the variables $w_1:=\cot{\frac{\theta_1}{2}}$, $t:=\cot{\frac{\theta_2}{2}}$, $w_2:=\cot{\frac{\theta_3}{2}}$, $z:=\cot{\frac{\theta_4}{2}}$ (see \texttt{Example1/helper1.nb} in~\cite{Nurmatov2025Repo}):
% \textbf{Insert the Bricard equations for our example; use $t$ instead of $z$!}
%{
%\small
\begin{align*}
\left\{
\begin{aligned}
%4P_2(w_1, t) &=
(1 + \sqrt{3} + 3t^2)w_1^2  + 2 \sqrt{2} (\sqrt{3} - 1) t w_1 - t^2  - \sqrt{3} + 1 &= 0,
\\
%4P_3(w_2, t) &=
(1 + \sqrt{3} + 3t^2) w_2^2 - 2 \sqrt{2} (\sqrt{3} - 1) t w_2 - t^2  - \sqrt{3} +  1 &= 0, \\
%4 \sqrt{2}P_4(w_2, z) &=
( 2 \sqrt{3} + (3 + \sqrt{3}) z^2)w_2^2 - 2 (3 + \sqrt{3})zw_2  +  2 z^2 + \sqrt{3} -1 &= 0,
\\
%4 \sqrt{2}P_1(w_1, z) &=
( 2 \sqrt{3} + (3 + \sqrt{3}) z^2)w_1^2 - 2 (3 + \sqrt{3}) zw_1 + 2 z^2 + \sqrt{3} -1 &= 0.
\end{aligned}
\right.
\end{align*}
%}
(Here, we have multiplied Bricard's equations by suitable constants. Also, more accurately, one should consider their bihomogenized versions, but this does not affect the following argument.)

Consider the first pair of equations. %which is the second and third equations.
Solving them as quadratic equations in $w_1$ and $w_2$, %respectively,
we parametrize the solutions by the first three expressions in~\eqref{eq:E.1} in a neighborhood of any point $(w_1,t,w_2)$ such that $2+6t^2+3t^4\ne0$ and $1 + \sqrt{3} + 3t^2\ne 0$. Here, we choose any continuous branch of the square root in this neighborhood, and the signs in $\pm$ need \emph{not} agree. We conclude that one one-dimensional irreducible component of the solution set of the first pair of Bricard's equations satisfies $w_1/w_2=-1$, and no other one-dimensional %irreducible
component satisfies $w_1/w_2=\mathrm{const}$ nor $w_1w_2=\mathrm{const}$.

Analogously, one one-dimensional irreducible component of the solution set of the other pair of equations satisfies $w_1/w_2=+1$, and no other one-dimensional irreducible component satisfies $w_1/w_2=\mathrm{const}$ nor $w_1w_2=\mathrm{const}$. As a result, no one-dimensional irreducible component of the solution set of all four equations satisfies $w_1/w_2=\mathrm{const}$ nor $w_1w_2=\mathrm{const}$. So, our example does not belong to the linear compound class, even after switching the
boundary strips.
\end{proof}

%%%%%%%%%%%%%%%%%%%%%%%%%%%%%%%%%%%%%%%%%%%

Let us present another particular case of Theorem~\ref{main_th:th1}, which is paradoxical.

\begin{example}[See Figures~\ref{exfig:pic001.1}--\ref{exfig:pic001.2}]
\label{ex:ex1a}
Let one polyhedron have flat angles
    \begin{align*}
    &\alpha_1 = \alpha_4 = \delta_2 = 180^\circ - \delta_3 = 15^\circ,
    &&\beta_1 = \beta_4 = \gamma_2 = 180^\circ -  \gamma_3 = 60^\circ,
    \\
    &\gamma_1 = \gamma_4 = \beta_2 = 180^\circ -  \beta_3 = 75^\circ,
    &&\delta_1 = \delta_4 = \alpha_2 =  \alpha_3 =  90^\circ,
\end{align*}
and the cotangents of dihedral half-angles
{\small
\begin{align*}\label{eq:E.1.a}\tag{15a}
\left.
\begin{aligned}
&\cot{\frac{\theta_1}{2}} = \frac{\pm\sqrt{2}\sqrt{\sqrt{3}t^4+6t^2+2\sqrt{3}} - 2\sqrt{6} t}{2-\left(\sqrt{3}+1\right) t^2}, &
&\cot{\frac{\theta_2}{2}} =  t,\\
&\cot{\frac{\theta_3}{2}}= \frac{\pm\sqrt{2}\sqrt{\sqrt{3}t^4+6t^2+2\sqrt{3}} + 2\sqrt{6} t}{2-\left(\sqrt{3}+1\right) t^2}, &
&\cot{\frac{\theta_4}{2}} =  \pm\frac{\sqrt{2}\sqrt{\sqrt{3}t^4+6 t^2+2\sqrt{3}}}{\left(\sqrt{3}-1\right) t^2+2},
\end{aligned}
\right.
\end{align*}
}
where $t\ne -\sqrt{\sqrt{3}-1},\sqrt{\sqrt{3}-1}$ is a parameter
%in the range $(0, \sqrt{\sqrt{3}-1})\cup (\sqrt{\sqrt{3}-1}, +\infty)$
and the signs in $\pm$ agree.

Let the other polyhedron have the same flat angles and the cotangents of dihedral half-angles
\begin{align*}\label{eq:E.1.b}\tag{15b}
\left.
\begin{aligned}
&\cot{\frac{\theta_1}{2}} = \cot{\frac{\theta_3}{2}} =0, &
&\cot{\frac{\theta_2}{2}} = \sqrt{\frac{2}{3}}\cot{\frac{\theta_4}{2}}=\sqrt{1 + \sqrt{3}}.
%&
%&\cot{\frac{\theta_3}{2}} = 0, &
%&\cot{\frac{\theta_4}{2}} = \sqrt{\frac{3}{2}\left(1 + \sqrt{3}\right)}.
\end{aligned}
    \right.
\end{align*}
\end{example}

% \begin{figure}[htbp!]
%   \centering
%   \begin{subfigure}[t]{0.5\textwidth}
%     \centering
%     \includesvg[scale=0.5]{img/exact_ex2_1-10.svg}
%     \caption{$t=0.1$}
%     \label{exfig:pic001.1a}
%   \end{subfigure}
%   \begin{subfigure}[t]{0.45\textwidth}
%     \centering
%     \includesvg[scale=0.45]{img/exact_ex2_1-2.svg}
%     \caption{$t=0.5$}
%     \label{exfig:pic001.1b}
%   \end{subfigure}
%   \caption{The first polyhedron %of equimodular elliptic type
%   in Example~\ref{ex:ex1a} at different values of parameter $t$. The sign ``$+$'' is chosen in each $\pm$ in~\eqref{eq:E.1.a}.}
%   \label{exfig:pic001.1}
% \end{figure}

\begin{figure}[htbp]
  \centering
  \begin{subfigure}[t]{0.475\textwidth}
    \centering
    \includegraphics[scale=0.1]{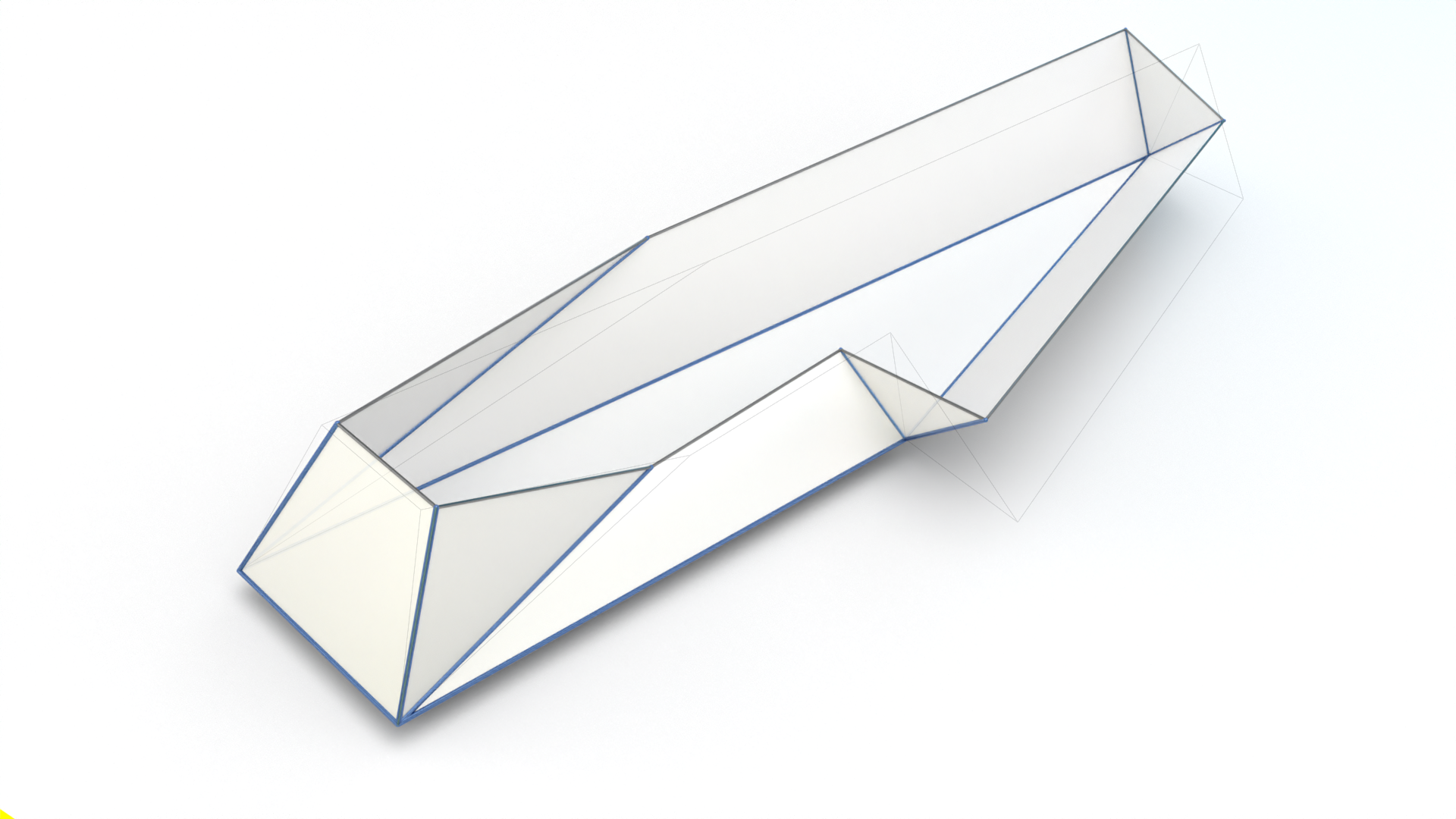}
    \caption{$t=0.1$}
    \label{exfig:pic001.1a}
  \end{subfigure}
  \begin{subfigure}[t]{0.475\textwidth}
    \centering
     \includegraphics[scale=0.1]{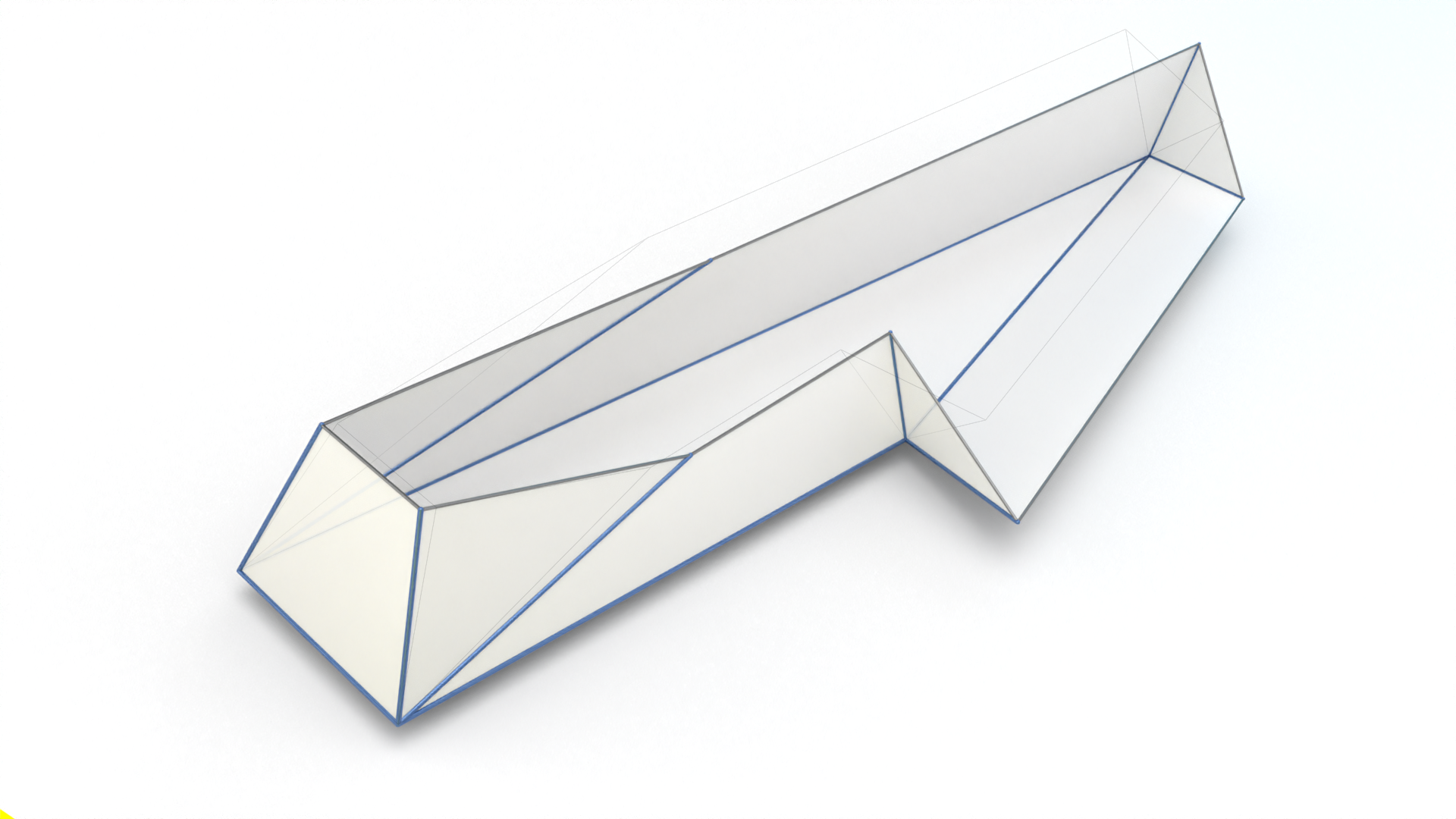}
    \caption{$t=0.5$}
    \label{exfig:pic001.1b}
  \end{subfigure}
  \caption{The first polyhedron %of equimodular elliptic type
  in Example~\ref{ex:ex1a} at different values of parameter $t$. The sign ``$+$'' is chosen in each $\pm$ in~\eqref{eq:E.1.a}.}
  \label{exfig:pic001.1}
\end{figure}

% \begin{figure}[htbp!]
%     \centering
%     \includesvg[scale=0.45]{img/exact_ex2_non-flexible.svg}
%     \caption{The second polyhedron in Example~\ref{ex:ex1a} for the fixed values of dihedral angles~\eqref{eq:E.1.b}. It has equimodular elliptic type but is not flexible.}
%     \label{exfig:pic001.2}
% \end{figure}

\begin{figure}[htbp]
    \centering
    \includegraphics[scale=0.175]{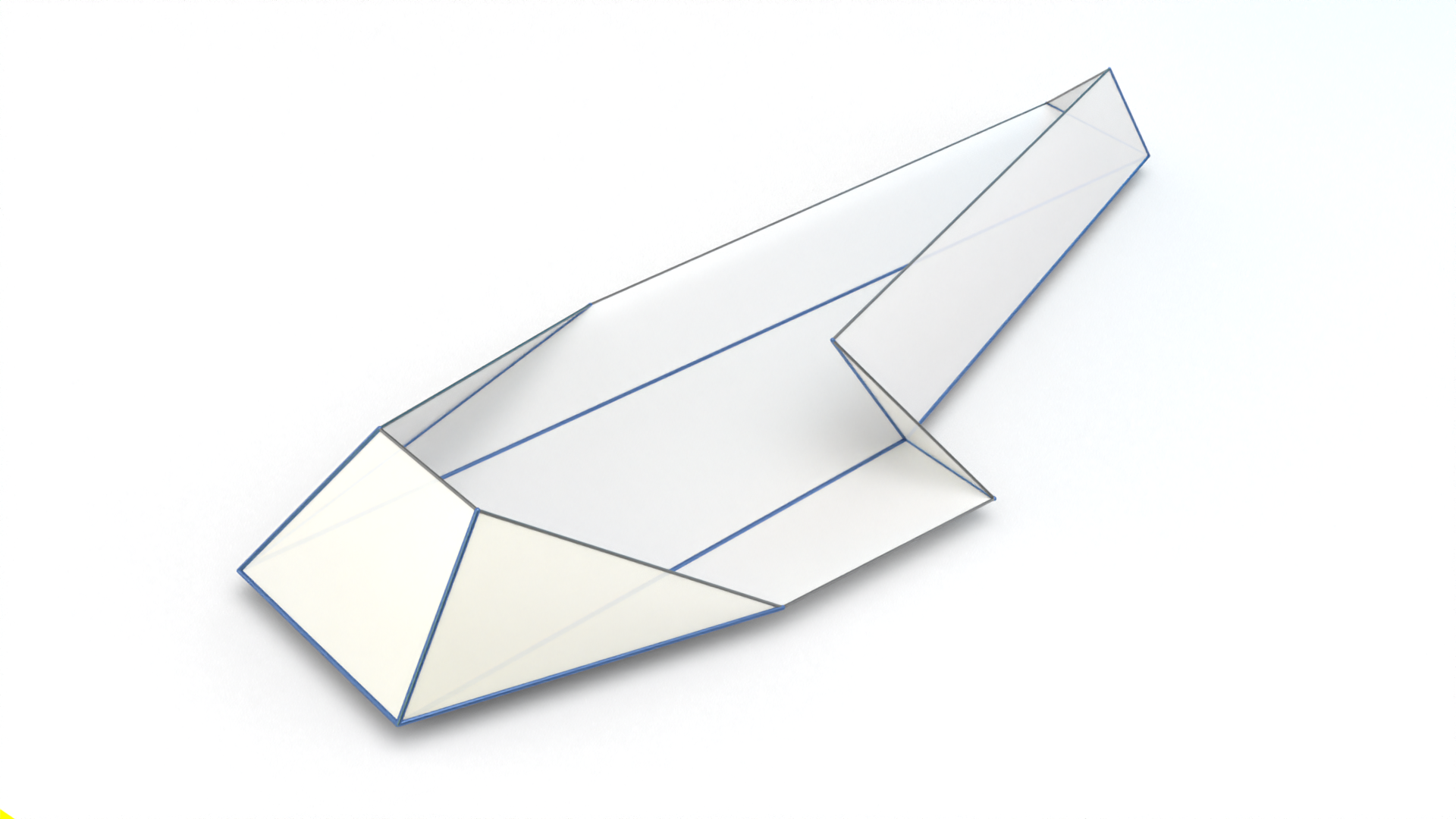}
    \caption{The second polyhedron in Example~\ref{ex:ex1a} for the fixed values of dihedral angles~\eqref{eq:E.1.b}. It has equimodular elliptic type but is not flexible.}
    \label{exfig:pic001.2}
\end{figure}

\begin{proposition}\label{ex:prop1a}
    Two polyhedra in Example~\ref{ex:ex1a} have the same flat angles, but one of them is flexible (in $\mathbb{R}^3$) and the other one is not (even in $\mathbb{C}^3$).
\end{proposition}

The explanation for this apparent paradox is that the solution set of the system of Bricard's equations has components of different dimensions.

\begin{proof}
We use a supplementary material (\texttt{Example2/helper2.nb})~\cite{Nurmatov2025Repo} for the proof, which %computes all the quantities $r_i$, $s_i$, $f_i$, $\sigma_i$, $M_i$ and
verifies every computation below.

First, we verify condition \eqref{eq:N.0} by considering all possible choices of signs in $\pm$. Next, we compute $\overline{\alpha}_1 = 105^\circ$, $\overline{\beta}_1 = 60^\circ$, $\overline{\gamma}_1 = 45^\circ$, $\overline{\delta}_1 = 30^\circ\in(0, 180^\circ)$. Thus, by Theorem~\ref{main_th:th1}, the first polyhedron is realizable and flexible in $\mathbb{R}^3$, with a flexion given by \eqref{eq:E.1.a}.

To demonstrate the existence and non-flexibility of the second polyhedron, consider the system of Bricard's equations \cite[Eq.~(20)]{izmestiev-2017} in the variables $w_1:=\cot{\frac{\theta_1}{2}}$, $t:=\cot{\frac{\theta_2}{2}}$, $w_2:=\cot{\frac{\theta_3}{2}}$, $z:=\cot{\frac{\theta_4}{2}}$ (see \texttt{Example2/helper2.nb} in~\cite{Nurmatov2025Repo}):
%(see \texttt{Example~1a/helper1a.nb} in~\cite{Nurmatov2025Repo}):
%(multiplied by suitable constants):
\begin{align*}
\frac{1}{\sqrt{3}}\left(\left(\sqrt{3}+1\right) t^2-2\right) w_1^2-4\sqrt{2}\, t w_1-\left(\sqrt{3}-1\right) t^2+2&=0,\\
\frac{1}{\sqrt{3}}\left(\left(\sqrt{3}+1\right) t^2-2\right) w_2^2+4\sqrt{2}\, t w_2-\left(\sqrt{3}-1\right) t^2+2&=0,\\
w_1^2 \left(z^2-\sqrt{3}-1\right)+2w_1z+\left(\sqrt{3}-1\right)z^2-3&=0,\\
w_2^2 \left(z^2-\sqrt{3}-1\right)+2w_2z+\left(\sqrt{3}-1\right)z^2-3&=0.
\end{align*}
(We have multiplied Bricard's equations by suitable constants.) Since~\eqref{eq:E.1.b} is a solution to the system, it follows that the second polyhedron exists.

% Let us show that $(w_1,t,w_2,z)=(0,\sqrt{\sqrt{3}+1},0,\sqrt{3/2}\sqrt{\sqrt{3}+1})$ given by
Let us show that~\eqref{eq:E.1.b} is an isolated solution to the system, so that the second polyhedron is not flexible (even in $\mathbb{C}^3$). Consider the first pair of equations. %which is the second and third equations.
Solving them as quadratic equations in $w_1$ and $w_2$, %respectively,
we parametrize the complex solutions by the first three expressions in~\eqref{eq:E.1.a} in a neighborhood of any point $(w_1,t,w_2)$ such that
$t^4+2\sqrt{3} t^2+2\ne 0$ and $\left(\sqrt{3}+1\right) t^2-2\ne 0$. Here, we choose any continuous branch of the square root in this neighborhood, and the signs in $\pm$ need \emph{not} agree.

The irreducible component of the solution set of the first pair of equations corresponding to the same choice of signs in $\pm$ does not pass through the point $(w_1,t,w_2)=(0,\sqrt{\sqrt{3}+1},0)$ and can be ruled out.

Consider the other irreducible component, where the signs in $\pm$ are opposite. On this component, $w_1=-w_2$. Substituting this relation into the second pair of Bricard's equations, we conclude that $w_1=w_2=0$ or $z=0$. In either case, we end up with a finite number of complex solutions to all four equations. %corresponding to the second irreducible component of the solution set of the first pair of equations.
Thus, solution~\eqref{eq:E.1.b} is isolated.
%, and the second polyhedron is not flexible (even in $\mathbb{C}^3$).
\end{proof}
%%%%%%%%%%%%%%%%%%%%%%%%%%%%%%%%%%%%%%%%%%%

Next, let us show how one can use the results from Sections~\ref{sec:notation}--\ref{sec:prexcrit} to construct different closed-form examples. % of polyhedra of equimodular elliptic type which not follow directly from definition of equimodular elliptic type.
For simplicity, we take $\delta_i = \frac{\pi}{2}$ for each $i=1,\ldots, 4$. We look over simple fractional values of $u$, $z_i$ and try to express $x_i$ and $y_i$ using \eqref{eq:M.1.b}--\eqref{eq:M.1.f} (see Proposition~\ref{main_th:prop1}) in order to satisfy conditions \eqref{eq:N.3}--\eqref{eq:N.5}. Here, conditions \eqref{eq:N.3} and \eqref{eq:N.4} are easy to satisfy. However, satisfying condition \eqref{eq:N.5} with explicit values of $z_i$ is not so easy. For this reason, we introduce an additional restriction on $z_i$ and $u$, useful for Lemma~\ref{ex:lem2} below. We use inequalities from Lemmas~\ref{plan_pr:lem2.5}--\ref{plan_pr:lem3} to get more restrictions. Next, using Lemma~\ref{plan_pr:prop1}, we calculate flat angles. Finally, we verify condition~\eqref{eq:N.0}.

The simplest choice satisfying all those restrictions is $u= \frac{1}{2}$, $z_1 = 1$, $z_2 = \frac{4}{3}$, $z_3 = \frac{3}{2}$, $z_4 = 2$, $y_1 = \frac{1}{2}$. It leads to the following example.

\begin{example}[See Figure~\ref{exfig:pic002}]\label{ex:ex2}
Consider the polyhedron with the cosines of flat angles
    \begin{align*}
    &\cos{\alpha_1} = \frac{1}{\sqrt{5}},
    &&\cos{\beta_1} = \frac{7}{5\sqrt{2}},
    &&\cos{\gamma_1} = -\frac{1}{\sqrt{10}},
    &&\cos{\delta_1} = 0,
    \\
    &\cos{\alpha_2} = \frac{1}{4\sqrt{11}},
    &&\cos{\beta_2} = \frac{7\sqrt{7}}{4\sqrt{22}},
    &&\cos{\gamma_2} = -\frac{1}{2\sqrt{2}},
    &&\cos{\delta_2} = 0,
    \\
    &\cos{\alpha_3} = \frac{1}{4\sqrt{11}},
    &&\cos{\beta_3} = \frac{7\sqrt{7}}{4\sqrt{22}},
    &&\cos{\gamma_3} = \frac{1}{2\sqrt{2}},
    &&\cos{\delta_3} = 0,
    \\
    &\cos{\alpha_4} = \frac{1}{\sqrt{5}},
    &&\cos{\beta_4} = \frac{7}{5\sqrt{2}},
    &&\cos{\gamma_4} = \frac{1}{\sqrt{10}},
    &&\cos{\delta_4} = 0,
\end{align*}
and the cotangents of dihedral half-angles
\begin{align*}\label{eq:E.2}\tag{16}
\left.
\begin{aligned}
    &\cot{\frac{\theta_1}{2}} =t, & &\cot{\frac{\theta_2}{2}}  = \frac{5\sqrt{7}\,t \mp\sqrt{10(3t^2 - 1)(2 - 3t^2)}}{4 + 9t^2}, \\
    &\cot{\frac{\theta_3}{2}}  = \frac{1}{t}, & &\cot{\frac{\theta_4}{2}}  = \frac{6\,t \mp\sqrt{2(3t^2 - 1)(2 - 3t^2)}}{1 + 3t^2},
\end{aligned}
    \right.
\end{align*}
where $t$ is a parameter in the range $\left(\frac{1}{\sqrt{3}}, \frac{\sqrt{2}}{\sqrt{3}}\right)$ and the signs in $\mp$ agree.

% \begin{figure}[htbp!]
%   \centering
%   \begin{subfigure}[t]{0.5\textwidth}
%     \centering
%     \includesvg[scale=0.5]{img/exact_ex3_3-5.svg}
%     \caption{$t=\frac{3}{5}$}
%     \label{exfig:pic002a}
%   \end{subfigure}
%   \begin{subfigure}[t]{0.45\textwidth}
%     \centering
%     \includesvg[scale=0.45]{img/exact_ex3_3-4.svg}
%     \caption{$t=\frac{3}{4}$}
%     \label{exfig:pic002b}
%   \end{subfigure}
%   \caption{The polyhedron of equimodular elliptic type in Example~\ref{ex:ex2} at different values of parameter $t$. The sign ``$-$'' is chosen in each $\mp$ in~\eqref{eq:E.2}.}
%   \label{exfig:pic002}
% \end{figure}

\begin{figure}[htbp!]
  \centering
  \begin{subfigure}[t]{0.475\textwidth}
    \centering
    \includegraphics[scale=0.11]{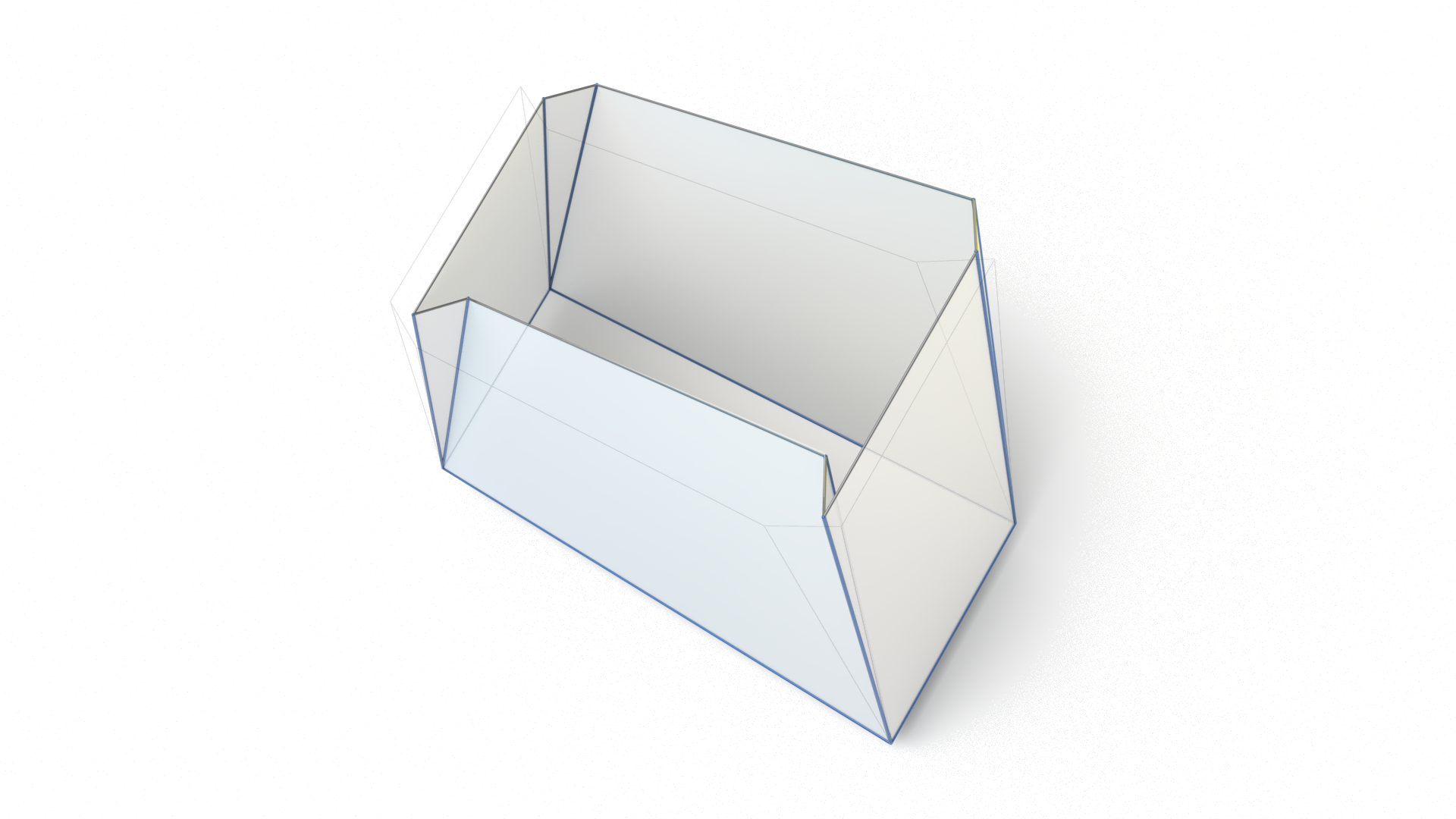}
    \caption{$t=\frac{3}{5}$}
    \label{exfig:pic002a}
  \end{subfigure}
  \begin{subfigure}[t]{0.475\textwidth}
    \centering
    \includegraphics[scale=0.11]{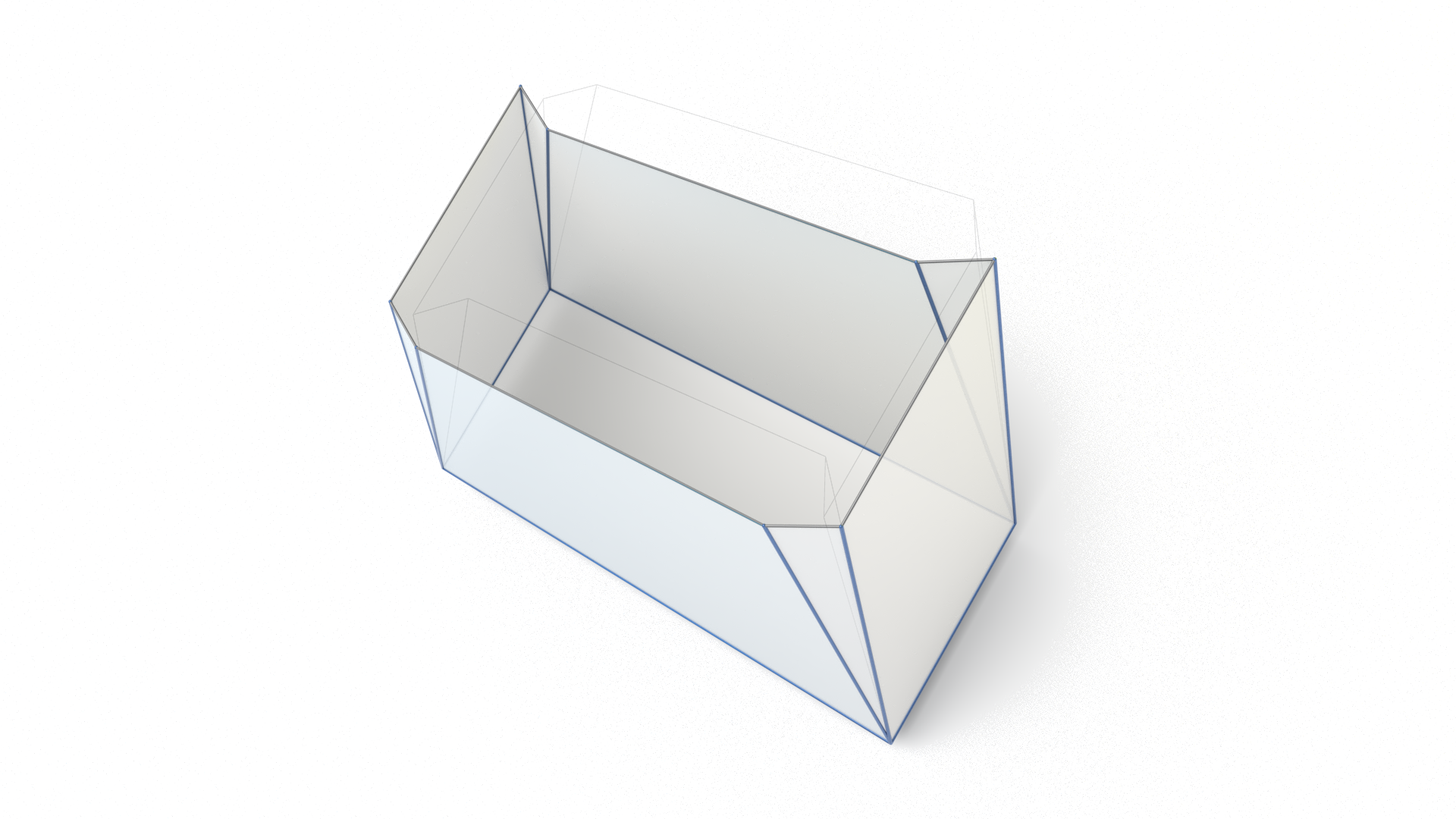}
    \caption{$t=\frac{3}{4}$}
    \label{exfig:pic002b}
  \end{subfigure}
  \caption{The polyhedron of equimodular elliptic type in Example~\ref{ex:ex2} at different values of parameter $t$. The sign ``$-$'' is chosen in each $\mp$ in~\eqref{eq:E.2}.}
  \label{exfig:pic002}
\end{figure}
\end{example}

\begin{proposition}\label{ex:prop2}
    The polyhedron from Example~\ref{ex:ex2} belongs to equimodular elliptic type and linear compound type (after switching one of the boundary strips), does not belong to any other classes in Izmestiev's classification, is flexible in $\mathbb{R}^3$, and a flexion is given by \eqref{eq:E.2}.
\end{proposition}

We need a simple lemma.

% \begin{proposition}\label{ex:prop2}
%     Polyhedron from Example~\ref{ex:ex2} belongs to equimodular elliptic type, is flexible in $\mathbb{R}^3$, and the flexion is given by \eqref{eq:E.2}.
% \end{proposition}
\begin{lemma}\label{ex:lem2}
Let $t_1$, $t_2 \in (0, \mathbf{i}K')$ and $k\in (0,1)$. Then $t_1 + t_2 = \mathbf{i}K'$ if and only if $k^2\operatorname{sn}^2 {t_1} \operatorname{sn}^2{t_2} = 1$.
\end{lemma}

\begin{proof}
    ($\Rightarrow$) If $t_1 + t_2 = \mathbf{i}K'$ then $\operatorname{sn}{t_2} = \operatorname{sn}{(-t_1 + \mathbf{i}K')} = -\frac{1}{k \operatorname{sn}{t_1}}$, where $\operatorname{sn}{t_1} \neq 0$ as $t_1 \in (0, \mathbf{i}K')$. Thus $k^2\operatorname{sn}^2 {t_1} \operatorname{sn}^2{t_2} = 1$.

     ($\Leftarrow$) Since $\operatorname{sn}{t_2}$ is odd and injective on $(-\mathbf{i}K',\mathbf{i}K')$, it follows that
     the equation $k^2\operatorname{sn}^2 {t_1} \operatorname{sn}^2{t_2} = 1$ implies $t_2=\mathbf{i}K'-t_1$.
     %
     %Since $t_i \in (0, \mathbf{i}K')$ then $\operatorname{sn}{t_i} = \frac{\mathbf{i}\sqrt{\operatorname{dn}^2{t_i} - 1}}{k}$, $\operatorname{cn}{t_i} = \frac{\sqrt{\operatorname{dn}^2{t_i} - 1 + k^2}}{k}$, and $\operatorname{dn}{t_i} > 1$, where $i=1, 2$. Thus $\operatorname{dn}{t_1}\operatorname{dn}{t_2} - k^2\operatorname{sn}{t_1}\operatorname{sn}{t_2}\operatorname{cn}{t_1}\operatorname{cn}{t_2} > 1$.
     %
     %On the other hand, $\operatorname{dn}{(t_1 + t_2)} = \frac{\operatorname{dn}{t_1}\operatorname{dn}{t_2} - k^2\operatorname{sn}{t_1}\operatorname{sn}{t_2}\operatorname{cn}{t_1}\operatorname{cn}{t_2}}{1 - k^2\operatorname{sn}^2 {t_1} \operatorname{sn}^2{t_2}}$, and since numerator is not zero and denominator is zero then $t_1 + t_2$ is a pole of delta amplitude function $\operatorname{dn}(t)$, i.e. $t_1 + t_2 = 2Km + 2\mathbf{i}K'n + \mathbf{i}K'$, where $m$, $n\in\mathbb{Z}$. Furthermore, since $t_1$, $t_2 \in (0, \mathbf{i}K')$ then $t_1 + t_2 = \mathbf{i}K'$.
\end{proof}

\begin{proof}[Proof of Proposition~\ref{ex:prop2}]
%Again, as in Example~\ref{ex:ex1},
In this a computer-assistant proof, we use a supplementary material (\texttt{Example3/helper3.nb})~\cite{Nurmatov2025Repo}, which computes all the quantities $r_i$, $s_i$, $f_i$, $\sigma_i$, $M_i$ and verifies every step below starting from expressing flat angles as inverse cosine function.

First, by %\cite[Theorem~1.1]{stachel2010} and \cite[Lemma 4.3]{izmestiev-2017}
Bricard's equations \cite[Eq.~(19)]{izmestiev-2017}
our example is realizable and flexible in $\mathbb{R}^3$ with the flexion given by \eqref{eq:E.2}, as verified in supplementary Mathematica notebook (\texttt{Example3/helper3.nb})~\cite{Nurmatov2025Repo}.

Second, we verify \eqref{eq:N.0} by considering all possible choices of signs in $\pm$. Then using \eqref{eq:N.2} we find $r_1 = r_2 = \frac{4}{3}$, $r_3= r_4 = \frac{5}{2}$, $s_1= s_4 = 3$, $s_2=s_3 = \frac{11}{6}$, $f_1 = 2$, $f_2=\frac{7}{4}$, $f_3 = \frac{5}{3}$, $f_4= \frac{3}{2}$, $M_1 = M_2 = M_3 = M_4 = M=\frac{1}{2}$, $\sigma_1 = \arccos{\left(-\frac{1}{\sqrt{2}}\right)}$, $\sigma_2 = \arccos{\left(-\frac{3\sqrt{7}}{2\sqrt{22}}\right)}$, $\sigma_3 = \arccos{\left(-\frac{2}{\sqrt{11}}\right)}$, and $\sigma_4 = \arccos{\left(-\frac{1}{\sqrt{5}}\right)}$. Thus \eqref{eq:N.0}, \eqref{eq:N.3}, and \eqref{eq:N.4} hold.

Third, observe that $\sigma_i < 180^\circ$ and $p_iq_i \in \mathbb{R}_{>0}$, see \eqref{eq:N.8}. Then by Table~\ref{tab:tpqsigma}, $t_i \in (0, \mathbf{i}K')$, for each $i=1, \ldots, 4$. Introduce $t_5:=t_1$. By \eqref{eq:N.7} and \eqref{eq:N.9}, we get $k^2\operatorname{sn}^2 {t_i} \operatorname{sn}^2{t_{i+1}} = \frac{(1 - \operatorname{dn}^2 {t_i})(1 - \operatorname{dn}^2 {t_{i+1}})}{k^2} = \frac{(1 - f_i)(1 - f_{i+1})}{1-M} = 1$ for $i=2, 4$. Then by Lemma~\ref{ex:lem2}, we have $t_2 + t_3 = \mathbf{i}K'$ and $t_1 + t_4 = \mathbf{i}K'$. Hence $t_1 - t_2 - t_3 + t_4 = 0$, and \eqref{eq:N.5} holds. Thus our example is of equimodular elliptic type.

Next, since $\cot{\frac{\theta_1}{2}} \cot{\frac{\theta_3}{2}}$ is constant it follows that our example belongs to linear compound class after switching the upper boundary strip. Thus it is in the intersection of linear compound and equimodular  elliptic classes.

Finally, analogously to the proof of Proposition~\ref{ex:prop1}, we check that this example does not fall into any of the other classes introduced by Izmestiev \cite{izmestiev-2017}, even after switching the boundary strips.
\end{proof}

\subsection{Numerical Examples} \label{sec-numeric}

\bluenew{In this subsection, we shift our perspective from the explicit generative approach to address the broader general case. Specifically, by treating the dihedral angles and flat angles of the base face as inputs, we reduce the problem to a system of eight equations in nine variables.}

To find concrete examples that satisfy the existence criterion in Proposition~\ref{main_th:prop1}, we employ a two-phase numerical search-and-verification procedure. The input consists of dihedral angles $\{\theta_i\}$ and flat angles $\{\delta_i\}$ with all angles in $(0,\pi)$ and $\sum_{i=1}^4 \delta_i = 2\pi$ (Eq.~\eqref{eq:M.1.0}), enforced by construction.

The unknowns are the parameters $(u,\,x_1,\,x_3,\,y_1,\,y_2,\,z_1,\,z_2,\,z_3,\,z_4)$
%$(u, x_i, y_i, z_i)$.
We first incorporate symmetry relations~\eqref{eq:M.1.a}--\eqref{eq:M.1.e} into~\eqref{eq:M.1.f} and~\eqref{eq:M.1.g}. To remove sign ambiguities, we square~\eqref{eq:M.1.f} and~\eqref{eq:M.1.g}. The resulting system has eight equations in the nine variables.
%\(
%V=(u,\,x_1,\,x_3,\,y_1,\,y_2,\,z_1,\,z_2,\,z_3,\,z_4).
%\)
We therefore augment it with a dummy equation $0=0$
% auxiliary constraint $f_9(V)=0$ so as
to obtain a square polynomial system. % $F(V)=0$.

To explore the (generally) non-convex solution set, we adopt a multi-start strategy, solving the system %$F(V)=0$
from many randomized initial guesses. These seeds are not chosen blindly: they are drawn from admissible domains implied by our theory, in particular Lemmas~\ref{plan_pr:lem2} and \ref{plan_pr:lem2.5}, as well as additional heuristic inequalities that bound the dihedral angles of realizable polyhedra.
%ranges for realizability in $\mathbb{R}^3$.
%(see \ref{plan_pr:lem4.75} in~\ref{app:guessflex}).
This significantly improves robustness and efficiency.

Each converged solution of the %squared
system is treated as a candidate and subjected to rigorous filtering. We discard non-physical ($u\ge 1$), non-convergent, or duplicate solutions. Survivors are required to admit real-valued angles, i.e., %the arguments of the inverse cosine
the right sides of expressions in Lemma~\ref{plan_pr:prop1} must lie in $[-1,1]$. Finally, each candidate is validated against the original (unsquared) equations~\eqref{eq:M.1.f} and~\eqref{eq:M.1.g}, %confirming the existence of a consistent sign pattern $\{\varepsilon_i\}$, satisfaction of the ellipticity
as well as conditions~\eqref{eq:N.0} and%fulfillment of the period condition
~\eqref{eq:N.5} for some choice of 
%$\{e_j\}$
$e_j$.
At this step, it is convenient to treat $\varepsilon_i$ and $\sigma_i$ as additional variables and their definitions~\eqref{eq:N.2} as additional equations to be validated.
The %full procedure is outlined in Algorithm~\ref{alg:polyhedron_search} in~\ref{app:vis}; the
%corresponding 
resulting Python implementation is available in the supplementary material (\texttt{Algorithm/numerical\_search.py})~\cite{Nurmatov2025Repo}.

%To construct numerical examples, we use Algorithm~\ref{alg:polyhedron_search}.
As for first example, we took as input
\begin{align*}
    \delta_1 &= 120^\circ, &\delta_2 &= 80^\circ,  &\delta_3 &= 85^\circ,  &\delta_4 &= 75^\circ,\\
\theta_1 &= 130^\circ, &\theta_2 &= 140^\circ, &\theta_3 &= 125^\circ, &\theta_4 &= 135^\circ.
\end{align*}
%$\delta_1 = 120^\circ$, $\delta_2 = 80^\circ$,  $\delta_3 = 85^\circ$,  $\delta_4 = 75^\circ$, $\theta_1 = 130^\circ$, $\theta_2 = 140^\circ$,  $\theta_3 = 125^\circ$, $\theta_4 = 135^\circ$.
One admissible solution returned by our algorithm %~\ref{alg:polyhedron_search}
is presented in the following example; %Example~\ref{ex:ex3};
full-precision values (additional 12 decimal digits) are provided in the supplementary repository (\texttt{Example4})~\cite{Nurmatov2025Repo}.

\begin{example}[see Figure~\ref{exfig:pic003}]\label{ex:ex3}
Consider the polyhedron with flat angles
\begin{align*}
    &\alpha_1 = 91.32...^\circ, &&\alpha_2 = 115.75...^\circ,  &&\alpha_3 = 31.19...^\circ,  &&\alpha_4 = 28.19...^\circ,\\
    &\beta_1 = 27.53...^\circ, &&\beta_2 = 29.33...^\circ,  &&\beta_3 = 113.66...^\circ,  &&\beta_4 = 107.67...^\circ, \\
    &\gamma_1 = 103.21...^\circ, &&\gamma_2 = 109.78...^\circ,  &&\gamma_3 = 89.61...^\circ,  &&\gamma_4 = 121.75...^\circ,\\
    &\delta_1 = 120^\circ, &&\delta_2 = 80^\circ,  &&\delta_3 = 85^\circ,  &&\delta_4 = 75^\circ,
\end{align*}
and the cotangents of dihedral half-angles
% \begin{align*}\label{eq:E.3}\tag{E.4}
% \left.
% \begin{aligned}
%     &\cot{\frac{\theta_1}{2}}  =t,\\
%     &\cot{\frac{\theta_2}{2}}  = \frac{7.58... \, t \mp 2.16... \,\sqrt{\left(1 - 0.57... \, t^{2}\right)\left(-1 + 7.59... \, t^{2}\right)}}{3.58... + 7.59... \, t^{2}}, \\
%     &\cot{\frac{\theta_3}{2}}  = \frac{-0.24...\, t
%       \mp 2.06... \,\sqrt{\left(1 - 0.57... \, t^{2}\right) \left(-1 + 7.59... \, t^{2}\right)}}{2.09... \, t^{2}-3.66...}, \\
%     &\cot{\frac{\theta_4}{2}}  = \frac{23.83... \, t
%       \mp 5.76... \,\sqrt{\left(1 - 0.57... \, t^{2}\right) \left(-1 + 7.59...\, t^{2}\right)}}{14.66... + 7.59... \, t^{2}},
% \end{aligned}
%     \right.
% \end{align*}
\begin{align*}\label{eq:E.3}\tag{17}
\left.
\begin{aligned}
    &\cot{\frac{\theta_1}{2}}  =t,\\
    &\cot{\frac{\theta_2}{2}}  = \frac{0.99... \, t \mp 0.28... \,\sqrt{-4.34...\,t^4 + 8.16...\, t^2  -1}}{t^{2} + 0.47...}, \\
    &\cot{\frac{\theta_3}{2}}  = \frac{-0.11...\, t
      \mp 0.98... \,\sqrt{-4.34...\,t^4 + 8.16...\, t^2  -1}}{t^{2}-1.74...}, \\
    &\cot{\frac{\theta_4}{2}}  = \frac{3.13... \, t
      \mp 0.75... \,\sqrt{-4.34...\,t^4 + 8.16...\, t^2  -1}}{t^{2} + 1.93...},
\end{aligned}
    \right.
\end{align*}
where $t$ 
%is a parameter 
lies in the range $\left(0.36..., 1.32...\right)$
%$\left[0.362..., 1.320...\right)\cup \left(1.320..., 1.322...\right]$.
and the signs in $\mp$ agree.

% \begin{center}
%     \centering
%     \includesvg[scale=0.8]{img/num_ex1.svg}
%     \captionof{figure}{Numerical example of a polyhedron of equimodular elliptic type. See Example~\ref{ex3}.}
%   \label{fig:pic6}
% \end{center}

% \begin{figure}[htbp!]
%     \centering
%     \includesvg[scale=0.8]{img/num_ex1.svg}
%     \caption{Numerical example of a polyhedron of equimodular elliptic type at the state $\theta_1 = 130^\circ$, $\theta_2 = 140^\circ$,  $\theta_3 = 125^\circ$, $\theta_4 = 135^\circ$. See Example~\ref{ex:ex3}, where the minus sign is chosen in each $\mp$.}
%     \label{exfig:pic003}
% \end{figure}

\begin{figure}[htbp!]
    \centering
    \vspace{2mm}
    \includegraphics[trim={11.5cm 8.5cm 8.5cm 10.5cm},clip, scale=0.15]{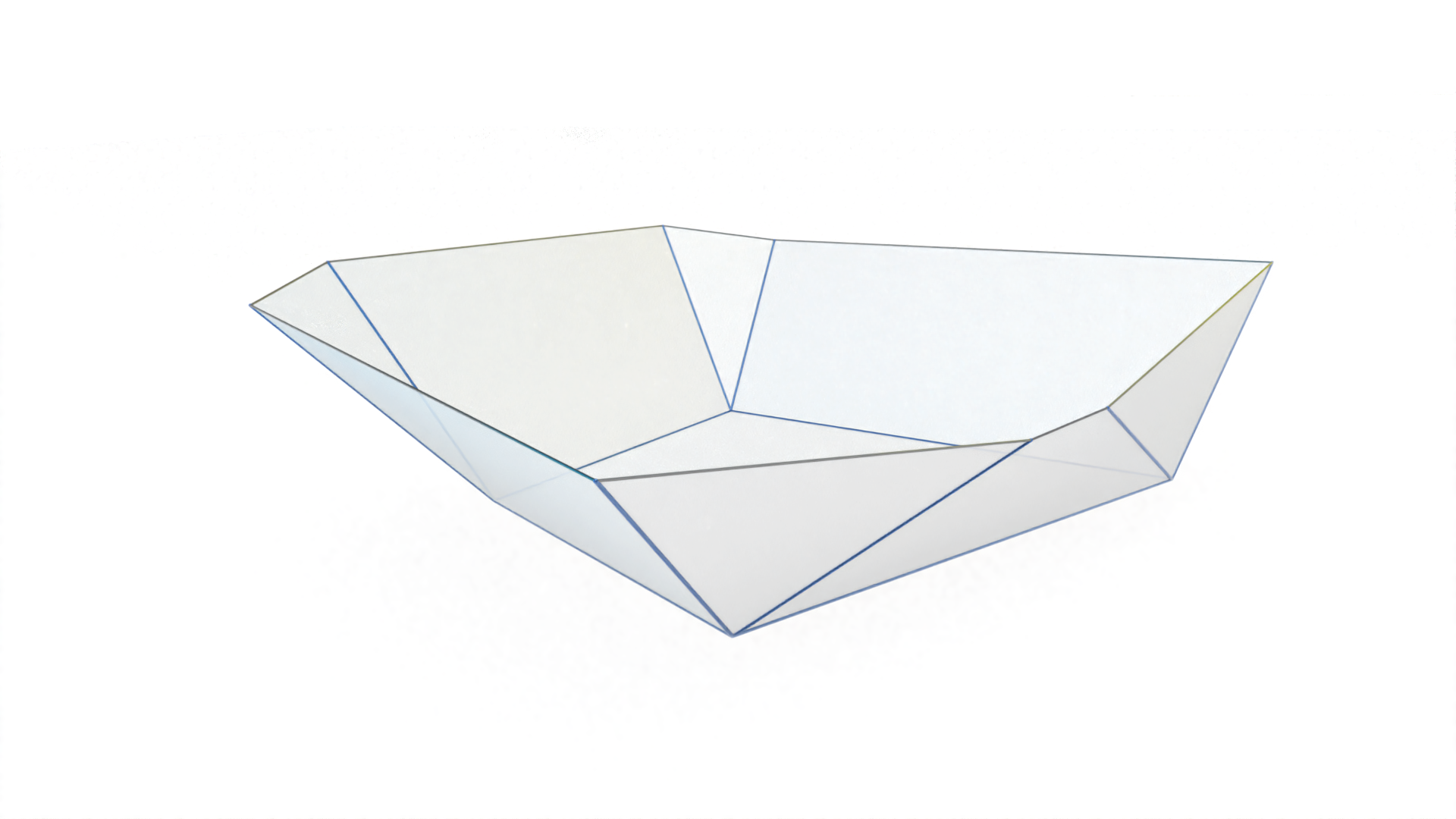}
    \vspace{2mm}
    \caption{Numerical example of a polyhedron of equimodular elliptic type at the state $\theta_1 = 130^\circ$, $\theta_2 = 140^\circ$,  $\theta_3 = 125^\circ$, $\theta_4 = 135^\circ$. See Example~\ref{ex:ex3}, where the minus sign is chosen in each $\mp$.}
    \label{exfig:pic003}
\end{figure}
\end{example}

To verify that this example has the equimodular elliptic type  (up to some tolerance), we processed it with the code in the supplementary material (\texttt{Example4/helper4.nb})~\cite{Nurmatov2025Repo}, following the same approach as in Examples~\ref{ex:ex1}--\ref{ex:ex2}.
Condition~\eqref{eq:N.0} was verified by considering all possible choices of signs in $\pm$. %Moreover, we observed the exact identities
We verified approximate equalities $M_1=M_2=M_3=M_4 = M = 0.92...$, $r_1=r_2 = 1.13...$, $r_3=r_4=0.68...$, $s_1=s_4=1.17...$, $s_2=s_3=1.09...$ with a further 10--14 digits available in the supplement. %Thus conditions \eqref{eq:N.0}, \eqref{eq:N.3} and \eqref{eq:N.4} are satisfied exactly.
Period condition~\eqref{eq:N.5} was validated numerically to a tolerance of $10^{-15}$. %In addition,
Bricard's equations~\cite[Eq.~(19)]{izmestiev-2017} were satisfied for~\eqref{eq:E.3} to a tolerance of $10^{-12}$. % by Lemma~2.2 of~\cite{izmestiev-2017} this implies that the
In a sense, this means that the example is flexible in \(\mathbb{R}^3\) to tolerance $10^{-12}$.

Finally, we confirmed that the example belongs \emph{exclusively} to the equimodular elliptic class (even after switching boundary strips),
by the 
%same 
approach used for Example~\ref{ex:ex1}; see (\texttt{Example4/helper4.nb})~\cite{Nurmatov2025Repo}.

In all preceding examples, we have \(M<1\). To complement these, we now consider a one with \(M>1\). While, for the same input data as in Example~\ref{ex:ex3}, our algorithm %~\ref{alg:polyhedron_search}
already returns several solutions with \(M>1\), we chose to test the effectiveness of the procedure on a different dataset. Specifically, we took as input
\begin{align*}
\delta_1 &= 60^\circ, & \delta_2&= 115^\circ, & \delta_3 &= 80^\circ, &  \delta_4 &= 105^\circ,\\
\theta_1 &= 120^\circ, & \theta_2 &= 140^\circ, & \theta_3 &= 110^\circ, &\theta_4 &= 130^\circ
\end{align*} and obtained an admissible solution reported in Example~\ref{ex:ex4}. Full-precision values (additional 12 decimal digits) are provided in the supplementary repository (\texttt{Example5})~\cite{Nurmatov2025Repo}.

\begin{example}[See Figure~\ref{exfig:pic004}]\label{ex:ex4}
Consider the polyhedron with flat angles
\begin{align*}
    &\alpha_1 = 26.20...^\circ, &&\alpha_2 = 16.16...^\circ,  &&\alpha_3 = 134.65...^\circ,  &&\alpha_4 = 117.95...^\circ,\\
    &\beta_1 = 82.24...^\circ, &&\beta_2 = 130.87...^\circ,  &&\beta_3 = 34.44...^\circ,  &&\beta_4 = 49.52...^\circ, \\
    &\gamma_1 = 21.94...^\circ, &&\gamma_2 = 18.85...^\circ,  &&\gamma_3 = 145.36...^\circ,  &&\gamma_4 = 149.02...^\circ,\\
    &\delta_1 = 60^\circ, &&\delta_2 = 115^\circ,  &&\delta_3 = 80^\circ,  &&\delta_4 = 105^\circ,
\end{align*}
and the cotangents of dihedral half-angles
% \begin{align*}\label{eq:E.4}\tag{E.5}
% \left.
% \begin{aligned}
%     &\cot{\frac{\theta_1}{2}}  =t,\\
%     &\cot{\frac{\theta_2}{2}}  = \frac{-0.75... \, t \pm 1.00... \,\sqrt{\left(1 + 3.45... \, t^{2}\right)\left(1 - 0.78... \, t^{2}\right)}}{1.14... + 3.45... \, t^{2}}, \\
%     &\cot{\frac{\theta_3}{2}}  = \frac{2.09...\, t
%       \pm 1.20... \,\sqrt{\left(1 + 3.45... \, t^{2}\right) \left(1 - 0.78... \, t^{2}\right)}}{2.60... + 3.84... \, t^{2}}, \\
%     &\cot{\frac{\theta_4}{2}}  = \frac{-1.35... \, t
%       \pm 1.49... \,\sqrt{\left(1 + 3.45... \, t^{2}\right) \left(1 - 0.78...\, t^{2}\right)}}{1.22... + 3.45... \, t^{2}},
% \end{aligned}
%     \right.
% \end{align*}
\begin{align*}\label{eq:E.4}\tag{18}
\left.
\begin{aligned}
    &\cot{\frac{\theta_1}{2}}  =t,\\
    &\cot{\frac{\theta_2}{2}}  = \frac{-0.21... \, t \pm 0.29... \,\sqrt{-2.70... \, t^4 + 2.67... \, t^2 + 1}}{t^2 + 0.33...}, \\
    &\cot{\frac{\theta_3}{2}}  = \frac{0.54...\, t
      \pm 0.31... \,\sqrt{-2.70... \, t^4 + 2.67... \, t^2 + 1}}{t^2 + 0.67...}, \\
    &\cot{\frac{\theta_4}{2}}  = \frac{-0.39... \, t
      \pm 0.43... \,\sqrt{-2.70... \, t^4 + 2.67... \, t^2 + 1}}{t^2 + 0.35...},
\end{aligned}
    \right.
\end{align*}
where $t$ is a parameter in the range $\left(0, 1.13...\right)$ and the signs in $\pm$ agree.
\end{example}
% \begin{center}
%     \centering
%     \includesvg[scale=0.8]{img/num_ex2.svg}
%     \captionof{figure}{Numerical example of a polyhedron of equimodular elliptic type for the case of $\Tilde{e}=+1$ at the state $\theta_1 = 120^\circ$, $\theta_2 = 140^\circ$,  $\theta_3 = 110^\circ$, $\theta_4 = 130^\circ$. See Example~\ref{ex:ex4}.}
%   \label{exfig:pic004}
% \end{center}
% \begin{figure}[htbp!]
%     \centering
%     \includesvg[scale=0.8]{img/num_ex2.svg}
%     \caption{The numerical example of a polyhedron of equimodular elliptic type in Example~\ref{ex:ex4} at the state $\theta_1 = 120^\circ$, $\theta_2 = 140^\circ$,  $\theta_3 = 110^\circ$, $\theta_4 = 130^\circ$. The sign ``$+$'' is chosen in each $\pm$ in~\eqref{eq:E.4}.}
%     \label{exfig:pic004}
% \end{figure}

\begin{figure}[htbp!]
    \centering
    \includegraphics[trim={9cm 9cm 6.3cm 9.5cm},clip, scale=0.15]{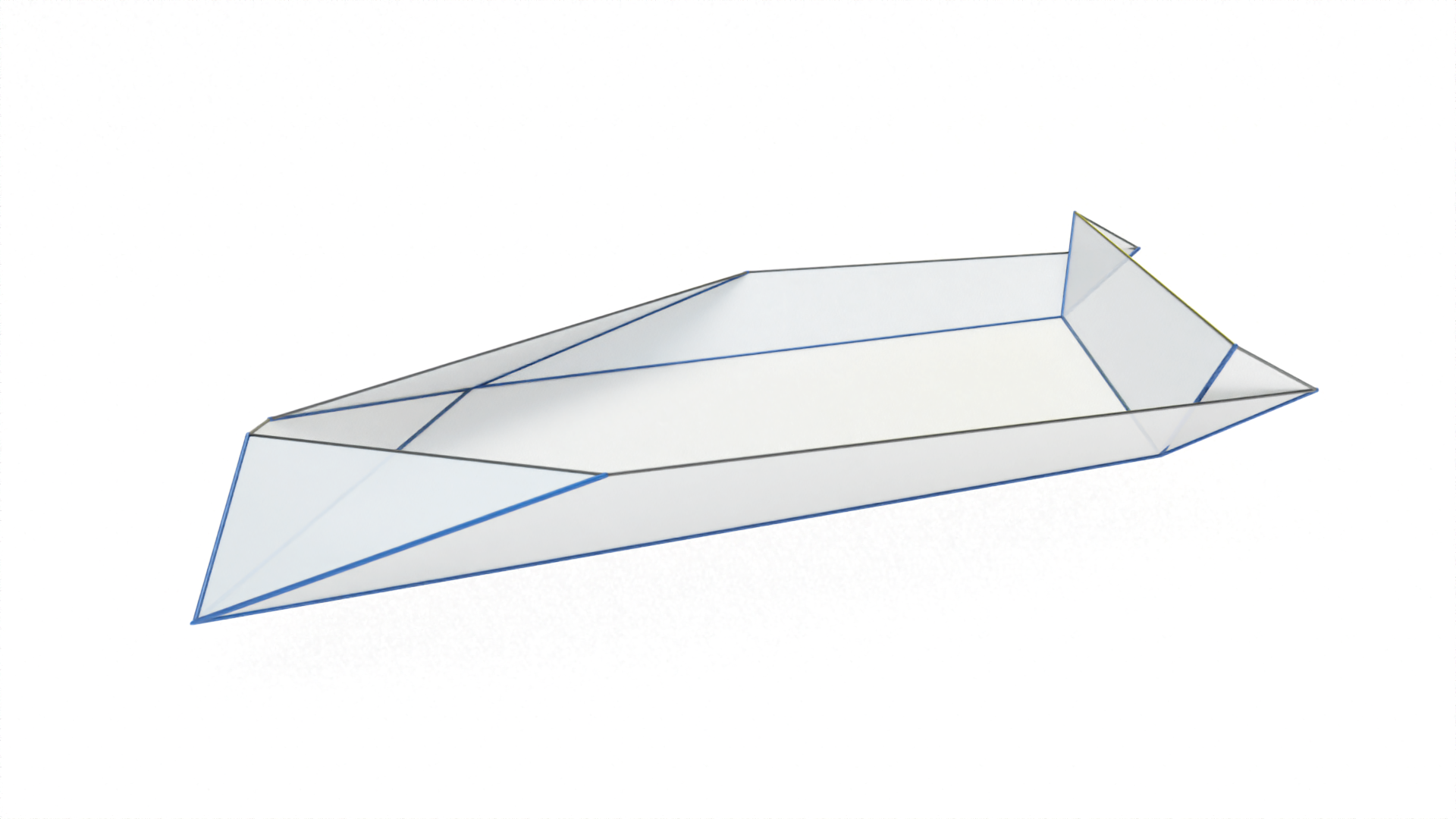}
    \vspace{2mm}
    \caption{The numerical example of a polyhedron of equimodular elliptic type in Example~\ref{ex:ex4} at the state $\theta_1 = 120^\circ$, $\theta_2 = 140^\circ$,  $\theta_3 = 110^\circ$, $\theta_4 = 130^\circ$. The sign ``$+$'' is chosen in each $\pm$ in~\eqref{eq:E.4}.}
    \label{exfig:pic004}
\end{figure}

To verify that this example has equimodular elliptic type (to numerical tolerance), we processed it with the supplementary Mathematica notebook (\texttt{Example5/helper5.nb})~\cite{Nurmatov2025Repo}, following the same approach as in previous examples. Condition~\eqref{eq:N.0} was verified by considering all possible choices. % of signs in $\pm$.
We verified the approximate equalities $M_1=M_2=M_3=M_4 = M = 1.22...$, $r_1=r_2 = 0.71...$, $r_3=r_4=0.88...$, $s_1=s_4=0.58...$, $s_2=s_3=0.80...$ with a further 11--13 digits available in the supplement.
%, which implies that~\eqref{eq:N.3} and~\eqref{eq:N.4} hold exactly for this instance.
Period condition~\eqref{eq:N.5} was validated numerically to a tolerance of \(10^{-15}\). In addition, Bricard's equations %~\eqref{def:D.1}
were satisfied for~\eqref{eq:E.4} to a tolerance of \(10^{-12}\).
%; by Lemma~2.2 of~\cite{izmestiev-2017} this implies that the example is flexible in \(\mathbb{R}^3\) to tolerance \(10^{-13}\).

Furthermore, this example does \emph{not} belong to the conjugate-modular class because $M_1 = M_2=M_3=M_4=M$ and $M\neq 2$; switching boundary strips does not affect %the value of
\(M_i\) and thus does not change this conclusion. As in Example~\ref{ex:ex1}, we verified that the example does \emph{not} belong to any other classes %other than equimodular elliptic class
(even after switching boundary strips); see (\texttt{Example5/helper5.nb})~\cite{Nurmatov2025Repo}.

\begin{remark}\label{ex:rem1}
In our simulations, Examples~\ref{ex:ex2}--\ref{ex:ex4} are non-self-intersecting, if we choose the upper sign in each $\mp$ and $\pm$,
and become self-intersecting, otherwise. Reproducible Python scripts are provided in the example folders of the supplementary repository~\cite{Nurmatov2025Repo}.
\end{remark}

\tmpcomment

\subsection{Small-scale Prototype}

We show that the proposed mechanisms can be successfully realized in practice by building the small-scale (about $150\times150\times100$ mm${}^3$) prototype of our QS-net shown in Figure \ref{dfig:pic001}, which has been previously presented in Figure \ref{exfig:pic001}. Unlike Maleczek et al. \cite{maleczek2022rapid} and others, we opted against rapid prototyping and the use of plastics, as these materials can introduce significant tolerances and undesired flexibility. Instead, we selected stainless steel for the fabrication of our model. Its motion is purely a consequence of its geometry, with no flexibility, deformation, or loose fitting of the material contributing to movement.

The mechanism is constructed from a 0.5 mm thick laser-cut hard-rolled stainless steel sheet (1.4404), CNC-cut cold-drawn stainless steel seamless capillary pipe (1.4301) with an outer diameter of 2.5 mm and a wall thickness of 0.5 mm, and straightened carbon spring steel wire (1.1200) with a diameter of 1.5 mm. The joint knuckles were fusion pulse TIG-welded to the edges of the laser-cut quadrilaterals. The fit between the knuckle and the pin at the joints is H10/h8 (knuckle inner diameter 1.5 mm -0/+50 µm, pin diameter 1.5 mm +0/-14 µm). The resulting mechanism is both stiff and smooth, with only a single degree of motion and no backlash or wiggle.

\begin{figure}[htbp!]
    \centering
    \includegraphics[width=\textwidth]{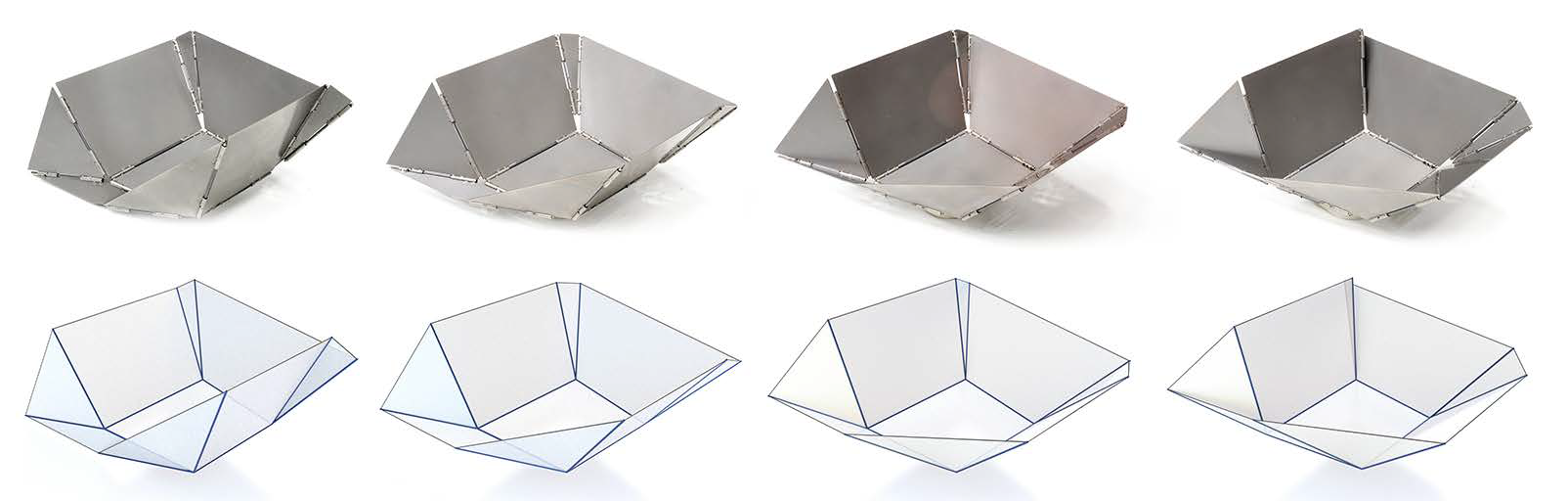}
    \caption{Illustration of several configurations of our fabricated flexible QS-net. Photos (top) and corresponding renderings (bottom) are shown.}
    \label{dfig:pic001}
\end{figure}

\endtmpcomment 

%\section{Discussion and Conclusions}
\section{Discussion}
\label{sec:disc}
In this work, we fill a notable gap in the classification of flexible Kokotsakis polyhedra with quadrangular base by providing the first \emph{explicit} constructions of the \emph{equimodular elliptic} type -- most directly via Theorem~\ref{main_th:th1} (the QS-net) -- and by deriving an \emph{explicit algebraic characterization}
%existence criterion}
(Proposition~\ref{main_th:prop1}) that links flat and dihedral angles through the variables \((u,x_i,y_i,z_i)\) and the relations~\eqref{eq:M.1.a}--\eqref{eq:M.1.g}. Although the definition of quasi-symmetric nets is elementary and direct, the path to this definition was via %through
the implementation and analysis of the more complicated algebraic characterization.
% existence criterion.

Our results complement Izmestiev's algebraic classification~\cite{izmestiev-2017}, which identified this type but did not provide concrete realizations. Beyond existence, we establish flexibility in \(\mathbb{R}^3\) by providing closed-form flexions for analytic examples (Examples~\ref{ex:ex1}--\ref{ex:ex2}), presenting numerical instances validated to tight tolerances (Examples~\ref{ex:ex3}--\ref{ex:ex4}), and giving systematic evidence that these instances %(except for Example~\ref{ex:ex2})
lie \emph{exclusively} within the equimodular elliptic class even after switching the boundary strips (Examples~\ref{ex:ex1}, \ref{ex:ex3}--\ref{ex:ex4}). Importantly, our examples cover both cases \(M<1\) and \(M>1\), indicating the breadth of realizable behavior within this class.
Across our examples we observed both non-self-intersecting motion %for \(\Tilde{e}=+1\)
and self-intersection %for \(\Tilde{e}=-1\)
(see Remark~\ref{ex:rem1}).

%%%%%%%% Too many wrong statements in the next paragraph %%%%%%%%%%%%%%%
%A general criterion predicting self-intersection directly from \(\{\alpha_i,\beta_i,\gamma_i,\delta_i\}\) and \(\{\theta_i\}\) remains open. Under the process of switching the boundary strips, our examples maintain \(M_1=\cdots=M_4=M\); this suggests (but does not prove) invariance of \(M\) under boundary switching and motivates a formal statement and proof. We also observed that \(\cos\alpha_i\cos\gamma_i\neq \cos\beta_i\cos\delta_i\) for all \(i=1,\dots,4\) persists under the process of switching the boundary strips in our data, raising the question of whether this non-orthodiagonality is preserved in general for the equimodular elliptic type. Likewise, in all examples the four angles \(\alpha_i,\beta_i,\gamma_i,\delta_i\) are pairwise distinct and no pair sums to \(180^\circ\), a property that also remained unchanged by the process of switching the boundary strips; determining whether these are generic invariants or a coincidence in our constructions is an interesting problem.
%%%%%%%%%%%%%%%%%%%%%%%%%%%%%%%%%%%%%%%%%%%%%%%%%%%%%%%%

For the numerical results, Bricard's equations~\cite[Eq.~(19)]{izmestiev-2017} are satisfied to \(10^{-12}\), period condition~\eqref{eq:N.5} to \(10^{-15}\), moduli condition~\eqref{eq:N.3} to \(10^{-14}\) or better, amplitudes condition~\eqref{eq:N.4} to \(10^{-12}\) or better, and condition~\eqref{eq:N.0} is satisfied with even better precision.
% and %the ellipticity, equal-moduli, and common-vertex amplitude,
% conditions~\eqref{eq:N.0}, \eqref{eq:N.3}, and~\eqref{eq:N.4} are satisfied with even better precision. %exactly.
These tolerances are more than sufficient for constructing physical prototypes. All scripts and notebooks are released to support full reproducibility~\cite{Nurmatov2025Repo}.

From a CAD perspective, Theorem~\ref{main_th:th1} offers an immediate source of simple, closed-form designs (QS-nets) with prescribed flat-angle relations. More generally, Proposition~\ref{main_th:prop1} is directly actionable: it reduces the synthesis of mechanisms of equimodular elliptic type to solving a structured algebraic system with simple symmetry constraints~\eqref{eq:M.1.a}--\eqref{eq:M.1.e} and two core relations -- a face-compatibility equation for \(\delta_i\)~\eqref{eq:M.1.f} and a kinematic equation for the dihedral angles~\eqref{eq:M.1.g}. This enables (i) fast feasibility checks during interactive modeling, (ii) \emph{inverse} design by prescribing partial angle data and solving for remaining parameters, and (iii) automatic rejection of designs drifting into non-elliptic regimes (violations of~\eqref{eq:N.0}) or non-physical solutions. The numerical pipeline (Section~\ref{sec-numeric}) is tailored for this workflow: it uses theory-informed sampling (Lemmas~\ref{plan_pr:lem2} and~\ref{plan_pr:lem2.5}) %, \ref{plan_pr:lem4.75})
to initialize solvers; candidates are verified against the \emph{unsquared} equations and consistent sign patterns \(\{\varepsilon_i\},\{e_j\}\); and the period condition~\eqref{eq:N.5} is certified at prescribed tolerances. Together with our constructive routines for coordinate formulae, edge and angle generation, and construction and interactive flexion (see \texttt{Appendix/appendix.pdf}~\cite{Nurmatov2025Repo}), this yields a robust pipeline from CAD to physical prototype.

Future work will address the open problems above, including the following directions:
%: formal criteria for self-intersection; invariance of \(M\) and non-orthodiagonality under the process of switching the boundary strips; and the status of pairwise distinctness/non-complementarity of \(\alpha_i,\beta_i,\gamma_i,\delta_i\). We also plan
to extend the theory and algorithms from \(3\times 3\) nets to general \(m\times n\) nets; to investigate whether all polyhedra of equimodular elliptic type are flexible in \(\mathbb{R}^3\) for some range of their dihedral angles and, if so, to derive explicit %single-parameter
flexion formulae; to seek simpler reformulations of period condition~\eqref{eq:N.5} in terms of elementary functions for faster verification; and to explore potential hidden symmetries within the equimodular elliptic type. 
\bluenew{Regarding the extension to general \(m\times n\) nets, a primary symbolic obstacle in scaling these explicit characterizations is the sheer complexity of the algebraic compatibility conditions. For example, imposing periodic boundary conditions on flat angles to form a quad-mesh requires solving highly coupled systems of Bricard-type equations. Alternatively, preliminary combinatorial approaches -- analogous to the ``building blocks'' in~\cite{dieleman2020} -- could be used to assemble larger patterns by treating our explicitly characterized \(3\times 3\) flexible meshes as local kinematic modules. We believe this interplay between local algebraic design and global combinatorial assembly provides a valuable direction for both the origami and discrete differential geometry communities.}

\section{Conclusion}
\label{sec:concl}
To summarize, we addressed a long–standing gap in the theory of flexible mechanisms by providing the first explicit constructions and an algebraic characterization %complete existence criterion for
of quadrangular Kokotsakis polyhedra of the equimodular elliptic type. Building on Izmestiev’s abstract framework, we translated the classification into a polynomial %concrete, verifiable
system that directly links flat and dihedral angles. This system -- formulated in Proposition~\ref{main_th:prop1} and complemented by the constructive Theorem~\ref{main_th:th1} -- serves as the central theoretical contribution and enables a rigorous, constructive pathway to identify members of this class.

On this foundation, we showed conclusively that the equimodular elliptic class is %nonempty and
physically realizable. We presented the first closed-form examples together with numerical instances discovered via a robust search-and-verify pipeline. For all of these, we established flexibility in \(\mathbb{R}^3\), provided motion parameterizations, and verified that most of them belong \emph{exclusively} to the equimodular elliptic class, even after switching the boundary strips. The examples include non-self-intersecting realizations and span both regimes \(M<1\) and \(M>1\), underscoring the breadth of attainable behavior.

Beyond these foundational contributions, we have developed and validated a complete computational pipeline that bridges the gap between theory and practice. The numerical search-and-verification algorithm, informed by our analytical results, provides an effective tool for discovering new mechanisms and can be integrated directly into CAD workflows for the inverse design and simulation of flexible structures.

The presented results resolve several open questions and simultaneously lay the groundwork for new avenues of inquiry, such as extending the framework to larger nets. In providing the first tangible examples and a clear pathway for their design and construction, this research enriches the classification of flexible polyhedra and offers a new family of mechanisms for applications in architectural geometry, robotics, and deployable systems. In future work along different trajectories, one could investigate the dynamic behavior and stability of these flexible mechanisms under external loads, drawing on techniques such as exponential integration \cite{10.1145/2508462,10.1145/2988458.2988464} for stiff problems \cite{MICHELS2015136,10.1145/3072959.3073706}. There is also potential to explore the application of graph learning \cite{NEURIPS2021_26337353}. Also, the exploration of flexible molds based on such mechanisms \cite{10.1145/3687906} could lead to novel fabrication techniques.

\section*{Acknowledgements}
This research has been supported by the baseline funding of the KAUST Computational Sciences Group. The authors are grateful to Ivan Izmestiev for useful discussions.

%The presented results resolve several open questions and simultaneously lay the groundwork for new avenues of inquiry, such as extending the framework to larger nets.

%% For citations use:
%%       \cite{<label>} ==> [1]

%%
% Example citation, See \cite{lamport94}.

%% If you have bib database file and want bibtex to generate the
%% bibitems, please use
%%
%%  \bibliographystyle{elsarticle-num}
%%  \bibliography{<your bibdatabase>}

%% else use the following coding to input the bibitems directly in the
%% TeX file.

%% Refer following link for more details about bibliography and citations.
%% https://en.wikibooks.org/wiki/LaTeX/Bibliography_Management
%\clearpage
\bibliographystyle{elsarticle-num}
% \bibliography{010_anonymized_references}
\bibliography{references}

% \clearpage
%% The Appendices part is started with the command \appendix;
%% appendix sections are then done as normal sections
% \appendix
% \section*{APPENDIX}\label{sec:appendix}
% \input{011_appendix}

\end{document}